\documentclass[a4paper,10pt]{amsart}

\usepackage{amsaddr} % Moves affiliations to the first page under the author
\usepackage{amsrefs}
\usepackage[utf8]{inputenc}
\usepackage{amscd,amssymb,amsmath,amsthm}
\usepackage[all]{xy}
\usepackage{array}
\usepackage{etoolbox}
\usepackage{mathtools}
\usepackage{tikz-cd}

\newtheorem{theorem}{Theorem}[section]

\newtheorem{lemma}[theorem]{Lemma}
\newtheorem{corollary}[theorem]{Corollary}

\newtheorem*{con7*}{Conjecture 7*}
\newtheorem*{Remark}{Remark}

\theoremstyle{definition}
\newtheorem *{Theorem A}{Theorem A}
\newtheorem *{Corollary B}{Corollary B}
\newtheorem *{Corollary D}{Corollary D}
\newtheorem *{Theorem B}{Theorem B}
\newtheorem{definition}[theorem]{Definition}
\newtheorem{example}[theorem]{Example}

\newtheorem *{Main Theorem}{Main Theorem}

\theoremstyle{remark}
\newtheorem{remark}[theorem]{Remark}

\newcommand{\Id}{\text {Id}}
\newcommand{\ob}{\text {ob}}
\newcommand{\md}{\text {Mod}}
\newcommand{\p}{\text{Pic}_k}
\newcommand{\Pic}{\text{Pic}}

\newcommand{\au}{\text{Aut}_k(K)}
\newcommand{\m}{\text{-}}

\newcommand{\bigslant}[2]{{\raisebox{.1em}{$#1$}\left/\raisebox{-.1em}{$#2$}\right.}}

\newcommand{\CPK}{{\mathcal{\tilde{C}}{(\phi,K)}}}

\newcommand{\CK}{{\mathcal{\tilde{C}}{(\phi)}}}

\newcommand{\Mod}{{\text{Mod}}}

\newcommand{\ExtP}{{\text{Ext}{(\phi,\text{Pic}(K))}}}

\newcommand{\Brk}{{\text{Br}(k)}}

\newcommand{\Br}{{\text{Br}}}
\newcommand{\Aut}{{\text{Aut}}}
\newcommand{\End}{{\text{end}}}

\newcommand{\Out}{{\text{Out}}}
\newcommand{\Hom}{{\text{Hom}}}
\newcommand{\res}{{\text{res}}}
\newcommand{\MCK}{\mathcal{MC}_K}

\newcommand{\Clf}{{\tilde{\text{Cliff}}(\phi)}}

\newcommand{\ClfK}{{\tilde{\text{Cliff}}(\phi,K)}}
\numberwithin{equation}{section}

\title[seven term exact sequence]{Actions of central Morita monoidal groupoids, and seven term exact sequences of all types}
\author{\small{Ammu Elizabeth Antony}}
\address{Department of Mathematics, University of Haifa, Haifa 31905,
Israel,\\
and  Faculty of Mathematics, Technion - I.I.T., Haifa 32000, Israel} \email{ammueliza@gmail.com}
\author{\small{Yuval Ginosar}}
\address{Department of Mathematics, University of Haifa, Haifa 31905, Israel}
\email{ginosar@math.haifa.ac.il}
\date{}

\newcommand{\ChC}[1]{{\bf \textcolor{violet}{#1}}}
\begin{document}
\begin{abstract}
We introduce a strict monoidal groupoid associated with a given commutative ring $K$.
It acts on the Clifford extension systems of $K$-central algebras which are graded by a group $G$, as well as on the generalized collective characters, both of type $\phi \in$ Hom$(G,$ Aut$(K))$.
These actions partition the two sets above with respect to graded Morita equivalence. By that we establish a seven term exact sequence of type $\phi$, which
boils down to the celebrated Chase-Harrison-Rosenberg sequence when $\phi$ is Galois.
\end{abstract}
\maketitle
%Let $k$ be a commutative ring and $G$ be a finite group.
\section{Introduction}
During the course of this paper, let $G$ be a group, $K$ a commutative ring,
and $\phi \in$ Hom$(G,$ Aut$_k(K))$ an action of $G$ on $K$ {fixing a subring $k\subset K$ elementwise}.

When $\phi$ is a Galois action, the well-known exact sequence
due to S.U. Chase, D.K. Harrison and A. Rosenberg \cite{CHR} is (the terms are explained in the sequel)
\begin{eqnarray}\label{7 term}\begin{array}{rl}
1&\rightarrow H^{1}(G,K^*_{\phi})\rightarrow \Pic(k)\rightarrow \Pic(K)^{\phi} \rightarrow H^{2}(G,K^*_{\phi})\\
&\rightarrow \text{Br}(K/k)\rightarrow H^{1}(G,\Pic(K)_{\phi})\rightarrow H^{3}(G,K^*_{\phi}).\end{array}
\end{eqnarray}
More particularly, when $k\subset K$ is a Galois extension of fields, then since the Picard groups $\Pic(K)$ and $\Pic(k)$ are trivial in this case,
it follows from \eqref{7 term} that $H^{1}(G,K^*_{\phi})=1$ (Hilbert's Theorem 90), as well as $H^{2}(G,K^*_{\phi})\cong$ Br$(K/k)$
(the Crossed Product Theorem).

%In 1968, introduced as a generalization of the concept of crossed product
The sequence \eqref{7 term} was rederived by T. Kanzaki \cite{Kanzaki} using ``generalized crossed products"
(``strongly graded algebras" herein), and their associated homomorphisms from $G$ to  Pic$(R)$ (``Generalized Collective Characters (GCC)" herein),
where $G$ is the grading group and Pic$(R)$ is the Picard group of the corresponding base algebra.
%for a commutative Galois extension $K$ of a ring $k$ with Galois group $G$.
% In 1971, F. DeMeyer and E, Ingraham also provided a proof for the exactness of the first six terms (cf: \cite{FDEI}).
Motivated by the work of Kanzaki, Y. Miyashita managed to relax the Galois
demand by introducing a seven term exact sequence \cite[Theorem 2.12]{miyashita}
for any action $\phi$. Miyashita's sequence reduces to \eqref{7 term} under certain restrictions when $\phi$ happens to be Galois.
Down the line, L. El Kaoutit and J. G{\'o}mez-Torrecillas constructed an analogue of Miyashita's seven terms exact sequence
in the context of rings with local units \cite{KG1,KG2}.  A variant of the four left terms in \eqref{7 term} is presented in \cite[Theorem 2.8]{MBDR}.
Further generalizations were given in a series of papers by M. Dokuchaev et al
\cite{DPP,DPPR,D3,D4}, for a unital ring extension $R\subseteq S$ and a unital partial representation of a group $G$ into a monoid of $R$-subbimodules of $S$.

Our link in this chain is a seven term exact sequence of
monoids, \eqref{main} herein, associated to any $G$-action $\phi \in$
Hom$(G,$ Aut$_k(K))$.
%Let $C(K)$ be any set of isomorphism types of K-central algebras which is closed under tensor product over K (see \S\ref{Sec2}).
In this way, the graded Morita class of every strongly $G$-graded $k$-algebra, whose base
algebra is $K$-central of isomorphism type that belongs to $C(K)$,
appears in such a sequence for an appropriate $\phi$.
Here, $C(K)$ is any set of isomorphism types of K-central algebras, which is closed under tensor product over $K$ and contains $[K]$ itself
(see \S\ref{Sec2}).
More formally, the sequence (see \eqref{main}) captures the graded Morita classes of all strongly $G$-graded $k$-algebras of type $\phi$, given as orbits
\begin{equation}\label{quotcliff}
\bigslant{\Clf}{{\MCK'}(1)}
%=\bigslant{(\bigsqcup_{[R]\in C(K)}\mbox{Cliff}_k(\phi,R))}{\MCK'(1)} ,
\end{equation}
under an action of a monoidal groupoid ${\MCK'(1)}$ (see \S\ref{MCKsec}),
as well as the Morita classes all general collective $G$-characters of type $\phi$, given as orbits
\begin{equation}\label{quotck}
\bigslant{\CK}{{\MCK'}(1)}
%=\bigslant{\{[\Phi]_{\eq}|\ \Phi\in\mbox{Hom}(G, \Pic_k(R))\text{ of type $\phi$},\ [R]\in C(K)\}}{\MCK(1)}
\end{equation}
under an action of the same monoidal groupoid.
%\ChA{Both \eqref{quotcliff} and \eqref{quotck} run over all isomorphism types of $K$-central algebras, which lie in $C(K)$}.
The sets \eqref{quotcliff} and \eqref{quotck} are endowed with  multiplications induced from the skew product (see \eqref{eq:1}),
and the homomorphism between them is essentially the GCC map (see \eqref{chi}). One more term in the sequence \eqref{main}, namely $\ExtP$,
extends the notion of $\Pic(k)$ in \eqref{7 term} and
requires some explanation. This is the group of isomorphism classes of modules over the skew group algebra $K_{\phi}G$, whose restrictions
to $K$ are invertible (see \eqref{extP}). We have
\begin{Theorem A}

Let $\phi \in$ Hom$(G,$ Aut$_k(K))$ be any action of a group $G$ on a commutative ring $K$ fixing a subring $k\subset K$ elementwise,
and and let $C(K)$ be any set of isomorphism types of $K$-central algebras that contains $[K]$ itself and is closed under $K$-tensoring.
Then there exists an exact sequence of monoids
\begin{eqnarray}\label{main}
\begin{array}{rl}
1&\to H^1(G,K^*_{\phi})\xrightarrow{}\ExtP\xrightarrow{}\text{Pic}(K)^{\phi}\xrightarrow{} H^2(G,K^*_{\phi})\\
&\xrightarrow{}\bigslant{\Clf}{{{\MCK'}(1)}}\xrightarrow{} \bigslant{\CK}{{{\MCK'}(1)}}
%\bigcup_{[R]\in C(K)}\mbox{Cliff}_R/_{\mbox{Pic}_R(\mbox{Cliff}_R)}\xrightarrow{\alpha_5'}\ \bigcup_{[R]\in C(K)}C_k(\phi,R)/_{\mbox{Pic}_R(C_k(\phi,R))}
 \xrightarrow{} H^3(G,K^*_{\phi})
\end{array}
\end{eqnarray}
holding, for all $K$-central algebras $R$ whose isomorphism types lie in $C(K)$
\begin{enumerate}
\item the graded Morita classes \eqref{quotcliff} of all strongly $G$-graded $k$-algebras of type $\phi$ over base algebra $R$, as well as
\item the Morita classes \eqref{quotck} of all general collective characters $G\to$Pic$(R)$ of type $\phi$.
\end{enumerate}
%\ChY{Moreover, if $K$ is an integral domain then there exists a $C(K)$ for which the rightmost homomorphism in \eqref{main} is surjective, terminating this sequence.}
\end{Theorem A}
To retrieve \eqref{7 term} from \eqref{main}, let $C(K)$ be the singleton consisting solely of $[K]$.
In this case we have an exact sequence of abelian groups (see \S \ref{Sec6})
\begin{eqnarray}\label{7K}
\begin{array}{rl}
1&\to H^1(G,K^*_{\phi})\xrightarrow{}\ExtP\xrightarrow{}\text{Pic}(K)^{\phi}\xrightarrow{} H^2(G,K^*_{\phi})\\
&\xrightarrow{}\ClfK/_{\mbox{Pic($K$)}}\xrightarrow{}H^{1}(G,\Pic(K)_{\phi})  \xrightarrow{}
H^3(G,K^*_{\phi}).
\end{array}
\end{eqnarray}
%we need an intermediate step, namely
The sequence \eqref{7K} may as well be considered a subsequence of \eqref{main}, isolating $K$ itself as $K$-central algebra in \eqref{quotcliff} and \eqref{quotck}.

Finally, under Galois setting, the sequence \eqref{7K} boils down precisely to the classic sequence \eqref{7 term}. Formally,
\begin{Theorem B}
Let $\phi \in$ Hom$(G,$ Aut$_k(K))$ be a Galois action.
Then there are group isomorphisms
\begin{enumerate}
\item $\ExtP\xrightarrow{\cong}\Pic(k)$, and
\item $\ClfK/_{\mbox{Pic($K$)}}\xrightarrow{\cong}\text{Br}(K/k)$,
%\item $\ \CPK/_{\mbox{Pic($K$)}}  \xrightarrow{\cong} H^{1}(G,\Pic(K))$,
\end{enumerate}
which commute with the corresponding homomorphisms in the exact sequences \eqref{7 term} and \eqref{7K}. Consequently, these sequences are equivalent.
\end{Theorem B}

Here is brief outline of the paper.
%In \ChY{ describe this section only after finalizing it}, we state the terminology, mostly from \cite{YG1},
% and certain equivalence relations that are often used throughout the paper. We recall the skew product, which is crucial in the construction of

After the background \S\ref{Sec2}, we record in \S\ref{mongroupsec} results about monoidal groupoids and their actions on monoids,
and in \S\ref{MCKsec} we introduce the (central) monoidal groupoid ${{\MCK'}(1)}$,
a main player here. This small monoidal category is {\em strict} (see \cite{S01}).

The monoidal groupoid ${{\MCK'}(1)}$ is shown to act on both monoids $\Clf$ (in \S\ref{Sec3}) and $ \CK$ (in \S\ref{Sec4}),
yielding monoid structures on the quotient sets $\bigslant{\Clf}{{{\MCK'}(1)}}$ and $\bigslant{\CK}{{{\MCK'}(1)}}$ respectively.
The first of these ${\mathcal{MC}_K'(1)}$-actions is an adjustment of the Morita contexts action on graded isomorphism classes introduced in \cite{JHDR} to the finer set $\Clf$.
We show that the GCC homomorphism between the monoids $\Clf$ and $ \CK$ respects the corresponding ${\mathcal{MC}_K'(1)}$-actions, giving rise to a homomorphism
between $\bigslant{\Clf}{{{\MCK'}(1)}}$ and $\bigslant{\CK}{{{\MCK'}(1)}}$.

In \S \ref{Sec5}
we interpret the first cohomology group, $H^1(G,K_{\phi}^*)$ as the kernel of a certain restriction map (Lemma \ref{ker1}), and prove Theorem A.
The exact sequence \eqref{7K} is established in \S \ref{Sec6},
and Theorem B is proven in \S\ref{Sec7}, which focuses on the Galois case.

{\bf Acknowledgments.}
The first author was funded by a Postdoctoral Fellowship from the Faculty of Mathematics at the Technion, and thanks Eli Aljadeff for his support.
We are grateful to Oren Ben-Bassat for his suggestions which helped us improve the paper significantly.

\section{Background}\label{Sec2}
All the rings in this paper are unital and associative. The
notation $K$ always stands for a commutative ring. Let
$R^*$ denote the multiplicative group of units of a ring $R$, and let
Aut$_k(R)$ denote the group of automorphisms of $R$ which fix a
central subring $k\subseteq Z(R)$ pointwise. For a group action
$\phi \in$ Hom$(G, \Aut_k(K))$, let $K_{\phi}$ and $K_{\phi}^*$ be
the left $G$-modules on the underlying abelian groups $K$ (with
addition) and its units $K^*$ (with multiplication) respectively.
We denote
%and where its left $G$-action is given by
$$g(y) :=\phi(g)(y)$$
for every $g \in G$ and $y$ either in $K_{\phi}$ or, in particular, in $K^*_{\phi}$.

Denote the $K$-isomorphism class of a $K$-algebra $R$ by $[R]$. Such an algebra $R$ is said to be $K$-central if its center is isomorphic to $K$.
Let $\md(R)$ and $\md(R\m R')$ denote the isomorphism classes of (left) $R$-modules and of $R\m R'$-bimodules respectively.
Left and right actions on bimodules are denoted by $``\star"$ and $``\cdot"$ respectively. If  $[M]\in \md(R)$ and $K$ is central in $R$,
then the left $R$-module $M$ can also be considered as a right $K$-module via its left $K$-structure.

For $[M]\in\md(R\m R')$, the set $\check{M}:=$ Hom$_{R'}(M,R')$ of all right $R'$-module homomorphisms from $M$ to $R'$, admits an $R'$-$R$-bimodule structure given by
\begin{equation}
\label{Hom(P,R)}
(r' \star f) (m) := r'\star f(m)\ \mbox{and}\ (f\cdot r)(m) := f(r\star m),
\end{equation}
for every $r \in R, r' \in R', f \in \check{M}$, and $m \in M$.

The relative Brauer group $\Br(K/k)$ consists of all Brauer similarity classes $[A]_{\Brk}$ of those $k$-Azumaya algebras, i.e. central separable over $k$, which split over $K$.
In other words, $A\otimes_kK$ is isomorphic to a ring of matrices over $K$.

An $R\m S$-bimodule $P$ is {\em invertible} if there is an $S\m R$-bimodule $P'$ such that
$P\otimes_S P'\cong R$ and $P'\otimes_R P\cong S$, as a free $R\m R$-bimodule of rank 1 and $S\m S$-bimodule of rank 1 respectively.
The Picard group Pic$_k(R)$ consists of the isomorphism classes of invertible $R$-bimodules, whose left and right $k$-module structures are the same.
The product in Pic$_k(R)$ is given by $[P_1]\otimes[P_2]:=[P_1\otimes_R P_2]$, and the identity is $[R]$.
In the special case $R=k$, the abelian group Pic$_k(k)$ is conventionally written just as Pic$(k)$.

The following subsections build the foundation for the main results in this paper.
\subsection{}\label{secCK}
Let $C_{\kappa}(K)$ be the set of all $K$-isomorphism classes of $K$-central algebras whose cardinality is bounded by some infinite cardinality $\kappa$.
Next, consider a subset $C(K) \subset C_{\kappa}(K)$, which contains the class $[K]$ and is closed under the tensor product,
forming an abelian monoid under the product $[R_1] \otimes_K [R_2] := [R_1 \otimes_K R_2]$.
%Next, let $C(K) \subset C_{\kappa}(K)$ be a monoid under the product $[R_1] \otimes_K [R_2] := [R_1 \otimes_K R_2]$, i.e., $C(K)$ is an abelian monoid that is closed under the tensor product,
%and whose two-sided identity is the class $[K]$.
The reader may bear in mind one of the following examples that fit the role of $C(K)$.
\begin{itemize}
\item[(i)] %The isomorphism classes of all unital algebras $R$ such that $K$ is isomorphic to the center of the ring $R$, which is free as a $K$-module, i.e.,
$C(K) := \{[R]\in C_{\kappa}(K)\mid R$ is free as a $K$-module$\}$.
%$R \cong \oplus_{I} K$ as $K$-modules for some $I\}$.
\item[(ii)]
%$C(K) := \{[R]\in C_{\kappa}(K)\mid R$ is $K$-separable$\}$.
$C(K) := \{[R]\mid R$ is $K$-Azumaya$\}$.
\end{itemize}
In \S\ref{Sec6} we focus on the extremal case, that is where $C(K)$ is the singleton

(iii) $C(K) := \{[K]\}$.

\subsection{}
% let $\p(R)<$Pic$(R)$ be the subgroup of classes of bimodules that admit the same left and right $k$-module structures.
Given $\eta\in\text{Aut}_k(R)$, endow the free left $R$-module of rank 1 with an additional right $R$-module structure via $\eta$, see \eqref{star}.
The formal structure of this $R$-bimodule, namely $R_{1,\eta}$, is given by
\begin{equation}\label{star}r\star s:=r s,\ \ s\cdot r:=s\eta(r),\ \ r\in R,\ \  s\in R_{1,\eta}.\end{equation}

Then there is a well-defined homomorphism %\begin{eqnarray}
$\begin{array}{ccc}\text{Aut}_k(R)&\rightarrow&\p(R)\\
\eta&\mapsto&[R_{1,\eta}]
\end{array},$\\
%\end{eqnarray}
~whose kernel is the inner automorphisms of $R$ (see \cite[Prop. II.5.3]{Bass}),
yielding an embedding
\begin{equation}\label{out}
\text{Out}_k(R):=\text{Aut}_k(R)/\text{Inn}(R)\hookrightarrow\p(R)
\end{equation}
of the outer $R$-automorphisms as a subgroup of the Picard group of $R$.

An action $\phi:G\to\Aut_k(K)$ of a group $G$ on a $K$-central algebra $R$ fixing $k\subseteq K$ elementwise induces a well-defined $G$-action on $\md (R)$ using the above notation.
\begin{equation}\label{Gactmd}
g\times [M]\mapsto[\phi(g)(M)],\text{ where }\phi(g)(M) := \ K_{1,\phi(g)} \otimes_K M \otimes_K \ K_{1,\phi(g^{-1})},\end{equation}
for every $g\in G$ and $[M]\in\md(R).$
Denote the subset of invariant classes
\begin{equation}\label{PKinv}\md(K)^{\phi} := \{[M] \in \md(K) \mid [\phi(g)(M)] = [M],\ \forall g\in G\}.\end{equation}
As can easily be checked, the action \eqref{Gactmd} stabilizes $\Pic(K)< \md(K).$
The intersection $$\Pic(K)^{\phi}:=\Pic(K)\cap\md(K)^{\phi},$$
is evidently a subgroup of Pic$(K)$.

\subsection{}
\label{alpha3sect}

Following \cite[Chapter IV]{FDEI}, here is a description of the third map in (\ref{7 term}), that is $ \Pic(K)^{\phi}\to H^2(G,K^*_{\phi})$.
The $G$-invariance of a class $[P] \in \Pic(K)^{\phi}$ (see \eqref{Gactmd} and \eqref{PKinv}) is realized by the set of $K$-module isomorphisms
\begin{equation}\label{BP}B_P=\{\psi_g:g(P) \rightarrow \ P\}_{g\in G}.
\end{equation}
For every $g\in G$ let
\begin{eqnarray}\label{psig}\tilde{\psi}_g :\begin{array}{ccc} P &\rightarrow& P\\
p&\mapsto& \psi_g(1 \otimes_K p \otimes_K 1)
\end{array}.\end{eqnarray}
Then a direct computation yields that the maps \eqref{psig} are invertible and satisfy
\begin{equation}\label{tildepsi}
g(x)\star\tilde{\psi}_g(p) = \tilde{\psi}_g(x\star p),\text{ and so } \ g^{-1}(x)\star\tilde{\psi}^{-1}_g(p) = \tilde{\psi}^{-1}_g(x\star p)\end{equation}
for every $g\in G, x\in K$ and $p\in P$.
Consequently, for every $g,h \in G, p\in P$ and $x\in K$,
we have $$\tilde{\psi}_{gh}\circ\tilde{\psi}_h^{-1}\circ\tilde{\psi}_g^{-1}(x\star p)
%= \psi_{gh}((^{h^{-1}g^{-1}}x)\psi_h^{-1}\psi_g^{-1}(p))
= x\star(\tilde{\psi}_{gh}\circ\tilde{\psi}_h^{-1}\circ\tilde{\psi}_g^{-1}(p)).$$
In other words,
\begin{equation}\label{fP}\tilde{\psi}_{gh}\circ\tilde{\psi}_h^{-1}\circ\tilde{\psi}_g^{-1} \in \End_K(P).\end{equation}
Since $\End_K(P) \cong K$ \cite[ Lemma  5.1]{FDEI}, then the $K$-isomorphism \eqref{fP} is just multiplication by an element ${{B^*_P}(g,h)}\in K^*_{\phi}$. Then
\begin{lemma} \cite[IV; 2.6, 2.7, 2.8]{FDEI}
With the above notation
\begin{enumerate}
\item The map $B^*_P:\begin{array}{ccc}
G\times G&\to & K^*\\
(g,h)&\mapsto & B^*_P(g,h)
\end{array}$ is a 2-cocycle.
\item If $\tilde{B}_P$ is another set of $K$-module isomorphisms
then $$[B^*_P]=[{\tilde{B}^*_P}]\in H^2(G,K^*_{\phi}).$$
%The cohomology class of this 2-cocycle is independent of the choice of the set \eqref{BP} of $K$-module isomorphisms.
\item The well-defined map
\begin{eqnarray}\label{alpha3}\alpha_3 :\begin{array}{ccc}
\Pic(K)^{\phi} &\rightarrow & H^2(G,K^*_{\phi})\\
~[P]&\mapsto & [B^*_P]
\end{array}\end{eqnarray} is a group homomorphism.
\end{enumerate}
\end{lemma}

\subsection{} We record a standard result (see \cite[\S{2}]{Kanzaki} for a special case).
Let $A\rtimes \Gamma$ be a semidirect product of an abelian group $A$ acted on by a group $\Gamma$ and let $\phi_0 : G\to\Gamma$ be a group homomorphism.
Then $\phi_0$ furnishes $A$ with a $G$-module structure $A_{\phi_0}$ via $$g(a):=\phi_0(g)(a),\ \  g\in G, a\in A_{\phi_0}.$$
Note that every group homomorphism
$\varphi\in\Hom(G,A\rtimes \Gamma)$ gives rise to a homomorphism $\varphi^*\in\Hom(G,\Gamma)$,
which is a composition of $\varphi$ with the projection of $A\rtimes \Gamma$ onto $\Gamma$. Let
\begin{equation}\label{homphi}
\Hom_{\phi_0}(G,A\rtimes \Gamma):=\{\varphi\in\Hom(G,A\rtimes \Gamma)|\ \ \varphi^*=\phi_0\}.
\end{equation}
Then $\Hom_{\phi_0}(G,A\rtimes \Gamma)$ is an abelian group with the product
\begin{equation}\label{homphiprod}
\varphi_1\star\varphi_2:g\mapsto \varphi_1(g)\cdot\phi_0(g)^{-1}\cdot\varphi_2(g), \ \ \ \varphi_1,\varphi_2\in \Hom_{\phi_0}(G,A\rtimes \Gamma),
\ \ g\in G.
\end{equation}
Furthermore, it is not hard to check that for every $\varphi\in$ \eqref{homphiprod}

$\varphi\phi_0^{-1}:\begin{array}{ccc}
G&\to& A_{\phi_0}\\
g&\mapsto&\varphi(g)\cdot\phi_0(g)^{-1}\end{array}$ is a 1-cocycle, and that
\begin{lemma}\label{Hphiz}
With the above notation there is a group isomorphism
\begin{eqnarray}\label{lambg}
\lambda_{\phi_0}:\begin{array}{ccc}
\Hom_{\phi_0}(G,A\rtimes \Gamma)&\xrightarrow{\cong} &Z^1(G,A_{\phi_0})\\
\varphi&\mapsto&\varphi\phi_0^{-1}
\end{array}.\end{eqnarray}\end{lemma}

\subsection{}
Here is another standard result.
\begin{lemma}\label{exactgrp}
Let $M_1\xrightarrow{f_1}M_2\xrightarrow{f_2}M_3\xrightarrow{f_3}M_4$ be an exact sequence of monoids. Suppose that $M_1$ and $M_3$ are groups, then so is $M_2$.
\end{lemma}
\begin{proof}
We need to find an inverse for every $m_2\in M_2$. Indeed, the image $m_3:=f_2(m_2)$ of such given $m_2$ is invertible in $M_3$.
Applying $f_3$ we obtain $$e_4=f_3(m_3\cdot m_3^{-1})=f_3( f_2(m_2))\cdot f_3(m_3^{-1})=f_3(m_3^{-1}).$$Therefore, $m_3^{-1}\in\ker(f_3)$ and hence there is $m'_2\in M_2$ such that $f_2(m'_2)=m_3^{-1}$.
It follows that $f_2(m_2\cdot m'_2)=m_3\cdot m_3^{-1}=e_3$, and so there is $m_1\in M_1$ such that $f_1(m_1)=m_2\cdot m'_2$.
Then $$m_2\cdot m'_2\cdot f_1(m_1^{-1})=f_1(m_1)\cdot f_1(m_1^{-1})=e_2,$$
and hence $m'_2\cdot f_1(m_1^{-1})$ is a right inverse of $m_2$. A left inverse is obtained similarly.
\end{proof}

\subsection{}\label{GalBack}
In this subsection we state a property of certain orthogonal idempotents in the Galois theory of commutative rings.
It plays a pivotal role in \S \ref{Sec7}. We remind
\begin{definition}\label{Galdef}
A commutative ring extension $k\subseteq K$ is a Galois extension with Galois group $G=$Aut$_k(K)$ if
\begin{enumerate}
\item $K^G=k$, and
\item for each maximal ideal $M\subseteq K$
and for every nontrivial $g\in G$ there exists $x\in K$ with $x-g(x)\notin M$.
\end{enumerate}
\end{definition}
Denote the set of functions from a group $G$ to a commutative ring $K$ by $\nabla (K:G)$,
and furnish it with pointwise additive and multiplicative rules.
Then the commutative $K$-algebra $\nabla (K:G)$ is isomorphic to the direct product of $|G|$ copies of $K$ generated by orthogonal idempotents $\{\tilde{\epsilon}_g\}_{g\in G}$.
One of the equivalent criteria \cite[Proposition III.1.2]{FDEI} of an action $\phi  \in$ Hom$(G, \Aut_k(K))$ with $k=K^G$ to be Galois, is that
%A standard result \cite[Proposition III.1.2 ]{FDEI} says that when $\phi  \in$ Hom$(G, \Aut_k(K))$ is a Galois action then
the $K$-algebra homomorphism
\begin{eqnarray}\label{l}\begin{array}{ccc}
K \otimes_k K &\to &\nabla (K:G)\\
a \otimes_k b &\mapsto & \sum\limits_{g \in G}a\cdot{\phi(g^{-1})}(b)\tilde{\epsilon}_g
\end{array},\ \ \ \ a,b\in K\end{eqnarray}
is an isomorphism.
Suppose then that $K$ is a Galois extension of $k$. For every $g\in G$ let $\epsilon_g\in  K \otimes_k K$ be the inverse image of $\tilde{\epsilon}_g\in \nabla (K:G)$ under the isomorphism \eqref{l}.
Certainly $\{\epsilon_g\}_{g\in G}$ are orthogonal idempotents in $K \otimes_k K$, which sum up to $1 \otimes_k 1$.
The following properties of $\{\epsilon_g\}_{g\in G}$ are important in \S \ref{Sec7}.
\begin{lemma}\label{Kidem}
Let $a\in K$ such that $(a\otimes_k1){\epsilon}_g={\epsilon}_g$. Then $a=1$.
\end{lemma}
\begin{proof}
Apply the isomorphism \eqref{l} to get $a\tilde{\epsilon}_g=\tilde{\epsilon}_g$, and hence the claim.
\end{proof}
The second claim is less trivial.
\begin{lemma}\label{L:2.2}
Let $\phi  \in$ Hom$(G, \Aut_k(K))$ be a Galois action. Then the orthogonal idempotents $\{\epsilon_g\}_{g\in G}\subset K \otimes_k K$ satisfy
\begin{equation}\label{theorth}
\epsilon_g(x \otimes_k 1) = \epsilon_g(1 \otimes_k \phi(g)(x)),\ \ \forall g \in G, x \in K.
\end{equation}
In particular,
\begin{equation}\label{eq:OI}
\epsilon_g(K\otimes_k K) = \epsilon_{g}(K \otimes_k 1) = \epsilon_{g}(1 \otimes_k K),\ \ \forall g \in G.
\end{equation}
Moreover, if an idempotent $\epsilon_g' \in K \otimes_k K$ satisfies
$\epsilon_g'(x \otimes_k 1) = \epsilon_g'(1 \otimes_k \phi(g)(x))$ for some $g \in G$ and for every $x\in K$,
then there exists an idempotent $a_g\in K$ such that $\epsilon_g' =\epsilon_g(a_g \otimes_k 1)$.
\end{lemma}
\begin{proof}
Apply the $K$-algebra homomorphism \eqref{l} for $g \in G$ and $x\in K$
$$\epsilon_g(x \otimes_k 1)-\epsilon_g(1 \otimes_k\ \phi(g^{})(x)) \mapsto
%\tilde{\epsilon}_{g^{-1}} \sum\limits_{\sigma \in G}x\tilde{\epsilon}_{\sigma} -
%\tilde{\epsilon}_{g^{-1}} \sum\limits_{\tau \in G}\ (^{\tau g}x)\tilde{\epsilon}_{\tau} =
\tilde{\epsilon}_{g^{}}x-\tilde{\epsilon}_{g^{}}x =0.$$
As mentioned above, in the Galois case \eqref{l} is an isomorphism, verifying \eqref{theorth} and consequently \eqref{eq:OI}.
%Thus, $e_g(x \otimes 1) = e_g(1\otimes\ ^gx)$. Clearly, \eqref{eq:OI} holds true for this set of idempotents.
For the last claim of the lemma, let $\epsilon_g' \in K \otimes_k K$ be an idempotent for some $g\in G$ such that
\begin{equation}\label{eps'}
\epsilon_g'(x \otimes_k 1) = \epsilon_g'(1 \otimes_k \phi(g^{})(x)),\ \forall x \in K.
\end{equation}
%We claim that $\epsilon_g' = \epsilon_g$.
By \eqref{eq:OI} there exist $a_{h} \in K$ for every $h\in G$ such that
\begin{equation}\label{epsum}
\epsilon_g' = \sum\limits_{h \in G}\epsilon_{h}(a_{h} \otimes_k 1).
\end{equation}
Thus, by the assumption \eqref{eps'} %and
we obtain
$$\sum\limits_{h \in G}\epsilon_{h}(a_{h} \otimes_k 1)((x \otimes_k 1)-(1 \otimes_k \phi(g^{})(x))) =\epsilon_g'((x \otimes_k 1)-(1 \otimes_k \phi(g^{})(x)))= 0,\ \forall x \in K.$$
Applying the isomorphism \eqref{l}, with the orthogonality of $\{\tilde{\epsilon}_g\}_{g\in G}$ in mind, yields
$$\sum\limits_{h\in G}\tilde{\epsilon}_{h^{}}a_{h}(x-\phi(h^{-1}g)(x)) = 0,\ \forall x \in K.$$
Consequently,
\begin{equation}\label{eq2}
a_{h}(x-\phi(h^{-1}g)(x))= 0,\ \forall h \in G, x\in K.\end{equation}

By the Galois assumption (see Definition \ref{Galdef}) it follows
that $a_{h} = 0$ for every $h \neq g$, otherwise there exists
$y\in K$ such that $y -\phi(h^{-1}g)(y) $ does not belong to some
maximal ideal containing Ann$_K(a_{h})$ and, particularly, does
not annihilate $a_h$ contradicting \eqref{eq2} for $h$.
Therefore, % $a_{h} = 0$ for every $h \neq \sigma$ and thus
\eqref{epsum} is just $\epsilon_g' = \epsilon_g(a_g \otimes_k 1)$.
Applying the isomorphism \eqref{l} once again on $\epsilon_g'^2=\epsilon_g'$ shows $a_g$ is an idempotent in $K$, and hence the claim.
\end{proof}
%\end{equation}

\section{actions of monoidal groupoids}\label{mongroupsec}
Recall that a groupoid is a category, all of whose morphisms are invertible.
In order to define a (left) action $(f,a)\mapsto f(a)$ of a groupoid $\mathcal{G}$ on a set $A$, one firstly needs an anchor (or a moment) map $\varsigma:A\to \text{ob}(\mathcal{G})$.
The action is defined on pairs $(f,a)\in \hom(\mathcal{G})\times A$, such that the source of $f$ is equal to $\varsigma(a)$. We call these pairs \textit{admissible}.
Then a (fibred) action should satisfy for
$f,g\in \text{hom}(\mathcal{G})$ and $a\in A$ (see \cite[\S 4]{Gilbert05})
\begin{itemize}\label{actgrp}
\item $\varsigma(f(a))$ is equal to the range of $f$,
\item $f(g(a))=f\circ g(a)$, whenever they are defined, and
\item $a=u_{\mathcal{G}}(\varsigma(a))(a)$, where
$u_{\mathcal{G}}:\begin{array}{ccc}\text{ob}(\mathcal{G})&\hookrightarrow &\text{hom}(\mathcal{G})\\
X&\mapsto& \text{Id}_X
\end{array}$ is the natural embedding
of the objects of $\mathcal{G}$ inside its morphisms.
\end{itemize}

The orbit of an element $a\in A$ under such action is naturally defined as
$$[a]_{\mathcal{G}}:=\{f(a)\mid (f,a)\text{ is admissible}\}\subseteq A.$$
The invertibility property of a groupoid yields that its action on a set induces an equivalence relation on this set determined by the orbits of the action.

To an action of a groupoid $\mathcal{G}$ on a set $A$ one can associate an \textit{action groupoid} $A^\mathcal{G}$, whose objects are just the elements of $A$, and
$$\hom(a,b):=\{f\in\hom(\mathcal{G})\mid f(a)=b\},\ \ a,b\in A.$$

A groupoid $\mathcal{G}$ is monoidal if it admits a bifunctor $\otimes:\mathcal{G}\times\mathcal{G}\to\mathcal{G}$ as well as a unit
$1\in\ob(\mathcal{G})$, both subject to certain coherence conditions (see \cite{CCH}).
In particular,
\begin{itemize}
\item For every two morphisms $f \in \hom(a,b)$, and $f' \in \hom(a',b')$ we have
\begin{equation}\label{tensoring}
f \otimes f' \in \hom(a\otimes a',b \otimes b').
\end{equation}
\item The tensor product is compatible with the composition via
\begin{equation}\label{MonGroup}
(f\otimes f'){\circ}( g\otimes g')= (f\circ g)\otimes(f'\circ g'),
\end{equation}
for every composable pairs $(f,g),(f',g')\in\text{hom}(\mathcal{G})\times\text{hom}(\mathcal{G})$.
\end{itemize}
Consequently,
\begin{lemma}\label{End_1}
Let $\mathcal{G}$ be a strict monoidal groupoid, then
\begin{enumerate}\item $\Id_{1_{\ob(\mathcal{G})}}$ is the unit of the monoid $\hom(\mathcal{G}).$
\item $\End(1_{\ob(\mathcal{G})})$ is closed under tensoring.
%That is, the full subcategory on $1_{\ob(\mathcal{G})}$ is itself a monoidal groupoid.
Moreover, the monoidal structure of $\End(1_{\ob(\mathcal{G})})$ identifies with its group structure by compositions. In particular, $(\End(1_{\ob(\mathcal{G})}),\otimes)$ is a group.
\end{enumerate}
\end{lemma}
\begin{proof}
\begin{enumerate}
\item This is immediate from the strictness assumption.
 \item The closure under tensoring follows from \eqref{tensoring}. The identification of the two structures is a consequence of \eqref{MonGroup}, Indeed, for every $f,g\in\End(1_{\ob(\mathcal{G})})$
$$f\circ g=(f\otimes 1_{\hom}){\circ}( 1_{\hom}\otimes g)= (f\circ 1_{\hom})\otimes(1_{\hom}\circ g)=f\otimes g.$$
 \end{enumerate}
\end{proof}

In \S\ref{MCKsec} we introduce a strict monoidal groupoid, namely $\mathcal{MC}'_K(1)$, which is key in this work. Here is a notion for an action of strict monoidal groupoids on monoids.
\begin{definition}\label{Defact}
Let a groupoid $\mathcal{G}$ act on a set $A$ as above.
Suppose that both $\mathcal{G}$ and $A$ admit monoidal products $\otimes$ and $\otimes_A$ respectively, where the monoidal structure of $\mathcal{G}$ is strict.
We say that $\mathcal{G}$ acts on $A$ as a monoidal groupoid if  the anchor $\varsigma:A\to{\ob(\mathcal{G})}$ is a monoid homomorphism, and the rules\\
$\begin{array}{ccc}
\ob(A^\mathcal{G})\times\ob(A^\mathcal{G})&\to&\ob(A^\mathcal{G})\\
(a,b)&\mapsto& a\otimes_A b
\end{array}$\ and\
$\begin{array}{ccc}
\hom(A^\mathcal{G})\times\hom(A^\mathcal{G})&\to&\hom(A^\mathcal{G})\\
(f,g)&\mapsto& f\otimes g
\end{array}$\\
determine a bifunctor $A^\mathcal{G}\times A^\mathcal{G}\to A^\mathcal{G}$, enriching the action groupoid $A^\mathcal{G}$ with a monoidal structure.
\end{definition}
In particular, the source and the target are homomorphisms of monoids from $\hom(A^{\mathcal{G}})$ to $A$, that is
%Comparing targets of morphism products we obtain the following equality in $A=\ob(A^\mathcal{G})$
\begin{equation}\label{secondition}
(f\otimes g)(a\otimes_A b)=f(a)\otimes_A g(b),
\end{equation}
for any two admissible pairs $(f,a)$ and $(g,b)$.
Throughout the paper, by an action of a monoidal groupoid $\mathcal{G}$ on a monoid $A$ we mean that equation \eqref{secondition} is satisfied without taking care of the other coherence conditions.
\begin{example}\label{conjex}
A groupoid $\mathcal{G}$ acts on the endomorphisms submonoid
$$A:=\bigsqcup_{a\in\ob(\mathcal{G})}\text{end}(a)(=\bigsqcup_{a\in\ob(\mathcal{G})}\text{aut}(a))\subset\hom(\mathcal{G})$$
by conjugation. The anchor $\varsigma(f)$ of a morphism in $f\in\text{end}(a)\subset A$ is $a\in\ob(\mathcal{G})$, and then
$g(f):= g\circ f\circ g^{-1}\in\text{end}(a')$ for every $a\in\ob(\mathcal{G})$ and $g\in$hom$(a,a')$. Assume, additionally, that the groupoid $\mathcal{G}$ is strictly monoidal.
Then being closed under the tensoring, $A$ is a monoid. Furthermore, condition \eqref{MonGroup}, which is satisfied by $\mathcal{G}$, implies \eqref{secondition} and hence
this is an action of a monidal groupoid on a monoid.
\end{example}
\begin{lemma}\label{orblemma}
Let a strict monoidal groupoid $\mathcal{G}$ with identity $1_{\text{ob}(\mathcal{G})}$ act on a monoid $(A,\otimes_A)$ with identity $1_A$. Then
\begin{enumerate}
\item The orbit $[1_A]_{\mathcal{G}}$ is a submonoid of $A$.
%the sub-orbit $$\{f(1_A)\mid f\in\text{end}(1_{\text{ob}(\mathcal{G})})\}< [1_A]_{\mathcal{G}}$$ is a subgroup which is an image of $\text{end}(1_{\text{ob}(\mathcal{G})})$.
\item The rule $f\mapsto f(1_A)$ defines a group homomorphism from $\text{end}(1_{\text{ob}(\mathcal{G})})$ onto the subgroup
$\{f(1_A)\mid f\in\text{end}(1_{\text{ob}(\mathcal{G})})\}$ of the monoid $[1_A]_{\mathcal{G}}$.
\item The orbits under this action are endowed with a well-defined monoidal structure
$$[a]_{\mathcal{G}}\otimes_{[A]}[a']_{\mathcal{G}}:=[a\otimes_A a']_{\mathcal{G}},\ \ \forall a,a'\in A.$$
\end{enumerate}
\end{lemma}
\begin{proof}
All claims follow from \eqref{secondition}.
\end{proof}
Denote the quotient monoid in Lemma \ref{orblemma}(3) by $A/\mathcal{G}$. The following lemma is straightforward.
\begin{lemma}\label{quotmor}
Let $\mathcal{G},\mathcal{G}'$ be strict monoidal groupoids acting on monoids $A,A'$ respectively.
Then a monoid homomorphism $\varphi:A\to A'$ gives rise to a well-defined monoid homorphism
$\bar{\varphi}:\begin{array}{ccc}
A/\mathcal{G}&\to& A'/\mathcal{G}'\\
~[a]_{\mathcal{G}}&\mapsto&[\varphi(a)]_{\mathcal{G}'}
\end{array}$
% ,\ a\in A $
if and only if for every admissible pair $(f,a)\in \hom(\mathcal{G})\times A$,
%$f\in$hom$(\mathcal{G})$ and $a\in A$ such that $f(a)$ is defined,
there exists $f'\in$hom$(\mathcal{G}')$ with $f'(\varphi(a))=\varphi(f(a))$.
In particular, the quotient map $\begin{array}{ccc}
A&\to& A/\mathcal{G}\\
~a&\mapsto&[a]_{\mathcal{G}}
\end{array}$ is a homomorphism of monoids.
\end{lemma}
\begin{definition}\label{grpmap}
Let a groupoid $\mathcal{G}$ act on two sets $A,B$.
A map $\varphi:A\to B$ is a \textit{$\mathcal{G}$-map} if when $(f,a)\in \hom(\mathcal{G})\times A$ is admissible,
% for $f\in$hom$(\mathcal{G}), a\in A$, whenever $f(a)$ is defined then
then so is $(f,\varphi(a))\in \hom(\mathcal{G})\times B$, and
$f(\varphi(a))=\varphi(f(a)).$
\end{definition}
By Lemma \ref{quotmor}, if $\varphi:A\to B$ is a $\mathcal{G}$-homomorphism of $\mathcal{G}$-monoids, where $\mathcal{G}$ is a strict monoidal groupoid, then
$\bar{\varphi}:A/\mathcal{G}\to B/\mathcal{G}$ is a homomorphism of monoids.

%Suppose now that the monoidal groupoid ${\mathcal{G}}$ has a single object).Evidently, in this case,.

Suppose now that $\mathcal{G}$ has a single object 1. Then by Lemma \ref{End_1}(2) $\hom(\mathcal{G})=\text{end}(1)$ is a group also under tensoring.
%Assume further that the monoidal structure of $\mathcal{G}$ is also a group.
Let $\mathcal{G}$ act on a monoid $(A,\otimes_A)$ as in Definition \ref{Defact}. In this case, the action is global, since the anchor of every $a\in A$ is 1.
%Then the orbit $[a]_{\mathcal{G}}=[a]_{}$.
Due to Lemma \ref{orblemma}(1), $A$ admits $[1_A]_{\mathcal{G}}$ as a submonoid. Apply \eqref{secondition} to obtain that $[1_A]_{\mathcal{G}}$ is a subgroup.
Moreover, \eqref{secondition} yields
\begin{eqnarray}\label{1A}\begin{array}{rl}
f(a)&=(\Id_{1}\otimes f)(a\otimes_A 1_A)\\
&=\Id_{1}(a)\otimes_A f(1_A)=a\otimes_A f(1_A)\end{array},
 \ \forall f\in\text{end}(1), a\in A,
\end{eqnarray}
and similarly $f(a) = f(1_A) \otimes_A a$. Therefore $[1_A]_{\mathcal{G}}$ is a central subgroup of the monoid $A$.
%Hence, the equivalence relation given below is a congruence; $$a_1 \sim a_2\ \mbox{if\ and\ only\ if}\ a_1 = a_2 \otimes_A f(1), \mbox{for\ some}\ f \in \hom(\mathcal{G}).$$
It is easily verified that the quotient monoid
$A/[1_A]_{\mathcal{G}}$ is well-defined.
The following consequence of \eqref{1A} is used in the sequel.
\begin{lemma}\label{isoqot}
With the above assumptions on $\mathcal{G}$, the rule $[a]_{\mathcal{G}}\mapsto a\otimes_A[1_A]_{\mathcal{G}}$ yields an isomorphism $A/{\mathcal{G}}\cong A/[1_A]_{\mathcal{G}}$ of monoids.
\end{lemma}

\section{
Monoidal groupoids over central $K$-algebras}\label{MCKsec}
Fixing $C(K)$ as in \S\ref{secCK}, let $\mathcal{MC}_K$ be a category, whose objects are the $K$-algebras that belong to an isomorphism type in $C(K)$, in particular these algebras are $K$-central.
For $R,S\in\ob(\mathcal{MC}_K)$, the morphisms hom$(S, R)=\Mod^*(R\m S)$, i.e. the isomorphism classes of invertible $R\m S$-bimodules (see \S\ref{Sec2}). The classes
\begin{equation}\label{twobmods0}[M]\in \Mod^*(R\m S),\ \ [M']\in \Mod^*(R'\m S'),\ \  R,R',S,S'\in\ob(\mathcal{MC}_K)
\end{equation} are composable if $S=R'$ and then
\begin{equation}\label{compmods}[M]\circ[M']=[M\otimes_SM']\in \Mod^*(R\m S').\end{equation}
By the invertibility property, $\mathcal{MC}_K$ is a groupoid.

Notice that the isomorphism classes $[_RP_R]\in\text{end}_{\mathcal{MC}_K}$ of the invertible $R$-bimodules are exactly the elements of the Picard group $\Pic (R)$ (see \S\ref{Picsec})
for every $K$-algebra $R$, yielding a group isomorphism
\begin{equation}\label{Picsub}
\Pic(R)\cong\text{end}_{\mathcal{MC}_K}(R),\ \  R\in\ob(\mathcal{MC}_K).
\end{equation}
We therefore identify Pic$(R)$ with the full subcategory of $\mathcal{MC}_K$ on a single object $R$.
Let us introduce a monidal groupoid $\mathcal{MC}_K(1)$, which is a a wide subcategory of $\mathcal{MC}_K$.

Consider the \textit{Morita monoidal groupoid} $\mathcal{MA}lg_K$ \cite[\S 2.1]{CCH}, whose objects are $K$-algebras.
The morphisms hom$(S, R)$ are isomorphism classes of a Morita equivalences $\text{Mod}(S)\xrightarrow{\sim} \text{Mod}(R)$, or in other words,
isomorphism classes of invertible $R\m S$-bimodules that admit the same left and right actions of $K$, i.e.
$$\hom(S,R):=\{[_RP_S]\mid x\star p=p\cdot x,\ \ \forall x\in K,p\in P\}\ \ \forall R,S\in \ob(\mathcal{MA}lg_K).$$
The tensor product of two bimodule classes
\begin{equation}\label{twobmods}[M]\in \hom(S,R),\ \ [M']\in \hom(S',R'),\ \ R,R',S,S'\in \ob(\mathcal{MA}lg_K)
\end{equation}
is
$$[M]\otimes[M']=[M\otimes_KM']\in \hom((S\otimes_KS'),(R\otimes_KR')).$$
Morphisms \eqref{twobmods} are composed as in \eqref{compmods}.
%The classes \eqref{twobmods} are composable if $S=R'$ and then \begin{equation}\label{compmods}[M]\circ[M']=[M\otimes_SM']\in \Mod^*(R\m S').\end{equation}

%Fixing $C(K)$ as in \S\ref{secCK}, let $\mathcal{MC}_K(1)$ be the full subcategory of $\mathcal{MA}lg_K$, whose objects are the $K$-algebras that belong to an isomorphism type in $C(K)$,
%in particular these algebras are $K$-central.
Let $\mathcal{MC}_K(1)$ be the full subcategory of $\mathcal{MA}lg_K$ on $\ob(\mathcal{MC}_K)$.
By the demand that $C(K)$ is an abelian monoid under $K$-tensoring of isomorphism classes, we have that $\mathcal{MC}_K(1)$ is itself a monoidal groupoid.
In the special case where $C(K) := \{[R]\mid R$ is $K$-Azumaya$\}$ (\S\ref{secCK}(ii)), we have that $\mathcal{MC}_K(1)$ is the \textit{Brauer categorical group} $\mathcal{B}r_K$ of $K$ \cite[\S 2.1]{CCH}.

We ``strictify" the monoidal groupoid $\mathcal{MC}_K(1)$ as follows.
Choose any set $C'$ of representatives of the isomorphism classes in $C(K)\setminus [K]$.
Let $\mathcal{MC'}_K(1)$ be a small category whose objects are associative finite words, including the empty word $\epsilon$, on the (perhaps infinite) alphabet $C'$.
For a word $R\in\ob(\mathcal{MC'}_K(1))$,
let \(\mathcal{P}({R})\) denote the ``underlying" single algebra as follows. Start with $\mathcal{P}(\epsilon) = K$ and $\mathcal{P}(R) = R$ for the words of length 1 (i.e. for $R\in C'$).
Proceed inductively by multiplying out the sequence from left to right.
For example, if $R_1, R_2, R_3\in C'$ then
$$\mathcal{P}(R_1 R_2 R_3) = (R_1 \otimes_K R_2) \otimes_K R_3.$$
A morphism in the category $\mathcal{MC'}_K(1)$
%%\(f: \mathbf{A} \to \mathbf{Y}\)
is defined as a morphism in $\mathcal{MA}\mathit{lg}_{K}$, that is
$$\text{hom}_{\mathcal{MC'}_{K}(1)}\big(R,S\big)=\text{hom}_{\mathcal{MA}\mathit{lg}_{K}}\big(\mathcal{P}({R}),\mathcal{P}({S})\big),\ \ \forall R,S\in\ob(\mathcal{MC'}_K(1) ),$$
under the same compositions as in $\mathcal{MA}\mathit{lg}_{K}$.
%\(f\in \text{hom}_{\mathcal{MA}\mathit{lg}_{K}}\big(\mathcal{P}({A}),\mathcal{P}({B})\big)\)
Also choose the canonical coherent algebra isomorphisms
\begin{equation}\label{convert}
\mathcal{I}_{R,S} : \mathcal{P}(R) \otimes_K \mathcal{P}(S)\xrightarrow{\sim}\mathcal{P}(RS)\end{equation}
built from associators and unitors.
The monoidal structure of $\mathcal{MC'}_K(1)$ is defined by the concatination rules

$$\begin{array}{ccc}\ob(\mathcal{MC'}_K(1))\times\ob(\mathcal{MC'}_K(1))&\to&\ob(\mathcal{MC'}_K(1))\\
(R,R')&\mapsto& RR'
\end{array},$$
and for every $R,R',S,S'\in\ob(\mathcal{MC'}_K(1))$
\begin{eqnarray}\label{concatprod}\ \ \ \ \ \ \begin{array}{ccc}
\hom(S,R)\times\hom(S',R')&\to&\hom(SS',RR')\\
([_{R}P_S],[_{R'}P'_{S'}])&\mapsto& [P\otimes_1 P']
%:\mathcal{P}({R})\otimes_K\mathcal{P}({S})\to\mathcal{P}({R'})\otimes_K\mathcal{P}({S'})
\end{array}.\end{eqnarray}
Here, $P\otimes_1P'$ has the same underlying $K$-bimodule structure as $P\otimes_K P'$. Denote its elements
by $\sum_jp_j\otimes_1p'_j$. The Mod$(RR'\m SS')$ structure of $P\otimes_1 P'$ is determined using the same $(R\otimes_KR'\m S\otimes_KS')$-bimodule notation of $P\otimes_K P'$
$$(r\otimes_Kr')\star(p\otimes_Kp')\cdot (s\otimes_Ks'):=(r\star p\cdot s)\otimes_K(r'\star p'\cdot s') ,\text{  by}$$
\begin{equation}\label{otimes1}r'\star(p\otimes_1p')\cdot ss':=\mathcal{I}^{-1}_{R,R'}(rr')\star(p\otimes_Kq)\cdot \mathcal{I}^{-1}_{S,S'}(ss')\end{equation}
for every $r'\in R',s'\in S', r\in R, s\in S$ and $p\otimes_1p'\in P\otimes_1 P',$ where

Then $\mathcal{MC'}_K(1)$ is a strict monoidal category, admitting a strict unit, namely the empty word.
Since $K$ does not belong to the alphabet $C'$,
%the original $K$-tensoring notation $\otimes_K$ rather than the concatination notation in $\ob(\mathcal{MC'}_K(1))$. Accordingly,
the strict unit will still be denoted by $K$.
We will use the bimodule notation $[_RP_S]$ for morphisms in $\text{hom}_{\mathcal{MC'}_{K}(1)}\big(R,S\big)$.
\begin{remark}\label{Kcase}
With the identification \eqref{Picsub} we have
$\Pic_K(R)\cong\text{end}_{\mathcal{MC'}_K(1)}(R)$ for every  $R\in\ob(\mathcal{MC'}_K(1)).$
In the special case where $C(K) := \{[K]\}$ (\S\ref{secCK}(i)), $C'$ is an empty set and the monoidal groupoids $\mathcal{MC}'_K$ and $\mathcal{MC}'_K(1)$ admit a single object $K$.
They are therefore identified with the groups $\Pic_k(K)$ and $\Pic(K)$ respectively.
\end{remark}

Since $\mathcal{MC}_K(1)$ inherits property \eqref{MonGroup} from $\mathcal{MA}lg_K$ and induces it to $\mathcal{MC'}_K(1)$, we have
\begin{corollary}\label{MCK1strmongrp}
With the above notation, the categories $\mathcal{MC}_K(1)$ and $\mathcal{MC'}_K(1)$ are monoidal groupoids, the latter being strict.
\end{corollary}
Let $\mathcal{MC}'_K$ be the groupoid whose objects are the same as $\ob(\mathcal{MC}'_K(1))$ and whose morphisms are the classes of all invertible bimodules
(as in $\hom(\mathcal{MC}_K)$).

Under the identification \eqref{Picsub}, the conjugation action of the groupoid $\mathcal{MC}'_K$ on $\bigsqcup_{R\in ob(\mathcal{MC}'_{K})}$Pic$(R)$
is a particular instance of Example \ref{conjex} as follows.
Start with the anchor map (see \S\ref{mongroupsec})
$$\varsigma:[M]\in \Pic(R)\mapsto R\in\ob(\mathcal{MC}'_K).$$
%takes $ $ to the class $[B]\in C(K).$

Then $([Q],[M])$ is an admissible pair for every $[Q]=[_SQ_R]\in \hom(R,S)$.
For any such admissible pair, let
\begin{equation}\label{M^C}
[Q]([M]) = [Q(M)]\in \Pic(S),\text{ where } Q(M):= Q \otimes_{R} M \otimes_R \check{Q}.
\end{equation}
Then we have
\begin{lemma}\label{wdPc}
The rule \eqref{M^C} gives rise to a well-defined action of the groupoid $\mathcal{MC}'_K$,
and hence also of the groupoid $\mathcal{MC}'_K(1)$, on $\bigsqcup_{R\in \ob({\mathcal{MC}'_{K})}}$Pic$(R)$.
\end{lemma}

The set $\bigsqcup_{R\in \ob({\mathcal{MC}'_K)}}\Pic(R)=\bigsqcup_{R\in \ob({\mathcal{MC}'_K(1))}}\Pic(R)$ is sliced in \S\ref{Picsec} into a disjoint union of monoids,
each of which is acted on by $\mathcal{MC}'_K(1)$ as a monoidal groupoid.

\subsection{The type of a morphism}
The homomorphism (see \cite[\S II.5]{Bass})
\begin{eqnarray}\label{hR0}
%h_R:\begin{array}{ccc}\Pic(R)&\to&\Aut(K)\\ ~[P]&\mapsto &\eta_{P}\end{array}
h_R: \Pic(R)\to\Aut(K)\end{eqnarray}
%assigning $K$-automorphisms to classes in the Picard group,
actually comes from a general property of $\mathcal{MC'}_K$ (and even of $\mathcal{MC}_K$) as
follows. For $K$-central algebras $R,S$, let
$$ [_RP_S]\in \Mod^*(R\m S)\subset\hom(\mathcal{MC}_K).$$
%By the bimodule endomorphism rings of both $P$ and $\check{P}$ are isomorphic to $K$.
Any $\varphi\in$End$(_RP_S)$ is given either by right multiplication
$p\mapsto p\cdot x$ or by left multiplication
$p\mapsto y\star p$, where $x,y\in K$ are uniquely
determined by $\varphi$ (see \cite[Theorem II.3.5(5)]{Bass}). Therefore, any invertible class $[_RP_S]$
determines a well-defined bijection $\eta_{P}:x\mapsto y$
of $K$
\begin{equation}\label{commrule}\eta_{P}(x)\star p=p\cdot x,\ \forall x\in K,\ p\in P.\end{equation}
A direct computation yields
\begin{lemma}
With the notation \eqref{commrule}, the bijection $\eta_{P}$ is a $K$-automorphism. Moreover, if $ [_RP_S]=[_R\tilde{P}_S]$ then $\eta_{P}=\eta_{\tilde{P}}$.
\end{lemma}
We say that $[_RP_S]$ is {\em of type} $\eta_{P}\in\Aut(K).$
\begin{lemma}
With the notation \eqref{commrule}, the well-defined map
\begin{eqnarray}\label{hR}
h:\begin{array}{ccc}
\text{hom}(\mathcal{MC'}_K)&\to&\Aut(K)\\
 ~ [_RP_S]_{}&\mapsto &\eta_{P}
\end{array}\end{eqnarray}
is a homomorphism of groupoids with $\ker(h)=\text{hom}(\mathcal{MC'}_K(1))$, such that with the notation \eqref{hR0}
\begin{equation}\label{hhR}
h([_RP_R]_{})=h_R([_RP_R]),\ \ \forall [_RP_R]\in\Pic(R), R\in \text{ob}(\mathcal{MC'}_K).\end{equation}
%\ChA{ Shall we keep h from $\hom(\mathcal{MC}_K)$?}
\end{lemma}
\begin{proof}
Let $[_RP_S]_{},[_SQ_{R'}]_{} \in \text{hom}(\mathcal{MC'}_K)$ be a composable pair.
Then for $x \in K, p \in P$ and $q \in Q$, we have
\begin{align*}
(p \otimes_S q)\cdot x = p \otimes_S(q\cdot x)=p \otimes_S \eta_Q(x)\star q =p\cdot\eta_Q(x)\otimes_S q =\eta_P(\eta_Q(x))\star(p \otimes_Sq).\end{align*}
Since the type is well-defined on bimodule isomorphism classes we obtain
$$ \eta_{([P]_{}\circ [Q]_{})}=\eta_{P\otimes_S Q}=\eta_P\circ\eta_{Q} =\eta_{[P]_{}}\circ\eta_{[Q]_{}}$$ proving the
multiplicative property of \eqref{hR}. The two other claims are straightforward.
\end{proof}

\subsection{Skew products}\label{Picsec}

Given an automorphism $\eta\in\Aut_k(K)$ that fixes the subring $k\subseteq K$ elementwise,
we endow $h^{-1}(\eta)< \text{hom}(\mathcal{MC'}_K)$ (see \eqref{hR}) with a monoidal structure as follows.
Fix two invertible $\eta$-classes
%\begin{equation}\label{twoinvbmods}[P]\in \Mod(R\m S),\ \ [P']\in \Mod(R'\m S'),\ \ [P],[P']\in h^{-1}(\eta),
%\end{equation}
$[_RP_S],[_{R'}P'_{S'}]\in h^{-1}(\eta)$,
noticing that the actions of $k$ on both $P$ and $P'$ from the left and from the right are the same.
Define the $\eta$-\textit{skew product}
%[C]\otimes_{\eta}[C']:=[(A\otimes_K A',B\otimes_K B',P\otimes_{\eta} P',Q\otimes_{\eta} Q',\tau\otimes_{\eta}\tau',\mu\otimes_{\eta}\mu')],
\begin{equation}\label{eq:1}
P\otimes_{\eta}P':=\bigslant{P\otimes_{k}P'}{W_{\eta}},\end{equation} where $W_{\eta}\subseteq P\otimes_{k}P'$ is the $k$-submodule generated by the elements
$$
\ p\cdot x\otimes_kp'-p\otimes_k\eta(x)\star p',\ \  \forall p\otimes_kp'\in P\otimes_{k}P',\ \forall x\in K.$$
The elements in $P\otimes_{\eta}P'$, denoted by
$\sum_ip^i\otimes_{\eta}p'^i,$ certainly satisfy
\begin{equation}\label{scalpass}
\ p\cdot x\otimes_{\eta}p'=p\otimes_{\eta}\eta(x)\star p',\ \  \forall p\otimes_{\eta}p'\in P\otimes_{\eta}P',\ \forall x\in K.
\end{equation}

Then it is not hard to verify that $P\otimes_{\eta} P'$ is an
$(R\otimes_KR'\m S\otimes_KS')$-bimodule by the rule
$$(r\otimes_Kr')\star(p\otimes_{\eta}p')\cdot (s\otimes_Ks'):=(r\star p\cdot s)\otimes_{\eta}(r'\star p'\cdot s').$$
In turn, $P \otimes_{\eta} P'$ is an $(R R')\m (S S')$-bimodule utilizing \eqref{convert} and using the above notation as follows
$$\begin{array}{ccc}
(rr')\star(p \otimes_{\eta} p')&:=&\mathcal{I}^{-1}_{R,R'}(rr')\star ( p \otimes_{\eta}  p'), \text{ and}\\
 (p \otimes_{\eta} p')\cdot (ss') &:=& (p \otimes_{\eta} p')\cdot \mathcal{I}^{-1}_{S,S'}(ss')\end{array}$$
for every $r\in R, r'\in R'  ,s \in S, s'\in  S'$ and $p \otimes_{\eta} p'\in P \otimes_{\eta} P'.$

Further, let $[Q]:=[P]^{-1}\in \Mod^*(S\m R),\ \ [Q']:=[P']^{-1}\in\Mod^*(S'\m R').$
Then since \eqref{hR} is a homomorphism we have $h([Q])=h([Q'])=\eta^{-1}$, and obtain an $(S S')\m (R R')$-bimodule $Q \otimes_{\eta^{-1}} Q'$.
Moreover, a direct computation says that $[Q \otimes_{\eta^{-1}} Q']=[P \otimes_{\eta} P']^{-1}$, and so $[P \otimes_{\eta} P']\in \text{hom}(\mathcal{MC}_K).$
It turns out that the class $[P \otimes_{\eta} P']$ does not depend on the choice of the bimodules $P$ and $P'$.
We may hence well-define a product
\begin{eqnarray}\label{eq:1cl}
\begin{array}{ccc}
h^{-1}(\eta)\times h^{-1}(\eta)&\to&h^{-1}(\eta)\\
([P],[Q])&\mapsto& [P]\otimes_{\eta}[Q]:=[P\otimes_{\eta} Q]
\end{array},\end{eqnarray}
which we continue to term a skew product. Note that the notation $\otimes_1$ in \eqref{otimes1} is a special instance of $\otimes_{\eta}$ in \eqref{eq:1cl} for trivial automorphism $\eta$.
Next, with the notation \eqref{star}, distinguish the class
$[K_{1,\eta}]\in h^{-1}(\eta)\cap \text{end}(K).$
Using the strict monoidal structure of $\ob({\mathcal{MC}'_{K})}=\ob({\mathcal{MC}'_{K}(1))}$ one can check
\begin{lemma}\label{skewclosure}
Let $\eta\in\Aut_k(K)$. Then the skew product \eqref{eq:1cl} furnishes $h^{-1}(\eta)$ with a monoid structure,
%\begin{equation}\label{eq:1cl}[P]\otimes_{\eta}[P']:= [P\otimes_{\eta}P'],\ \ \forall [P],[P']\in h^{-1}(\eta),\end{equation}
whose identity is $[K_{1,\eta}]$.
\end{lemma}

Consider the union of the Picard elements inside the monoid $h^{-1}(\eta)$, i.e.
\begin{equation}\label{etaPic} A_{\eta}:=\bigsqcup_{R\in \ob({\mathcal{MC}'_{K})}}\Pic(R)\cap h^{-1}(\eta)\subseteq\text{hom}(\mathcal{MC}'_K).\end{equation}
It is not hard to verify that it is a submonoid of $h^{-1}(\eta)$. %In fact, the product \eqref{eq:1cl} extends \cite[\S 2]{YG1}, which was defined for \eqref{etaPic}.
%$$\Pic(R)\cap h^{-1}(\eta)<\text{hom}(\mathcal{MC}_K).$$
With the notation of Lemma \ref{wdPc} we assert
\begin{lemma}
Let $\eta\in\Aut_k(K)$. Then the submonoid \eqref{etaPic} of $h^{-1}(\eta)$ is stable under the conjugation action \eqref{M^C} of $\text{hom}(\mathcal{MC}'_K(1)).$
\end{lemma}
\begin{proof} The claim follows from the fact that \eqref{hR} is a homomorphism, whose kernel is $\text{hom}(\mathcal{MC'}_K(1))$.
\end{proof}
\begin{lemma}\label{MC1actPic}
Fix $\eta\in\Aut_k(K)$. With the notation \eqref{etaPic}, let
$$([_{S_1}{P_1}_{R_1}]_{},[_{R_1}{P'
_1}_{R_1}]_{}),([_{S_2}{P_2}_{R_2}]_{},[_{R_2}{P_2'}_{R_2}]_{})
%([P_1],[P_1']),([P_2],[P'_2])
\in\hom(\mathcal{MC'}_K(1))\times A_{\eta}$$
%\bigsqcup_{R\in \ob(\mathcal{MC}_{K})}\Pic(R)\cap h^{-1}(\eta)$$
be two admissible pairs. Then so is
$({[P_1]_{}\otimes_1[P_2]_{}},[P'_1]_{}\otimes_{\eta}[P'_2]_{})$. Furthermore
%Furthermore, with the notation \eqref{M^C},
\begin{equation}\label{eqcompat}
([P_1]_{}\otimes_1[P_2]_{})\big{(}[P'_1]_{}\otimes_{\eta}[P'_2]_{}\big{)}=[P_1]_{}([P'_1]_{})\otimes_{\eta}[P_2]_{}([P'_2]_{}).
\end{equation}
Consequently, \eqref{M^C} determines an action of $\mathcal{MC'}_K(1)$ on the monoid \eqref{etaPic} as a monoidal groupoid.
%for any two admissible pairs $$([P_1],[P_1']),([P_2],[P'_2])\in\hom(\mathcal{MC}_K(1))\times \bigsqcup_{R\in \ob(\mathcal{MC}_{K})}\Pic(R)\cap h^{-1}(\eta).$$
%Consequently, \eqref{M^C} determines an action of $\mathcal{MC}_K(1)$on the monoid \eqref{etaPic} as a monoidal groupoid.
\end{lemma}
\begin{proof}
Putting $R:=R_1 R_2$, a direct computation yields that the rule
$$\begin{array}{ccc}
P_1(P_1')\otimes_{\eta}P_2(P_2')&\to & (P_1\otimes_1 P_2)\big{(}P_1'\otimes_{\eta}P_2'\big{)}\\
(p_1\otimes_{R_1} p_1' \otimes_{R_1} q_1) \otimes_{\eta}(p_2\otimes_{R_2} p_2' \otimes_{R_2} q_2)&\mapsto &(p_1\otimes_1 p_2)\otimes_{R}(p_1' \otimes_{\eta} p_2')
\otimes_{R} (q_1 \otimes_1 q_2)
\end{array}$$
%$$\begin{array}{ccc}
%[P_1']^{[P_1]}\otimes_{\eta}[P_2']^{[P_2]}&\to & ([P_1']\otimes_{\eta}[P_2'])^{[P_1]\otimes [P_2]}\\
%[q_1\otimes_{R_1} p_1' \otimes_{R_1} p_1] \otimes_{\eta}(q_2\otimes_{R_2} p_2' \otimes_{R_2} p_2)&\mapsto &(q_1\otimes_K q_2)\otimes_{R}(p_1' \otimes_{\eta} p_2')
%\otimes_{R} (p_1 \otimes_K p_2)
%\end{array}$$
for every $ p_1\in P_1, p_2 \in P_2, q_1\in \check{P_1}, q_2\in \check{P_2}$, $p_1'\in P_1'$ and $p_2'\in P_2'$, well defines an $S_1 S_2$-bimodule isomorphism,
%Conjugating further by the converter $\mathcal{I}_{S_1,S_2}$ \eqref{convert} yields an $(S_1S_2)$-bimodule isomorphism,
verifying \eqref{eqcompat}. This is just \eqref{secondition} in terms of the action \eqref{M^C}, proving that this action respects the monoidal structure of the groupoid $\mathcal{MC}'_K(1)$.
\end{proof}

\section{An action of $\mathcal{MC}'_K$ on the Generalized Collective Characters}\label{Sec4}

After a short reminder about generalized collective characters we explain in \S\ref{subsec1} how the groupoid $\mathcal{MC}'_K$ is acting on them.
\begin{definition} \cite[\S 2]{CG}
A \textit{Generalized Collective Character} (GCC) is a group homomorphism from a group $G$ to the Picard group $\Pic_k(R)$ of a $k$-algebra $R$.
\end{definition}
Any GCC in Hom$(G,$Pic$_k(R))$ gives rise to a $G$-action on the center of $R$ as follows.
The homomorphism \eqref{hR0} yields a well-defined functorial map
$$\begin{array}{ccc}
\mbox{Hom}(G, \Pic_k(R)) &\rightarrow &\mbox{Hom}(G,\Aut_k(K))\\
\Phi & \mapsto &\ g\mapsto h_R\circ{\Phi(g)}
\end{array},\   g \in G.$$
We say that the GCC $\Phi \in \mbox{Hom}(G, \Pic_k(R))$ is \textit{of type} $h_R\circ{\Phi}$.

%\ChY{where are we using this? Let $$C_k(G,K) := \{[\Phi]_{eq}|\ \Phi\in\mbox{Hom}(G, \Pic_k(R))\ \mbox{and} \ R\in C(K)\}.$$ }
Given a group action $\phi \in $ Hom$(G,\Aut_k(K))$, use \eqref{hR0} to denote
%\begin{equation}\label{CkpR}C_k(\phi,R) := \{[\Phi]_{\eq}|\ \Phi\in\mbox{Hom}(G, \Pic_k(R))\ \mbox{and} \ h_R\circ{\Phi}=\phi\},\end{equation} and let
\begin{equation}\label{Ckp}
{\CK} := \{ \Phi\in\mbox{Hom}(G, \Pic_k(R)) |\ {R \in \ob(\mathcal{MC}'_{K})} \mbox{ and} \ h_R\circ{\Phi}=\phi\}.
\end{equation}
Let $\Phi_1:G\to\Pic_k(R_1),\Phi_2:G\to\Pic_k(R_2)\in \CK$.
%Fix $\Phi_1\in\mbox{Hom}(G, \Pic_k(R_1)),\Phi_2\in\mbox{Hom}(G, \Pic_k(R_2))$ in $ \CK$.
With the notation \eqref{eq:1cl} set
\begin{eqnarray}\label{prodck}\Phi_1\otimes_{\phi}\Phi_2:\begin{array}{ccl}
G&\to&\p (R_1R_2)\\
g&\mapsto &\Phi_1(g)\otimes_{\phi(g)}\Phi_2(g)
\end{array},\ \  g\in G.
\end{eqnarray}
Similarly to \cite[Lemma 3.2(1),(2)]{YG1} (for the external tensor product of central algebras,) one can show that $\Phi_1 \otimes_{\phi} \Phi_2$ is a well-defined group homomorphism of type $\phi$.
Furthermore, it is not hard to check that this product forms a monoid structure on \eqref{Ckp} with $1_{\phi}\in \CK$ as the identity, where
\begin{eqnarray}\label{1_phi}
1_{\phi}:\begin{array}{ccc}
G&\to&\Pic_k(K)\\
g&\mapsto&[K_{1,\phi(g)}]\end{array},\forall g\in G.
\end{eqnarray}

Note that since the singleton $\{K\}$ is closed under multiplication in Ob($\MCK'$), the GCC's of type $\phi$ over $K$ itself
\begin{equation}\label{CkPhiK}
{\CPK}:= \{\Phi\in\mbox{Hom}(G, \Pic_k(K))\ |\ h_K\circ{\Phi}=\phi\}\end{equation}
form a submonoid of the monoid $\CK$.

\subsection{}\label{subsec1}

The conjugation \eqref{M^C} yields an action of  the groupoid $\mathcal{MC}'_K$ on
\begin{equation}\label{GCCset}
\{ \Phi\in\mbox{Hom}(G, \Pic_k(R))|\ {R\in \ob(\mathcal{MC}'_{K})}\}
\end{equation}
of GCC's as follows.
%\ChY{Let us describe this action, which together with \eqref{actrule} on the set $\bigsqcup_{[R]\in C(K)}\mbox{Cliff}_k (G, R)$, turn \eqref{chi} into an $\mathcal{MC}_K$-map
%(see Lemma \ref{compat}).}
First, we define anchor map. For a GCC $\Phi: G \to \Pic_k(R)$, where $R\in\ob(\mathcal{MC}'_K)$, let
$$\varsigma:\Phi \mapsto R.$$
%takes $ $ to the class $[B]\in C(K).$
Then $\Phi$ as above is admissibly paired to any $[P]=[_SP_R]\in \hom(R,S)$, where $S\in\ob(\mathcal{MC}'_K)$.
% Next, for an admissible pair, is for every $[P]=[_RP_S]\in \hom(R,S)$.
For any such pair $([P],\Phi)$, define (see \eqref{M^C})
\begin{eqnarray}\label{GCC^C}
[P](\Phi):\begin{array}{rcc} G &\to &\Pic_k(S)\\
g&\mapsto &[P](\Phi(g))\end{array}.\end{eqnarray}
%\begin{equation}\label{M^C}[M]^{[P]} := [\check{P} \otimes_{R} M \otimes_R P]\in \Pic(S).\end{equation}
We infer from Lemma \ref{wdPc}
\begin{lemma}\label{wdphic}
The rule \eqref{GCC^C} determines a well-defined action of the groupoid $\mathcal{MC}'_K$, and hence of its subgroupoid $\mathcal{MC}'_K(1)$,
on \eqref{GCCset}.
%disjoint union $\bigsqcup_{[R]\in C(K)}\{[\Phi]_{\eq}|\ \Phi\in\mbox{Hom}(G, \Pic_k(R))\}$.
\end{lemma}
The set \eqref{GCCset} can be sliced as a disjoint union according to the various $k$-actions of $G$ on $K$
(see \eqref{Ckp})
\begin{equation}\label{slices}
\{\Phi\in\mbox{Hom}(G, \Pic_k(R))|\ {R\in \ob(\mathcal{MC}'_{K})}\}=\bigsqcup_{\phi\in \Hom(G,\Aut_k(K))}\CK.
\end{equation}
The right hand side of \eqref{slices} is a decomposition into monoids under the various skew products.
The monoid $\CK$ is stable under the $\mathcal{MC}'_K$-action \eqref{GCC^C} if and only if $\phi$ is central in $\Hom(G,\Aut_k(K))$.
Let us focus on the action of the monoidal subgroupoid $\mathcal{MC}_K'(1)<\mathcal{MC}'_K$, which stabilizes all monoids $\CK$, respecting their skew products as follows.
%Furthermore, the action \eqref{P(Phi)} respects the product in $\CK$ as follows.
\begin{lemma}\label{CPiclosure}
Fix $\phi\in \Hom(G,\Aut_k(K))$, and let
$$([P_1],\Phi_1),([P_2],\Phi_2) \in \hom(\mathcal{MC}'_K(1))\times\CK$$
%(\subset \hom(\mathcal{MC}_K)\times\bigsqcup_{[R]\in C(K)}\{[\Phi]_{\eq}|\ \Phi\in\mbox{Hom}(G, \Pic_k(R))\})
be two admissible pairs.
%With the above notation, let $$([P_1],[P_1']),([P_2],[P'_2])\in\hom(\mathcal{MC}_K(1))\times \bigsqcup_{R\in \ob(\mathcal{MC}_{K})}\Pic(R)\cap h^{-1}(\eta)$$ be two admissible pairs.
Then so is $({[P_1]\otimes_1[P_2]},\Phi_1\otimes_{\phi}\Phi_2)$. Furthermore,
\begin{equation}\label{eqnew}
([P_1]\otimes_1[P_2])\big{(}\Phi_1\otimes_{\phi}\Phi_2\big{)}=[P_1]({\Phi_1})\otimes_{\phi}[P_2]({\Phi_2}).
\end{equation}
\end{lemma}
\begin{proof}
The claim is a consequence of Lemma \ref{MC1actPic}, and particularly of \eqref{eqcompat}.
\end{proof}
Equation \eqref{eqnew} is just condition \eqref{secondition} in terms of $\mathcal{MC}'_K(1)$ and $\CK$.
Recall (Corollary \ref{MCK1strmongrp}) that the groupoid $\mathcal{MC}'_K(1)$ is monoidal.
In light of Definition \ref{Defact} and Lemma \ref{orblemma} we conclude
\begin{corollary}\label{coract}
The rule \eqref{GCC^C} determines an action of ${\mathcal{MC}_K'(1)}$ on the monoid $\CK$ as a monoidal groupoid,
and in turn a monoidal structure on $\CK/\MCK'(1)$.
\end{corollary}

\section{An action of $\mathcal{MC}'_K$ on Clifford System Extensions}\label{Sec3}

This section introduces an action of the groupoid $\mathcal{MC}'_K$ on Clifford system extensions. Let us explain what the latter means.
\begin{definition}
Let $G$ be a group with identity $e$ and let $K$ be a commutative ring containing a subring $k.$
A $k$-algebra is \textit{strongly $G$-graded} if it admits an additive decomposition
\begin{equation}\label{eq:D:1}
A_G = \bigoplus_{g \in G} A_g
\end{equation}
such that for every $g,h \in G$
\begin{equation}\label{AgAh}A_g\cdot A_h=A_{gh}.\end{equation}
\end{definition}
Condition \eqref{AgAh} yields that the trivial homogeneous component $A_e$ is in particular a subalgebra of $A_G$
called its {\it base algebra} (or its {\it e-component}).

A strongly G-graded k-algebra \eqref{eq:D:1} is a {\em crossed product} if every homogeneous component $A_g$ contains a unit, say $u_g$.
In this case, $A_G$ is usually written as
\begin{equation}\label{CPnota}
A_e*G:=\bigoplus_{g\in G}A_e\cdot u_g=\bigoplus_{g\in G}u_g\cdot A_e.\end{equation}
\begin{definition}\label{cliffextdef}%\begin{enumerate}\item
(\textit{i}) A $k$-algebra homomorphism
$\psi_G : A_G \rightarrow A'_G$ between two strongly $G$-graded $k$-algebras
\begin{equation}\label{AA'}
A_G = \bigoplus_{g \in G}A_g \text{ and }A'_G =  \bigoplus_{g \in G}A_g'
\end{equation}
%\eqref{eq:D:1} and another $A' = \oplus_{g\in G}A'_g$
is a \textit{graded homomorphism} if $\psi_G(A_g) = A'_g$ for every $g \in G$.
The graded algebras \eqref{AA'} are \textit{graded isomorphic}, and written $A_G\stackrel{G}\cong A'_G$, if there exists a graded isomorphism between them.

%In this paper we are concerned with the setting where the subalgebra $A_e$ is $k$-isomorphic to a given $K$-central algebra $R$.\\
(\textit{ii})(see \cite[Definition 2.1]{CG}) A \textit{$G$-graded $k$-Clifford system
extension} of $R$ is a pair $(A_G,\iota)$ where $A_G$ is a strongly
$G$-graded $k$-algebra \eqref{eq:D:1} and $\iota : R \to A_G$ is a
$k$-algebra embedding with $\iota(R) = A_e$.
%\ChY{{\em (iii)} Two $G$-graded Clifford system extensions $(A_G,\iota)$ and $(A'_G,\iota')$ of $R$ and $R'$ respectively are equivalent if
%there exists a $K$-isomorphism $\psi:R\to R'$ and a graded isomorphism $\psi_G : A_G \xrightarrow{G} A'_G $ such that $\psi_G\circ \iota = \iota'\circ\psi$.}
Two $G$-graded Clifford
system extensions $(A_G,\iota)$ and $(A'_G,\iota')$ of $R$ are equivalent if
there exists a graded isomorphism $\psi_G : A_G \xrightarrow{G} A'_G $ such
that $\psi_G\circ \iota = \iota'$. Denote the equivalence class of $(A_G,\iota)$  by $[(A_G,\iota)]$.
\end{definition}
In fact, it is unnecessary to check invertibility when verifying equivalence of Clifford extension classes as follows.
\begin{theorem}\label{grhom}(see \cite[Corollary 2.10]{Dade80})
Two $G$-graded Clifford system extensions $(A_G,\iota)$ and $(A'_G,\iota')$ of $R$ are equivalent if and only if there exists a graded homomorphism $\psi_G: A_G \to A'_G$ such
that $\psi_G\circ \iota = \iota'$.
\end{theorem}
%\end{enumerate}

As can easily be checked \eqref{eq:D:1} describes a decomposition of a $G$-graded $k$-Clifford system
extension $(A_G,\iota)$ of $R$ as a direct sum of $R$-bimodules.
Moreover, by \eqref{AgAh} the homogeneous components $A_g$ are invertible $R$-bimodules, that is $[A_g ] \in \Pic_k (R)$
(notice that the $R$-bimodule structure of $A_g$ depends on the embedding $\iota$).
%A strongly graded algebra \eqref{eq:D:1} admits a GCC by assigning the invertible class $[A_g ] \in \Pic_k (A_e )$ to a group element $g \in G$.

Let Cliff$_k(G, R)$ denote the set of equivalence classes of
$G$-graded Clifford system extensions of $R$.
Then there is a well-defined GCC map \cite[(4.2)]{YG1}
%\begin{eqnarray}\label{chi*}
\begin{equation}\label{prechi}\chi_R:\begin{array}{ccc}
\mbox{Cliff}_k (G, R) &\rightarrow\  &\mbox{Hom}(G, \Pic_k(R))\\
~[(\bigoplus_{g \in G} A_g,\iota)]  &\mapsto& (g\mapsto [A_g])
\end{array}.\end{equation}
%%\end{eqnarray}

 The following result is needed in the sequel.
\begin{lemma}\label{cpch} \cite[\S 3]{CG}, A $G$-graded Clifford system extension of $R$ is a crossed product if and only if its image under the GCC
map \eqref{prechi} is some outer action in $\mbox{Hom}(G, \Out_k(R))\hookrightarrow\mbox{Hom}(G, \Pic_k(R))$ (see \eqref{out}).
\end{lemma}

Let $\text{Cliff}_k(G,C(K))$ denote the set of equivalence classes of $G$-graded $k$-Clifford system extensions of all the $K$-central algebras
in $\ob(\mathcal{MC}'_K)$.
Then the map \eqref{prechi} extends to
%over the isomorphism types of $K$-central algebras
\begin{equation}\label{chi}
\tilde{\chi}:\text{Cliff}_k(G,C(K)) \rightarrow \{\Phi\in\mbox{Hom}(G, \Pic_k(R))|\ {\mbox{for}\ R \in \ob(\MCK')}\}.
\end{equation}
Fix $\phi\in\Hom(G,\Aut_k(K))$. With the notation \eqref{Ckp}, let
%$$\ClfR := \chi{} ^{-1}(C_k(\phi,R)), \text{  and  }$$
\begin{equation}\label{ClfPhi}
\tilde{\mbox{Cliff}}(\phi) := \tilde{\chi} ^{-1}(\CK).\end{equation}
 We say that a Clifford system extension class in
$\tilde{\mbox{Cliff}}(\phi)$ is of \textit{type} $\phi.$
%Let $\ClfR := \chi{^*} ^{-1}(C_k(\phi,R))$ denote the Clifford system extension classes of $R$ of \textit{type} $\phi$, and set \ChY{Cliff$_R := \ClfR$}.\\

Let $\phi\in$ Hom$_k(G,\mbox{Aut}_k(K))$ be a $G$-action on a commutative ring $K$.
Then a crossed product $K*G=\bigoplus_{g\in G}K\cdot u_g$ of type $\phi$ determines a two-place function
$c:\begin{array}{ccc}G\times G & \to & K^*\\
(g,h)& \mapsto & u_g\cdot u_h\cdot u_{gh}^{-1}\end{array}\in Z^2(G,K_{\phi}^*),$ that is a 2-cocycle.
We denote this crossed product, subjected to the choice of units $\{u_g\}_{g\in G}$, by $K_{\phi}^{c}G$.
It is common to choose the unit in the $e$-component to be 1. In this case, the corresponding 2-cocycle is called {\em normalized}.
When $c=1\in Z^2(G,K_{\phi}^*)$, the corresponding crossed product $K_{\phi}G=K^1_{\phi}G$ is the {\it skew group algebra} of type $\phi$.
A Clifford system $K$-extension class $[(\iota(K)_{\phi}^{\iota(c)}G,\iota)]$ of a crossed product will be written $[K_{\phi}^{c}G]$ for short.

%Now a product in Cliff$_k(\phi)$ can be defined as follows:
Two strongly $G$-graded algebras \eqref{AA'}
% $$A = \oplus_{g \in G}A_g \text{ and }A' =  \oplus_{g \in G}A_g'$$
with $[A],[A'] \in\tilde{\mbox{Cliff}}(\phi)$ determine a new $G$-graded algebra
 \begin{equation}\label{newgralg}A\otimes_{\phi}A' := \bigoplus_{g\in G} (A_g \otimes_{\phi(g)}A_g'),\end{equation}
 where the homogeneous components of \eqref{newgralg} are defined using \eqref{eq:1}:
$$(A \otimes_{\phi} A')_g := A_g \otimes_{\phi(g)} A_g',$$
and where the multiplication in $A\otimes_{\phi}A'$ is given by
\begin{equation}\label{grprod}
(a_g\otimes_{\phi(g)} a'_g)(a_h\otimes_{\phi(h)} a'_h):=a_g\cdot a_h\otimes_{\phi(gh)} a'_g\cdot a'_h\in(A\otimes_{\phi}A')_{gh}
\end{equation} for
$$a_g\otimes_{\phi(g)} a'_g\in (A\otimes_{\phi}A')_g\text{ and }a_h\otimes_{\phi(h)} a'_h\in (A\otimes_{\phi}A')_h.$$

Similarly to \cite[Lemmas 4.5 and 4.6]{YG1}, it is not hard to show
\begin{lemma} With the above notations
\begin{enumerate}
\item $\tilde{\mbox{Cliff}}(\phi)$ is a monoid under the product \eqref{newgralg}, whose identity is the class $[K_{\phi}G ]$ of the corresponding skew group algebra.
Moreover,
\item the GCC map \eqref{chi} well-defines a homomorphism
\begin{equation}\label{chiphi}\tilde{\chi}_\phi: \tilde{\mbox{Cliff}}(\phi)\to \CK
\end{equation}   of monoids.
\end{enumerate}\end{lemma}
Note that those Clifford system extensions of type $\phi$ over $K$ itself
\begin{equation}\label{ClfkPhiK}
\ClfK:= \tilde{\chi}^{-1}(\CPK)\end{equation}
can be considered as a submonoid of the monoid $\Clf$
since the singleton $\{K\}$ is multiplicatively closed in Ob($\MCK'$).

\subsection{}
Based on the conjugation \eqref{M^C},
%of the monoidal groupoid $\mathcal{MC}_K$ on $\bigsqcup_{[R]\in C(K)}\mbox{Pic}(R)$,
we describe an action of $\mathcal{MC}'_K$ on $\text{Cliff}_k(G,C(K))$,
%the disjoint union $\bigsqcup_{[R]\in C(K)}\mbox{Cliff}_k (G, R)$,
resembling the action introduced by J. Haefner and A. Del R\'{i}o in \cite[Definition 2.5]{JHDR}, implementing therein the concept of Morita contexts.
The following properties of $[P]\in \Mod^*(S\m R)$ are needed to define this action.
The first one is a ring isomorphism (see \cite[Thm. II.3.4]{Bass})
\begin{eqnarray}\label{thetaP}\theta_P :\begin{array}{ccc}
S &\xrightarrow{\cong}& \End_R(P)\\
x &\mapsto &(p\mapsto x\star p)
\end{array}\ \ x \in S,p \in P.\end{eqnarray}
 The second property \cite[Prop. II.4.4]{Bass} is a bimodule isomorphism given $[P] \in \Mod^*(S\m R)$ via the evaluation map using $\check{P}:=\Hom_R(P,R)$ (see \eqref{Hom(P,R)})
\begin{eqnarray}\label{nuP}\tau_P :\begin{array}{ccc}
\check{P}\otimes_{\End_R(P)}P &\xrightarrow{\cong}& R\\
f\otimes_{\End_R(P)}p &\mapsto &  f(p)
\end{array}\ \ f \in \check{P},\ \ p \in P.\end{eqnarray}

Another bimodule isomorphism (see \cite[Proposition II.4.4(c)]{Bass}) is determined by the rule
\begin{eqnarray}\label{phiP}\mu_P :\begin{array}{ccc}
P\otimes_{R}\check{P} &\xrightarrow{\cong}& \End_{R}(P)\\
p\otimes_{R} f &\mapsto & (p'\mapsto p\cdot f(p'))
\end{array} \ f \in \check{P},\  p,p' \in P.\end{eqnarray}

Then with the notations \eqref{thetaP} and \eqref{phiP}, there is an $R$-module isomorphism
\begin{equation}\label{Delta_P}
\Delta_P :=\theta_P^{-1}\circ\mu_P:P\otimes_R \check{P} \xrightarrow{\cong} S
.\end{equation}
%In particular, $[\check{P}]=[P]^{-1}\in \p(R)$.
Note that for $p, p' \in P$ and $f \in \check{P}$, we have
\begin{equation}\label{DeltaTau}
\Delta_P(p \otimes_R f) \star p' = p\cdot f(p').
\end{equation}
Now, let $(B_G= \oplus_{g\in G} B_g,\iota)$ be a $G$-graded
$k$-Clifford system extension of a $K$-central algebra $R$ and let
$[P] \in \Mod^*(S\m R)$. Apply the action \eqref{M^C}
of $P$ on the homogeneous components of $B_G$ as well as
\eqref{nuP}, to define multiplication in $P(B_G) :=
\bigoplus_{g \in G}  P(B_g)$ by
\begin{equation}\label{mult}
{(p_1\otimes_{R} b_g \otimes_R f_1)\cdot(p_2\otimes_{R} b_h \otimes_R f_2) := p_1\otimes_{R} b_g\cdot \iota(\tau_P(f_1\otimes_{\End_R(P)} p_2))\cdot b_h\otimes_R f_2}
\end{equation}
for every $f_1,f_2 \in \check{P}, p_1,p_2 \in P, b_g \in B_g$ and $b_h \in B_h,$ endowing it with a $G$-graded algebra structure. With the notation \eqref{nuP}, set
\begin{equation}\label{A^C}
 P(\iota) :a\mapsto\sum_ja\star p_j\otimes_R\iota(1)\otimes_Rp_j'\in P(B_e),
\end{equation} where $\sum_jp_j\otimes_{R}  p_j'=\Delta_P^{-1}(1)\in P\otimes_{R} \check{P} .$
%\begin{equation}\label{A^C}
%B_G^C := \bigoplus_{g \in G}  B_g^C ,\mbox{ and }\  \iota^C :=(1\otimes_B \iota \otimes_B 1)\circ\tau^{-1},\end{equation}\ChY{Should we explain $\iota^C$?}
The fact that \eqref{A^C} is indeed a $k$-algebra homomorphism is left to the readers.

We can now define the following action of the groupoid $\mathcal{MC}'_K$ on $\text{Cliff}(G,C(K))$.\\
(a) the corresponding anchor map takes $\mbox{Cliff}_k(G, R)\subset\text{Cliff}_k(G,C(K))$ to $R\in\ob(\mathcal{MC}'_K)$.\\
(b) For any admissible pair $$([P],[(B_G,\iota)])\in\hom(R,S)\times\mbox{Cliff}_k (G, R)\subset\hom({\mathcal{MC}'_K)}\times \text{Cliff}_k(G,C(K))$$
let
\begin{equation}\label{actrule}
[P]([(B_G,\iota)]):=[(P(B_G),P(\iota))].\end{equation}

The following result was formulated in \cite{JHDR} in the language of Morita context isomorphism classes acting on graded isomorphism classes of Clifford system extensions
and can be adjusted to our language.
\begin{lemma}\label{HdRlemma}
The rule \eqref{actrule} determines a well-defined $\mathcal{MC}'_K$-action on $\text{Cliff}_k(G,C(K))$.
\end{lemma}
\begin{proof}
To prove that \eqref{actrule} is well-defined, let $[(B_G,\iota)]=[(B'_G,\iota')]\in\mbox{Cliff}_k (G, R)$, and let $[P]=[P'] \in \Mod^*(R'\m R)$.
We claim that $[(P(B_G),P(\iota))]=[(P'(B'_G),P'(\iota'))]$ (compare with \cite[Lemma 2.6]{JHDR}).
Indeed, let $\lambda: P \stackrel{\cong}{\to} P'$ be an $S\m R$ bimodule isomorphism admitting a dual $R\m R'$ bimodule isomorphism
${\lambda}^*: \begin{array}{rcl}\check{P'} &\to &\check{P}\\
f'&\mapsto & f'\circ\lambda\end{array}$.
Next, let $\psi_G: B_G \stackrel{G}{\to} B_G'$ be a $G$-graded $k$-algebra isomorphism
giving rise to an equivalence of the above Clifford system extensions of $R$, i.e. satisfying $\psi\circ \iota = \iota'$.
Then a direct computation yields that
$$P(\psi) :\begin{array}{rcl}
P(B_G)&\to &{P'}({B'}_G)\\
p \otimes_R b_g\otimes_R f&\mapsto & \lambda(p) \otimes_R \psi_G(b_g)\otimes_R {\lambda^*}^{-1}(f)
\end{array},\ \forall g \in G, b_g \in B_g, p \in P, f \in \check{P}$$
is a $G$-graded $k$-algebra isomorphism satisfying $P(\psi)\circ P(\iota) = P'(\iota')$, that is an equivalence of
$(P(B_G),P(\iota))$ and $(P'({B'_G}),P'({\iota'}))$ as Clifford system $R'$-extensions.

Now, for any Clifford system extension class $[(B_G,\iota)])\in\mbox{Cliff}_k (G, R)$ and two composable morphisms (see \eqref{compmods})
$[_{R'}P_{R}],[_{R''}P'_{R'}]\in\text{hom}({\mathcal{MC}'_K)}$,
the rule
$$ \begin{array}{ccc}
(P'\otimes_{R'} P)\big{(}B_G\big{)}& \to& (P'(P(B_G))\\
(p' \otimes_{R'} p)\otimes_R b_g \otimes_R (f \otimes_{R'} f')& \mapsto& p' \otimes_{R'}(p \otimes_R b_g \otimes_R f)\otimes_{R'} f',
\end{array}$$
for $f'\in \check{P'}, f \in \check{P}, p \in P, p' \in P'$, and $b_g \in B_g$ determines an isomorphism of $G$-graded $k$-algebras,
which respects the corresponding embeddings $(P'\otimes_{R'}P)(\iota)$ and $P'(P(\iota))$.
It follows that
$$([P'\otimes_{R'}P]([(B_G,\iota)]) = [P']([P]([(B_G,\iota)]))\in\mbox{Cliff}_k (G, R'')$$
(see \cite[Lemma 2.7]{JHDR}), that is $\mathcal{MC}'_K$ acts on $\text{Cliff}_k(G,C(K))$ as a groupoid.
\end{proof}
How is the type of a Clifford extension (see \eqref{chi}) affected by the action \eqref{actrule}? Since the extended map \eqref{hR} is a homomorphism, we have
%at this action induces an action of the isomorphism classes of Morita contexts on the Clifford system extensions this action is multiplicative ().
%Furthermore, when a Morita context $C := (R,R,P,Q,\tau,\mu)$ with $[P],[Q]\in$ Pic$(R)$ acts on a Clifford system extension of $R$, the action preserves the type as shown below.
\begin{lemma}\label{L:3.1}
Let $([P],[(B_G,\iota)]\in\hom({\mathcal{MC}'_K)})\times\tilde{\mbox{Cliff}}(\phi)$ be an admissible pair. Then
%the corresponding conjugated Clifford system extension class satisfies
%$$\eta\circ\chi([(B_G,\iota)]^{[P]})=h([P])\cdot\eta\circ\chi([(B_G,\iota)])\cdot h([P])^{-1}.$$
$$[P]([(B_G,\iota)])\in\tilde{\mbox{Cliff}} (\phi'),\ \mbox{for}\ \phi'= h([P])^{-1}\cdot\phi\cdot h([P]).$$
In particular, $\tilde{\mbox{Cliff}}(\phi)$ is stable under the action of the subgroupoid ${\mathcal{MC}_K'(1)}$ for every $\phi\in$Aut$_k(K)$.
%the $G$-graded algebras $\oplus_{g \in G}B_g$ and $(\oplus_{g \in G}B_g)^P$ are of the same type when $P$ and $Q$ admits the same $K$-action from both sides.
\end{lemma}

Owing to Lemma \ref{L:3.1}, we concentrate on the action \eqref{actrule} restricted to that of the subgroupoid ${\mathcal{MC}_K'(1)}$ on $\Clf$.
The skew product in $\Clf$ is respected by this action as follows.
\begin{lemma}\label{CLfPiclosure}
Let
$([P],[(B_G,\iota)]),([P'],[(B'_G,\iota')])\in\hom(\mathcal{MC}'_K(1))\times \Clf$
be two admissible pairs, where $\phi\in$Aut$_k(K)$. Then so is the pair
$$([P]\otimes_1[P'],[(B_G\otimes_{\phi}B'_G,\iota\otimes_{\phi}\iota')])\in\hom(\mathcal{MC}'_K(1))\times \Clf,$$
satisfying
\begin{equation}\label{eq: 3.3}
([P]\otimes_1[P'])\big{(}[(B_G,\iota)]\otimes_{\phi}[(B'_G,\iota')]\big{)}=[P]([(B_G,\iota)])\otimes_{\phi}[P']([(B'_G,\iota')]).
\end{equation}
\end{lemma}
%This result can alternatively be formulated as follows. With the above notation, for any two Morita contexts $C := (A,B,P,\check{P},\tau,\mu)$
\begin{proof}

Write explicitly $B_G = \oplus_{g \in G}B_g,\ \ B'_G = \oplus_{g \in G} B'_g$, $P={_S}P_R$ and $P'={_{S'}}P'_{R'}$.
From Lemma \ref{MC1actPic}, every $g \in G$ gives rise to an $S S'$-bimodule isomorphism
$$ P(B_g) \otimes_{\phi(g)} P'(B_g')  \to (P\otimes_1 P')\big{(}B_g \otimes_{\phi(g)} {B'}_g\big{)}.$$
These extend to a $G$-graded $S S'$-bimodule isomorphism of the corresponding direct sums
%\begin{equation}\label{sumext}
$$\psi_G:P(B_G)\otimes_{\phi} P'({B'}_G)\to (P\otimes_1 P')\big{(}B_G\otimes_{\phi}{B'}_G\big{)}.$$
%\end{equation}
The multiplicative property of $\psi_G$, as well as the equality
$$\psi_G\circ(P(\iota)\otimes_{\phi}P'(\iota'))=(P\otimes_1P')\big{(}\iota\otimes_{\phi}\iota'\big{)}$$
can be directly verified. From the well-definition in Lemma \ref{HdRlemma} we get \eqref{eq: 3.3}.
\end{proof}

Equation \eqref{eq: 3.3} is \eqref{secondition} for $\mathcal{MC}'_K(1)$ and $\Clf$.
%Recall (Corollary \ref{MCK1strmongrp}) that the groupoid $\mathcal{MC}_K(1)$ is monoidal. In light of Definition \ref{Defact} and Lemma \ref{orblemma} we conclude
As in Corollary \ref{coract} we have
\begin{corollary}\label{smgact}
The rule \eqref{actrule} determines an action of $\mathcal{MC}'_K(1)$ on the monoid $\Clf$ as a monoidal groupoid.
In particular, it endows a monoidal structure on $\Clf/\MCK'(1)$.
\end{corollary}

By Lemma \ref{cpch}, crossed products $K*G$ are $G$-graded Clifford systems extensions of $K$,
whose images under \eqref{chi} are identified with $G$-actions in Hom$(G,\Aut_k(K))$. We have
\begin{lemma}\label{P_inv}
Let $K*G$ be a crossed product and let $[P]\in \Pic(K)$.
Then with the notation \eqref{A^C}, $P(K*G)$ is a crossed product (of the same type) if and only if $[P]$ is $G$-invariant.
\end{lemma}
\begin{proof}
Both conditions are equivalent to $[P]\in \Pic(K)$ lying in the centralizer of $\{[K_{1,\phi(g)}]\}_{g\in G}<\Pic(K)$.
\end{proof}

\section{Proof of Theorem A}\label{Sec5}
The starting point for the exact sequence \eqref{main} is inspired from the realization-obstruction 4-term exact sequence of monoids of type $\phi \in$ Hom$(G,$ Aut$_k(K))$
\cite[Main Theorem]{YG1}, formulated there in terms of equivariance.

Every GCC $\Phi\in\Hom(G,\Pic_k(R))$ of type $\phi$ gives rise to a well-defined cohomology class $\tilde{T}_{\phi}(\Phi)\in H^3(G,K^*_{\phi})$ \cite[\S 3]{CG},
which is an obstruction to its realization
as the GCC associated to some Clifford extension of $R$. Formally
\begin{lemma}\label{reob}
With the notations in \eqref{Ckp} and \eqref{ClfPhi}
\begin{enumerate}
\item The map $\tilde{T}_{\phi}:\CK\rightarrow{} H^3(G,K^*_{\phi})$ is a homomorphism of monoids.
\item Together with the homomorphism \eqref{chiphi}, the sequence of monoids
\begin{eqnarray}\label{S02} \begin{array}{ccc}
 \Clf\xrightarrow{\tilde{\chi}_{\phi}} \CK \xrightarrow{\tilde{T}_{\phi}} H^3(G,K^*_{\phi})
\end{array}\end{eqnarray} is exact.
\end{enumerate}
\end{lemma}
\begin{proof}\begin{enumerate}
\item See \cite[(5.9)]{YG1}.
\item This realization-obstruction property follows from \cite[Theorem 3.4]{CG}.
\end{enumerate}
\end{proof}
As in \cite[Lemmas 4.8, 4.9]{YG1}, there is a well-defined homomorphism of monoids
$$\tilde{\Sigma}_{\phi}: \begin{array}{ccc}H^2(G,K_{\phi}^*)& \to& \Clf\\
~[c]&\mapsto & [K^{c}_{\phi}G]\end{array}.$$
%such that the composition $\tilde{\chi}_{\phi}\circ \tilde{\Sigma}_{\phi}$ is trivial.
By Lemma \ref{cpch}, the associated GCC of a Clifford extension $(A_G,\iota) \in \Clf$ is $1_{\phi}$ (see \eqref{1_phi})
if and only if it is a crossed product over $K$. Consequently,
\begin{lemma}
The sequence of monoids
\begin{eqnarray}\label{S2} \begin{array}{ccc}
\ 0 \xrightarrow{}  H^2(G,K^*_{\phi})\xrightarrow{\tilde{\Sigma}_{\phi}} \Clf\xrightarrow{\tilde{\chi}_{\phi}} \CK
\end{array}\end{eqnarray} is exact.
\end{lemma}

At this point, the actions of the monoidal groupid $\mathcal{MC}'_K(1)$ on the middle terms of the combined exact sequence
$$0 \xrightarrow{}  H^2(G,K^*_{\phi})\xrightarrow{\tilde{\Sigma}_{\phi}} \Clf\xrightarrow{\tilde{\chi}_{\phi}} \CK\xrightarrow{\tilde{T}_{\phi}} H^3(G,K^*_{\phi})$$
is incorporated.
First, recall that by Corollary \ref{coract}, $\CK/\MCK'(1)$ inherits a monoidal structure from $\CK$.

\begin{lemma}\label{Tphibar}
Let $([_SP_R],\Phi)\in$hom$(\mathcal{MC}'_K(1))\times\CK$ be an admissible pair.
Then $$\tilde{T}_{\phi}(\Phi)=\tilde{T}_{\phi}([P](\Phi))\in H^3(G,K^*),$$
determining a well-defined monoid homomorphism
\begin{eqnarray}\label{alpha6}\alpha_6 :\begin{array}{ccc}
 \CK/{\MCK'(1)} \ &\to &H^3(G,K^*_{\phi})\\
  \mathcal{MC}'_K(1)(\Phi) &\mapsto &\tilde{T}_{\phi}(\Phi)
  \end{array},\ \Phi \in \CK.
\end{eqnarray}
\end{lemma}
\begin{proof}
Let $ [M_g]:=\Phi(g) \in \Pic_k(R)$ for every $g \in G$.
Since $\Phi$ is a group homomorphism, there exist (non-unique) $R$-bimodule isomorphisms
$$\Gamma_{g,h} : M_g \otimes_R M_h \to M_{gh},\ \  \forall g,h \in G,$$
which yield the 3-cocycle $T=T^{\Phi}:G^3\to K^*$ satisfying (see \cite[(11)]{CG})
\begin{equation}\label{T(g,h,t)}
\Gamma_{gh,t}(\Gamma_{g,h} \otimes_R M_t) = T_{g,h,t}\cdot\Gamma_{g,ht}(M_g \otimes_R \Gamma_{h,t}),\ \  \forall g,h,t\in G.
\end{equation}
Next, set $r(f,p):=\tau_P(f \otimes_{\End_R(P)} p)\in R$ for $\tau_P$ as in  \eqref{nuP}, and let
$$ \Gamma_{g,h}^P:
\begin{array}{ccc}
P(M_g) \otimes_S P(M_h)&\to & P(M_{gh})\\
(p_1 \otimes_{R} a_g \otimes_R f_1) \otimes_S(p_2 \otimes_{R} a_h \otimes_R f_2) &\mapsto &
p_1 \otimes_{R} \Gamma_{g,h}(a_gr(f_1, p_2)\otimes_R a_h) \otimes_R f_2,
\end{array}
$$
for every $f_1,f_2 \in \check{P}, p_1,p_2 \in P, a_g \in M_g,\ a_h \in M_h$.
It is not hard to verify that for all $g,h \in G$, the maps $ \Gamma_{g,h}^P$ are isomorphisms of $S$-bimodules. Let us inspect their corresponding 3-cocycle. Let
$$x_g = p_1 \otimes_{R} a_g \otimes_{R} f_1\in P(M_g),\ x_h = p_2 \otimes_{R} a_h \otimes_{R} f_2\in P(M_h),\ x_t = p_3 \otimes_{R} a_t \otimes_{R} f_3\in P(M_t),$$
where $a_g \in M_g, a_h \in M_h,a_t \in M_t,$ and $p_i \in P, f_i \in \check{P}$ for $i \in \{1,2,3\}$. Then we have
\begin{align*}
&\Gamma_{gh,t}^P(\Gamma_{g,h}^P(x_g \otimes_S x_h)\otimes_S x_t)\\
=&\ \Gamma_{gh,t}^P((p_1\otimes_{R}\Gamma_{g,h}(a_g\tau_P(f_1 \otimes_{\End_R(P)} p_2) \otimes_R a_h)\otimes_R f_2) \otimes_S (p_3 \otimes_{R} a_t \otimes_{R} f_3))\\
=&\  p_1\otimes_{R} \Gamma_{gh,t}(\Gamma_{g,h}(a_g\tau_P(f_1 \otimes_{\End_R(P)} p_2) \otimes_R a_h)\tau_P(f_2 \otimes_{\End_R(P)} p_3)\otimes_R a_t)\otimes_R f_3 \cdots
\end{align*}
applying \eqref{T(g,h,t)} yields
\begin{align*}
\cdots=& \  p_1\otimes_R T_{g,h,t}\cdot\Gamma_{g,ht}(a_g f_1(p_2) \otimes_R \Gamma_{h,t}(a_hf_2(p_3)\otimes_R a_t))\otimes_R f_3\\
=& \ T_{g,h,t}\cdot(p_1\otimes_{R}\Gamma_{g,ht}(a_g \tau_P(f_1 \otimes_{\End_R(P)} p_2) \otimes_R \Gamma_{h,t}(a_h\tau_P(f_2 \otimes_{\End_R(P)} p_3)\otimes_R a_t))\otimes_R f_3)\\
=&\ T_{g,h,t}\cdot\Gamma_{g,ht}^P(x_g \otimes_S (p_2 \otimes_{R} \Gamma_{h,t}(a_h \tau_P(f_2 \otimes_{\End_R(P)} p_3) \otimes_R a_t) \otimes_R f_3))\\
=&\ T_{g,h,t}\cdot\Gamma_{g,ht}^P(x_g \otimes_S \Gamma_{h,t}^P(x_h \otimes_S x_t)).
\end{align*}
Consequently, $T^{P(\Phi)}=T^{\Phi}\in Z^3(G,K^*)$ and thus $\tilde{T}_{\phi}(P(\Phi))= \tilde{T}_{\phi}(\Phi)$.
\end{proof}

The homomorphism \eqref{chiphi} respects the actions \eqref{actrule} and \eqref{GCC^C} of $\mathcal{MC}'_K(1)$ on the sets $\Clf$ and {$\CK$} respectively as follows.
\begin{lemma}\label{compat}
The map $\tilde{\chi}_\phi: \Clf\to \CK$ is a homomorphism of $\mathcal{MC}'_K(1)$-monoids (see Definition \ref{grpmap}).
As such, it gives rise to a $\mathcal{MC}'_K(1)$-homomorphism
\begin{eqnarray}\label{alpha5}\alpha_{5} :\begin{array}{rcl}
 \Clf)/\mathcal{MC}'_K(1)&\to\ &\CK/\mathcal{MC}'_K(1) \\
\MCK'(1)([(A,\iota)]) &\mapsto\ &\MCK'(1)(\tilde{\chi}_{\phi}([(A,\iota)]))
  \end{array}%,\ [(A,i)]\in \mbox{Cliff}_k(\phi).
\end{eqnarray}
%\begin{equation}\label{barchiphi}\overline{\chi_\phi}: \mbox{Cliff}_k(\phi)/\mathcal{MC}_K(1)\to \CK/\mathcal{MC}_K(1)\end{equation}
of the corresponding quotient $\mathcal{MC}'_K(1)$-monoids.
\end{lemma}
\begin{proof}
%\ChA{For every admissible pair $([P], [(A,\iota)]),$ note that $([P], \chi_\phi([(A,\iota)]))$ is admissible, and since
%$\chi_{\phi}([(A,\iota)])^{[P]} = \chi_{\phi}([(A,\iota)]^{[P]})$, the claim is settled by Lemma \ref{quotmor}.}
The first claim
follows directly from the definitions of the
$\mathcal{MC}'_K(1)$-actions on both monoids, the second claim is
settled by Lemma \ref{quotmor}.
\end{proof}

%Moreover, by Lemma \eqref{CLfPiclosure} it can be easily seen that $\alpha_5$ is a well-defined morphism of monoids.
Thirdly, the exactness of \eqref{S02} yields
%Furthermore, the monoids $ \mbox{Cliff}_k(\phi)/_{\mbox{Mor}}$ and $ \CK/_{\mbox{Mor}}$ are connected to the cohomology group $H^3(G, K^*_{\phi})$ via an exact sequence of monoids as follows.
\begin{lemma}\label{E@C_k(phi)}
With the notations in \eqref{alpha6} and \eqref{alpha5}, the sequence
$$  \Clf/\MCK'(1)\xrightarrow{\alpha_5}\ \CK/\MCK'(1) \xrightarrow{\alpha_6} H^3(G,K^*_{\phi})$$
is exact.
\end{lemma}
%\begin{proof}\end{proof}
The next step is to compose the monoidal homomorphism $\tilde{\Sigma}_{\phi}$ with the quotient homomorphism $\Clf\to\Clf/\MCK'(1)$
(see Lemma \ref{quotmor}) so as to establish a homomorphism of monoids
\begin{eqnarray}\label{alpha4}
\alpha_4 :\begin{array}{ccc}
H^2(G,K^*) &\rightarrow\  &\Clf/\MCK'(1)\\
~[c] &\mapsto &\MCK'(1)([K^{c}_{\phi}G])
\end{array}.
\end{eqnarray}

\begin{lemma}\label{E@Cliff}
With the notations \eqref{alpha5} and \eqref{alpha4}, the sequence
$$H^2(G,K^*) \xrightarrow{\alpha_4}\  \Clf/\MCK'(1) \xrightarrow{\alpha_5}\CK/\MCK'(1) $$ is exact.
\end{lemma}
\begin{proof}
%The fact that Im$(\alpha_4)<\ker(\alpha_5)$ follows from the exactness of \eqref{S1}.
Let $(A_G,\iota)$ be a $G$-graded Clifford system extension of $R$.
The assumption that $\MCK'(1)([{(A_G,\iota)}]) \in \ker(\alpha_5)$ is equivalent to the existence of $[_K P_R] \in \md^*(K\m R)$ such that for every $g \in G$
%\begin{equation}\label{ftp}
$$[P \otimes_{R} A_g \otimes _R \check{P}] =  [ K_{\phi(g^{-1}),1}].$$%\end{equation}
By Lemma \ref{cpch}, this is the same as the existence of such an invertible bimodule with $P(A_G) \overset{G}{\cong} K_{\phi}^fG$ for some $f \in Z^2(G,K^*)$,
or in other words $\MCK'(1)([{(A_G,\iota)}]) \in \text{Im}(\alpha_4)$.
\end{proof}

Unlike the sequence \eqref{S2}, the homomorphism \eqref{alpha4} is not necessarily injective. The failure of injectivity is remedied in
\begin{lemma}\label{E@H2}
With the notations \eqref{alpha3} and \eqref{alpha4}, the sequence
$$\Pic(K)^{\phi} \xrightarrow{\alpha_3} H^2(G,K^*) \xrightarrow{\alpha_4} \Clf/\MCK'(1) $$ is exact.
\end{lemma}
{ In order to prove Lemma \ref{E@H2} we need the following preliminary result.
%\begin{lemma}Let $\alpha_3([P])=[f_P]\in H^2(G,K^*)$, where $[P] \in \Pic(K)^{\phi}$.
%Then with the notation \eqref{nuP}, \eqref{Delta_P},\eqref{A^C}, $K_{\phi}^{f_P}G\stackrel{G}{\cong}(K_{\phi}G)^{C}$, where $C:=(K,K,P,\check{P},\Delta_P,\nu_P)$.\end{lemma}}
\begin{lemma}\label{lem7.6}
Let $\alpha_3([P])=[c_P]\in H^2(G,K^*)$, where $[P] \in \Pic(K)^{\phi}$.
Then % with the notation \eqref{actrule},
\begin{equation}\label{cleqcl}
[P]([(K_{\phi}^{c_P}G)= \oplus_{g\in G} Kv_g,\iota_1])=[(K_{\phi}G= \oplus_{g\in G} Ku_g,\iota_2)],
\end{equation} where
$\iota_1:x\mapsto x\cdot v_e$ and $\iota_2:x\mapsto x\cdot u_e$ for every $x\in K$. In particular, $[K_{\phi}^{c_P}G]\in\Pic(K)([K_{\phi}G])$.
\end{lemma}}
\begin{proof}
Let \eqref{psig} be the collection of bijections determining the 2-cocycle $c_P$ as an image of $\alpha_3$, that is (see \eqref{fP})
\begin{equation}\label{f_P}
c_P(g,h)\star p =  \tilde{\psi}_{gh}\tilde{\psi}_h^{-1}\tilde{\psi}_g^{-1}(p),\ \ \forall g,h \in G, p\in P.
\end{equation}
 Using these bijections and the map $\Delta_P :=\theta_P^{-1}\circ\mu_P:P\otimes_1 \check{P} \xrightarrow{\cong}K$ (see \eqref{Delta_P}), consider the map
% \begin{equation}\label{fipi}
%\sum_i f_i(p)p_i = p,\ \mbox{for\ every}\ p \in P,
%\end{equation}
\begin{eqnarray}\label{Gamma}\Gamma_P : \begin{array}{ccc}
P(K_{\phi}^{c_P}G) &\rightarrow & K_{\phi}G\\
p \otimes_1 a_gv_g \otimes_1 f &\mapsto &{g}(\Delta_P(\tilde{\psi}_g^{-1}(p) \otimes_1f))a_gu_g,
\end{array}
\end{eqnarray}for $p \in P, f \in \check{P},$ and $a_g \in K.$
Now by the multiplication rule in $K_{\phi}G$, for $p_1,p_2 \in P, f_1 ,f_2 \in \check{P},$ and $a_g,a_h \in K$
\begin{align*}
&\Gamma_P(p_1\otimes_1 a_gv_g \otimes_1 f_1)\cdot\Gamma_P(p_2\otimes_1 a_hv_h \otimes_1 f_2)\\
=&\ {g}(\Delta_P(\tilde{\psi}_g^{-1}(p_1) \otimes_1 f_1))a_gu_g\cdot {h}(\Delta_P(\tilde{\psi}_h^{-1}(p_2) \otimes_1 f_2))a_hu_h\\
=&\ a_g\cdot g(a_h)\cdot {g}(\Delta_P(\tilde{\psi}_g^{-1}(p_1) \otimes_1 f_1))\cdot {gh}(\Delta_P(\tilde{\psi}_h^{-1}(p_2) \otimes_1 f_2))u_{gh}\cdots
\end{align*}
%now since $\Delta_P(\tilde{\psi}_g^{-1}(p_i) \otimes_K f_j)u_h = u_h(^{h^{-1}}\Delta_P(\tilde{\psi}_g^{-1}(p_i) \otimes_K f_j))$, we have
%\begin{align*}
Since $\Delta_P$ is a $K$-bimodule isomorphism and further by \eqref{tildepsi}, we obtain
\begin{align*}
\cdots =&\ a_g\cdot g(a_h)\cdot {gh}(\Delta_P(h^{-1}(\Delta_P(\tilde{\psi}_g^{-1}(p_1) \otimes_1 f_1))\star\tilde{\psi}_h^{-1}(p_2) \otimes_1 f_2))u_{gh}\\
=&\ a_g\cdot g(a_h)\cdot {gh}(\Delta_P(\tilde{\psi}_h^{-1}(\Delta_P(\tilde{\psi}_g^{-1}(p_1) \otimes_1 f_1)\star p_2) \otimes_1 f_2))u_{gh}\cdots
\end{align*}
Now applying \eqref{DeltaTau} and then using \eqref{tildepsi} yields
\begin{align*}
\cdots=&\  a_g\cdot g(a_h)\cdot {gh}(\Delta_P(\tilde{\psi}_h^{-1}(\tilde{\psi}_g^{-1}(p_1) \cdot f_1( p_2)) \otimes_1 f_2))u_{gh}\\
=&\ a_g\cdot g(a_h)\cdot g(f_1( p_2))\cdot {gh}(\Delta_P(\tilde{\psi}_h^{-1}(\tilde{\psi}_g^{-1}(p_1)) \otimes_1 f_2))u_{gh}\cdots
\end{align*}
By \eqref{f_P} and then again by \eqref{tildepsi} we have
\begin{align*}
\cdots=&\  a_g\cdot g(a_h)\cdot g(f_1( p_2))\cdot {gh}(\Delta_P(\tilde{\psi}_{gh}^{-1}(c_P(g,h)\star p_1) \otimes_1 f_2))u_{gh}\\
=&\  a_g\cdot g(a_h)\cdot g(f_1( p_2))\cdot c_P(g,h)\cdot {gh}(\Delta_P(\tilde{\psi}_{gh}^{-1}(p_1) \otimes_1 f_2))u_{gh}\cdots
\end{align*}
Finally, by the definition of $\Gamma_P$ \eqref{Gamma}, and further using the multiplication rule \eqref{mult} in $P(K_{\phi}^{c_P}G)$, we obtain
\begin{align*}
\cdots=&\ \Gamma_P(p_1 \otimes_1 a_g\cdot g(a_h) \cdot g(f_1(p_2)) \cdot c_P(g,h)v_{gh} \otimes_1 f_2)\\
=&\ \Gamma_P((p_1 \otimes_1 a_gv_g \otimes_1 f_1)\cdot(p_2 \otimes_k a_hv_h \otimes_1 f_2)).
\end{align*}
Therefore, \eqref{Gamma} is a $G$-graded homomorphism.
By Lemma \ref{grhom} it is left to show that \eqref{Gamma} respects the embeddings $\iota_1$ and $\iota_2$.
Indeed, by choosing a normalized representative $c_P$ in its cohomology class, one can choose $\tilde{\psi}_e$=Id$_P$ in \eqref{psig}. Then,
for every $x \in K$, by using the definition of $P(\iota_1)$ (see \eqref{A^C}),
we have $$\Gamma_P\circ P(\iota_1)(x)=\Gamma_P(\sum\limits_{j}x\star p_j \otimes_1 u_e \otimes_1 p_j')= \sum\limits_j \Delta_P(x\star p_j\otimes_1 p_j')u_e = xu_e ,$$
and \eqref{cleqcl} holds true.
\end{proof}
\noindent{\em Proof of Lemma \ref{E@H2}.}
The fact that Im$(\alpha_3)<\ker(\alpha_4)$ is a direct consequence of Lemma \ref{lem7.6}.
Conversely, let $[c]\in\ker(\alpha_4)$, that is there exists
$[P] \in \Pic(K)^{}$ such that $[P]([K_{\phi}^{c}G])= [(K_{\phi}G)]$.
%a Morita context $C:=(K,K,P,\check{P},\mu,\tau)$ such that $K_{\phi}^{f}G\stackrel{G}{\cong}(K_{\phi}G)^{C}$.
By Lemma \ref{P_inv} we deduce that $[P] \in \Pic(K)^{\phi}$, and hence by Lemma \ref{lem7.6} we obtain that $[K_{\phi}^{c}G]=[K_{\phi}^{c_P}G]$,
where $[c_P]=\alpha_3([P])$.
Then $[c]=[c_P]\in H^2(G,K^*),$ and so $[c]\in$Im$(\alpha_3)$.\ \ \ \ \ \ \ \ \ \ \ \ \ \ \ \ \ \ \ \ \  \ \ \ \ \ \ \ \ \ \ \ \ \ \ \ \ \ \ \ \ \ \ \ \ \ \ \ \ \ \ \ \ \ \ $\Box$
%\begin{flushright}$\Box$\end{flushright}

\subsection{An interpretation of $H^1(G,K^*_{\phi})$}

Let $K_{\phi}G$ be the skew group algebra corresponding to a given group action $\phi \in$ Hom$(G, \Aut_k(K))$.
We consider left $K_{\phi}G$-modules as $K_{\phi}G\m K$-bimodules, whose left and right $K$-module structures are the same, and as such can be tensored over $K$.
The comultiplication in $K_{\phi}G$ furnishes the tensor product $M\otimes_K\tilde{M}$ of any left $K_{\phi}G$-modules $M$ and $\tilde{M}$ with the following left $K_{\phi}G$-module structure
\begin{equation}\label{tensrepr}
\sum_{g\in G}x_gu_g\star(m \otimes_K \tilde{m}) := \sum_{g\in G}x_g(u_g\star m \otimes_K u_g\star\tilde{m}),\ \  \forall m \in M, \tilde{m} \in \tilde{M}.\end{equation}
Then the classes in $\md (K_{\phi}G)$ form a monoid under the product
$$[M]\otimes_K[\tilde{M}]:=[M\otimes_K\tilde{M}],\ \ [M],[\tilde{M}]\in\md (K_{\phi}G).$$
The identity of $\md (K_{\phi}G)$ is the class $[K_{\phi}]$ of the free right $K$-module of rank 1, endowed with a left $K_{\phi}G$-module structure on via
\begin{equation}\label{KPhi}
xu_g\star y :=x \phi(g)(y),\ \  \forall g \in G,  x\in K,y \in K_{\phi}.
\end{equation}
Next, the well-defined restriction map
\begin{equation}\label{resmap}\res:\md (K_{\phi}G) \rightarrow \md (K)\end{equation}
between the isomorphism classes of left $K_{\phi}G$-modules and those of $K$-modules (with the monoid structure determined by $``\otimes_K"$) respectively is a homomorphism of monoids. We claim

\begin{lemma}\label{ker1} $\ker(\res) \cong H^1(G,K_{\phi}^*)$.\end{lemma}
\begin{proof}
Let $M$ be a left $K_{\phi}G$-module such that $\res([M])=[K]$.
Then there exists an isomorphism $\varphi:K\xrightarrow{\cong} M$ of $K$-modules. For every $g\in G$ let
$$\varphi_g:=\varphi^{-1}(u_g\star\varphi(1))\in K.$$

\textit{Step 1.} For every $g\in G$, the element $\varphi_g$ is invertible. Moreover,\begin{eqnarray}\label{cocy}\pi_{{\varphi}}=\pi_{M,\varphi}:\begin{array}{ccc}
G&\to& K^*\\
g&\mapsto &\varphi_g
\end{array}\end{eqnarray}
is a 1-cocycle, where the $G$-module structure of $K^*$ is via $\phi$.

\textit{Proof of Step 1.}
For every $g,h\in G$
\begin{eqnarray}\label{1cc}\begin{array}{rl}
\varphi_{gh}&=\varphi^{-1}(u_{gh}\star\varphi(1))=\varphi^{-1}(u_gu_h\star\varphi(1))=\varphi^{-1}(u_g\star(u_h\star\varphi(1)))\\
&=\varphi^{-1}(u_g\star\varphi(\varphi_h))=\varphi^{-1}(u_g\star(\varphi_h\cdot\varphi(1)))=\varphi^{-1}(u_g\cdot \varphi_h\star(\varphi(1)))\\
&=\varphi^{-1}(g(\varphi_h)\cdot u_g\star(\varphi(1)))=g(\varphi_h)\cdot \varphi^{-1}(u_g\star(\varphi(1)))=g(\varphi_h)\cdot \varphi_g.
\end{array}\end{eqnarray}
Since $\varphi_e=1$, putting $h=g^{-1}$ in \eqref{1cc} we deduce that $\varphi_g\in K^*$ for every $g\in G$.
In fact, \eqref{1cc} says that $g\mapsto \varphi_g$
%\begin{eqnarray}\label{cocy}\pi_{{\varphi}}=\pi_{M,\varphi}:\begin{array}{ccc}G&\to& K^*\\g&\mapsto &\varphi_g\end{array}\end{eqnarray}
is a 1-cocycle.

\textit{Step 2.} With the notation \eqref{cocy}, the rule $[M]\mapsto [\pi_{M,\varphi}]$ determines a well-defined map $\ker(\res)\to H^1(G,K^*_{\phi})$.

\textit{Proof of Step 2.} We show that the cohomology class
$[\pi_{M,{\varphi}}]\in H^1(G,K^*_{\phi})$ neither depends on the representative $M\in [M]$, nor on the isomorphism $\varphi$.
Indeed, let $M$ and $\tilde{M}$ be left $K_{\phi}G$-modules via the actions $``\star"$ and $``\tilde{\star}"$ respectively, such that $[M]=[\tilde{M}] \in \ker(\res)$,
and let $\lambda:M \stackrel{\cong}\rightarrow \tilde{M}$ be a $K_{\phi}G$-module isomorphism.
In particular, for every $g\in G$ and $m\in M$
\begin{equation}\label{lamb}
\lambda(u_g\star m)=u_g\tilde{\star}\lambda(m).
\end{equation}
%How does the 1-cocycle \eqref{cocy} depend on the choice of the $K$-module isomorphism $\varphi$?
Let
\begin{equation}\label{varphitilde}
{\varphi}:K\xrightarrow{\cong} M,\ \ \ \tilde{\varphi}:K\xrightarrow{\cong} \tilde{M}\end{equation}
be $K$-module isomorphisms determining the corresponding 1-cocycles
$\pi_{M,{\varphi}}$ and $\pi_{\tilde{M},\tilde{\varphi}}$ as in \eqref{cocy}.
We claim that these 1-cocycles defer by the 1-coboundary
$$g\mapsto g(x)\cdot x^{-1},\text{  where  } x:=\varphi^{-1}\circ\lambda^{-1}\circ\tilde{\varphi}(1)\in K^*.$$
First, since $\tilde{\varphi}^{-1}\circ\lambda\circ \varphi$ is a $K$-morphism we have
$$x\cdot(\tilde{\varphi}^{-1}\circ\lambda\circ \varphi( 1))=\tilde{\varphi}^{-1}\circ\lambda\circ \varphi( x\cdot 1)=\tilde{\varphi}^{-1}\circ\lambda\circ \varphi( x)
=\tilde{\varphi}^{-1}\circ\lambda\circ \varphi(\varphi^{-1}\circ\lambda^{-1}\circ\tilde{\varphi}(1))=1,$$
hence $\tilde{\varphi}^{-1}\circ\lambda\circ \varphi( 1)=x^{-1}.$
Next, for every $g\in G$
$$\begin{array}{rl}
\tilde{\varphi}_g&=\tilde{\varphi}^{-1}(u_g\tilde{\star}\tilde{\varphi}(1))=\tilde{\varphi}^{-1}(u_g\tilde{\star}\lambda\circ\varphi\circ\varphi^{-1}\circ\lambda^{-1}\circ\tilde{\varphi}(1))\\
&=\tilde{\varphi}^{-1}(u_g\tilde{\star}\lambda\circ\varphi((\varphi^{-1}\circ\lambda^{-1}\circ\tilde{\varphi}(1))\cdot 1))
=\tilde{\varphi}^{-1}(u_g\tilde{\star}\lambda(\varphi(x\cdot 1)))=\cdots
\end{array}$$
The element $x\in K$ can be pulled out of the $K$-morphism $\varphi$, thus
$$\begin{array}{rl}
\cdots&=\tilde{\varphi}^{-1}(u_g\tilde{\star}\lambda(x\cdot\varphi( 1)))=\tilde{\varphi}^{-1}(u_g\cdot x\tilde{\star}\lambda(\varphi( 1)))
%=\tilde{\varphi}^{-1}(u_g\cdot x*\varphi( 1))
=\tilde{\varphi}^{-1}(g(x)\cdot u_g\tilde{\star}\lambda(\varphi( 1)))=\cdots
\end{array}$$
By \eqref{lamb}
$$\begin{array}{rl}
\cdots&=\tilde{\varphi}^{-1}(g(x)\cdot \lambda(u_g\star\varphi( 1)))=g(x)\cdot\tilde{\varphi}^{-1}( \lambda(u_g\star\varphi( 1)))=g(x)\cdot\tilde{\varphi}^{-1}( \lambda(\varphi( \varphi_g)))\\
&=g(x)\cdot\tilde{\varphi}^{-1}( \lambda(\varphi( \varphi_g\cdot 1)))=\cdots
\end{array}$$
Again, pull $\varphi_g\in K$ out of the $K$-morphism $\tilde{\varphi}^{-1}\circ\lambda\circ \varphi$ and apply the above description of $x^{-1}$ to obtain
$$\cdots=g(x)\cdot\varphi_g\cdot(\tilde{\varphi}^{-1}\circ \lambda\circ \varphi( 1))=g(x)\cdot\varphi_g\cdot x^{-1}.$$
We conclude that $\tilde{\varphi}_g\cdot\varphi_g^{-1}=g(x)\cdot x^{-1}$
and hence we are done.

\textit{Step 3.}
The well-defined map
\begin{equation}\label{Pi-1}\begin{array}{ccc}\ker(\res)&\to& H^1(G,K_{\phi}^*)\\~[M]&\mapsto &[\pi_{M,\varphi}]\end{array}
\end{equation} is a group homomorphism.

\textit{Proof of Step 3.}
Let $M$ and $\tilde{M}$ be (not necessarily isomorphic) left $K_{\phi}G$-modules such that $[M],[\tilde{M}] \in \ker(\res)$.
Let \eqref{varphitilde} be $K$-module isomorphisms determining the corresponding 1-cocycles
$\pi_{M,{\varphi}}$ and $\pi_{\tilde{M},\tilde{\varphi}}$ as in \eqref{cocy}.
Let us compute the 1-cocycle
$$\pi_{M\otimes_K\tilde{M},\hat{\varphi}}:G\to K^*$$
with respect to the $K$-module isomorphism
\begin{eqnarray}\label{cocytens}\hat{\varphi}:\begin{array}{ccc}
K&\to& M\otimes_K\tilde{M}\\
x&\mapsto &\varphi(1)\otimes_K\tilde{\varphi}(x)
\end{array}.\end{eqnarray}
By the rule in step 1 (see \eqref{cocy}) applied on \eqref{cocytens}, and the $K_{\phi}G$-action \eqref{tensrepr} on the tensor product, we have for every $g\in G$
$$\begin{array}{rl}
\pi_{M\otimes_K\tilde{M},\hat{\varphi}}(g)&=\hat{\varphi}_g=\hat{\varphi}^{-1}(u_g(\hat{\varphi}(1)))=\hat{\varphi}^{-1}(u_g(\varphi(1)\otimes_K\tilde{\varphi}(1)))\\
&=\hat{\varphi}^{-1}(u_g*\varphi(1)\otimes_Ku_g\tilde{*}\tilde{\varphi}(1))=\hat{\varphi}^{-1}(\varphi(\varphi_g)\otimes_K \tilde{\varphi}(\tilde{\varphi}_g))\\
&=\hat{\varphi}^{-1}(\varphi(\varphi_g\cdot 1)\otimes_K \tilde{\varphi}(\tilde{\varphi}_g))=\varphi_g\cdot\hat{\varphi}^{-1}(\varphi( 1)\otimes_K \tilde{\varphi}(\tilde{\varphi}_g))\\
&=\varphi_g\cdot\tilde{\varphi}_g=\pi_{M,{\varphi}}(g)\cdot\pi_{\tilde{M},\tilde{\varphi}}(g).
\end{array}$$
In particular, the $H^1(G,K^*_{\phi})$-classes $[\pi_{M\otimes_K\tilde{M},\hat{\varphi}}]$ and $[\pi_{M,{\varphi}}]\cdot[\pi_{\tilde{M},\tilde{\varphi}}]$ are equal, proving the homomorphism property.

\textit{Step 4. The map \eqref{Pi-1} is an isomorphism.}

\textit{Proof of Step 4.}
Here is a construction of the inverse of the group homomorphism \eqref{Pi-1}.
We endow the trivial element in Mod$(K)$, that is free $K$-module of rank 1, with a left $K_{\phi}G$-action for
a given 1-cocycle $\pi:G\to K^*_{\phi}$ as follows.
%, let $K_{\pi}$ be the free $K$-module of rank 1 with the $G$-action $\phi$. In addition, endow it with the following left $K_{\phi}G$-action
\begin{equation}\label{act*}(\sum_{g\in G}x_gu_g)\star x:=\sum_{g\in G}x_g\cdot\pi(g)\cdot g(x),\ \ \ \ \sum_{g\in G}x_gu_g\in K_{\phi}G,\ x\in K_{\pi}.\end{equation}
%for every $\sum_{g\in G}x_gu_g\in K_{\phi}G$ and $x\in K_{\pi}.$
Then it is not hard to verify, using the 1-cocycle condition, that \eqref{act*} indeed describes a left $K_{\phi}G$-module on $K$.
Denote this module by $K_{\pi}$.

Furthermore, if $[\pi]=[\hat{\pi}]\in H^1(G,K^*_{\phi})$ then the corresponding $K_{\phi}G$-modules $K_{\pi}$ and $K_{\hat{\pi}}$ are isomorphic as follows.
The fact that the 1-cocycles are cohomologically equivalent says that there exists $y\in K^*$ such that
$$\pi(g)=y^{-1}\cdot g(y)\cdot \hat{\pi}(g),\ \ \ \forall g\in G.$$
Then it is easily checked that $\begin{array}{ccc}K_{\pi}&\to &K_{\hat{\pi}}\\ x&\mapsto& y\cdot x\end{array}$ is a $K_{\phi}G$-module isomorphism.
In other words, there is a well-defined map
\begin{equation}\label{piphi}
H^1(G,K_{\phi}^*)\to \ker(\res).\end{equation}
To conclude the proof, the fact that \eqref{piphi} is indeed the inverse of the group homomorphism \eqref{Pi-1} is left to the reader.
\end{proof}
With the notations \eqref{PKinv} and \eqref{resmap} we have
\begin{lemma}\label{resext} Im$(\res) \subseteq \md(K)^{\phi}$.\end{lemma}
\begin{proof}
Let $M$ be a left $K_{\phi}G$-module.
To show that $\res([M])$ is $G$-invariant we need to present a $K$-module isomorphism between $$g(M)=_{\phi(g^{-1})}K_1 \otimes_K M \otimes_K \ _{\phi(g)}K_1$$ (see \eqref{Gactmd})
and $M$ for every $g\in G$.
%Observe that, for $x \in K$, we have $x.m = m.x$, for every $m \in M$. Now c
Indeed, let
\begin{eqnarray}\label{psigdef}\psi_g :\begin{array}{ccc}g(M) &\rightarrow& M\\
x \otimes_K m \otimes_K y& \mapsto& g(x)\cdot y\cdot u_g\star m
\end{array},\ \ x,y \in K,m \in M.\end{eqnarray}
We first check that $\psi_g$ agrees with the different actions of $K$ on the three components of $g(M)$. Note that any $z \in K$ and $x \otimes_K m \otimes_K y$ in the module $g(M)$ satisfy
$$\begin{array}{rl}
z(x \otimes_K m \otimes_K y)&=g^{-1}(z)\cdot x \otimes_K m \otimes_K y\\ &= x\otimes_K\ {g^{-1}}(z)m \otimes_K y = x\otimes_K m\otimes_K {g^{-1}}(z) (y)\\&=x\otimes_K m \otimes_K z\cdot y.
\end{array}$$
Applying $\psi_g$ on each of the above yields
\begin{equation}\label{difact}z\cdot g(x)\cdot y\cdot u_g\star m,\end{equation} and so $\psi_g$ is well-defined. Moreover, since \eqref{difact} is equal to $z(\psi_g(x \otimes_K m \otimes_K y)),$
%$$\begin{array}{c}=\psi_g(g^{-1}(z)\cdot x \otimes_K m \otimes_K y) = z\cdot g(x)\cdot y\cdot u_g*m \\
%= \psi_g(x\otimes_K\ {g^{-1}}(z)m \otimes_K y) = \psi_g(x\otimes_K m\otimes_K {g^{-1}}(z) (y))=\psi_g(x\otimes_K m \otimes_K z\cdot y).\end{array}$$
it follows that $\psi_g$ is a homomorphism of $K$-modules.
To show that $\psi_g$ is invertible let
$$\psi'_g :\begin{array}{ccc}M &\rightarrow& g(M)\\
m&\mapsto& 1 \otimes_K\ u_g^{-1}\star m \otimes_K 1\end{array}.$$
Then, for every $g\in G, m\in M$ and $x,y\in K$ we have
$$\psi'_g(\psi_g(x \otimes_K m \otimes_K y)) = \psi'_g( g(x)\cdot y\cdot u_g\star m) = 1 \otimes_K x\cdot g^{-1}(y)m \otimes_K 1 = x \otimes_K m \otimes_K y.$$
On the other hand
$$\psi_g(\psi'_g(m))=\psi_g(1 \otimes_K\ u_g^{-1}\star m \otimes_K 1)=m.$$
Therefore, $\psi'_g=(\psi_g)^{-1}$, and so $g(M) \cong M$, proving the claim.
\end{proof}
Next, under the notation \eqref{resmap}, set
\begin{equation}\label{extP}\ExtP := \res^{-1}(\text{Pic}(K))<\md (K_{\phi}G).\end{equation}
This is the subgroup of isomorphism classes of left $K_{\phi}G$-modules, which restrict to invertible modules over $K$.
With the notation \eqref{extP} we can view the group isomorphism \eqref{piphi} as an embedding
\begin{equation}\label{alpha1}
\alpha_1:H^1(G,K_{\phi}^*)\hookrightarrow \ExtP,
\end{equation}
and the restriction \eqref{resmap}, on the subdomain \eqref{extP}, as a homomorphism
\begin{equation}\label{alpha2}\alpha_2:\ExtP\xrightarrow{}\text{Pic}(K)^{}.
\end{equation}
By Lemmas \ref{ker1} and \ref{resext}, we conclude

\begin{lemma}\label{piphres}
The sequence
$$1\to H^1(G,K^*_{\phi})\xrightarrow{\alpha_1}\ExtP\xrightarrow{\alpha_2}\text{Pic}(K)^{\phi}$$
is exact.
\end{lemma}
The homomorphisms $\alpha_2$ and $\alpha_3$ are tied up as follows.

\begin{lemma}\label{exactPKG}
With the notations \eqref{alpha3} and \eqref{alpha2}, the sequence
$$\ExtP\xrightarrow{\alpha_2}\text{Pic}(K)^{\phi}\xrightarrow{\alpha_3} H^2(G,K^*_{\phi})$$
is exact.
\end{lemma}
\begin{proof}
Let $[P]\in$Im$(\alpha_2)=$Im$(\res)\cap$Pic$(K)$. Then we regard $[P]$ as a left $K_{\phi}G$-module. Apply the set $B_P$ of left $K$-module isomorphisms \eqref{BP} putting $P$ for $M$
(as proven in Lemma \ref{resext} these isomorphisms realize the invariance of $[P]$).
With the notation of \S\ref{alpha3sect}, the corresponding set of maps $\{\tilde{\psi}_g : P \rightarrow P\}_{g\in G}$ are given by
$$\tilde{\psi}_g(p) :=u_g*p,\ \  g\in G, p\in P$$ and satisfy \eqref{tildepsi}.
By the definition \eqref{alpha3} $$\alpha_3([P]) = [c_{B_P}]\in H^2(G,K^*),$$ where $c_{B_P}(g,h) = \tilde{\psi}_{gh}\circ\tilde{\psi}_{h}^{-1}\circ\tilde{\psi}_{g}^{-1}$.
But for every $g,h\in G$ and $p\in P$
$$\tilde{\psi}_{gh}(\tilde{\psi}_{h}^{-1}(\tilde{\psi}_{g}^{-1}(p)))=u_{gh}\cdot u_h^{-1}\cdot u_g^{-1}*p=p,$$
and so $c_{B_P}(g,h) =1$. Therefore, $[P] \in \ker(\alpha_3)$.

On the other hand, let $[P] \in \ker(\alpha_3)$. Since $P$ is $G$-invariant there exists a set $\{\tilde{\psi}_g\}_{g\in G}: P \rightarrow P$ of bijections satisfying \eqref{tildepsi}.
%$\psi_g(xp)=\ ^gx\psi_g(p), \forall x \in K, p\in P$,
Moreover, since $\alpha_3([P])$ is cohomologically trivial, there exist elements $\{\lambda_{g}\}_{g\in G} \in K^*$ such that
\begin{equation}\label{tildemult}
\tilde{\psi}_{gh}\circ\tilde{\psi}_h^{-1}\circ\tilde{\psi}_g^{-1} = \lambda_{gh}\cdot\lambda_{g}^{-1}\cdot g(\lambda_{h}^{-1}),\ \  \forall g,h \in G.\end{equation}
Then $P$ can be furnished with a left $K_{\phi}G$-module structure by defining
$$u_g*p := \lambda_g^{-1}\cdot\tilde{\psi_g}(p),\ \ \forall g \in G, p \in P.$$
Indeed, by \eqref{tildemult} we deduce that for every $g,h \in G$ and $p \in P$
$$u_{gh}*p =\lambda_{gh}^{-1}\cdot\tilde{\psi}_{gh}(p) =
% \lambda_{g}^{-1}\cdot(^g\lambda_h^{-1})\cdot\tilde{\psi}_g\cdot\tilde{\psi}_h\cdot\tilde{\psi}_{gh}^{-1}\cdot\tilde{\psi}_{gh}(p) =
\lambda_g^{-1}\cdot g(\lambda_h^{-1})\cdot\tilde{\psi}_g\circ\tilde{\psi}_h(p)=\lambda_g^{-1}\cdot\tilde{\psi}_g(\lambda_h^{-1}\tilde{\psi}_h(p))=u_g*(u_h*p).$$
Therefore, $[P]$ belongs to the image of \eqref{resmap}.
\end{proof}

%Now connecting with \eqref{alpha4} yields;

lemmas \ref{E@C_k(phi)}, \ref{E@Cliff}, \ref{E@H2}, \ref{piphres} and \ref{exactPKG} give rise to the seven term exact sequence of monoids
\begin{eqnarray}\label{GS}
\begin{array}{rl}
1&\to H^1(G,K^*_{\phi})\xrightarrow{\alpha_1}\ExtP\xrightarrow{\alpha_2}\text{Pic}(K)^{\phi}\xrightarrow{\alpha_3} H^2(G,K^*_{\phi}) \\ &\xrightarrow{\alpha_4}
\Clf/{\MCK'(1)}\xrightarrow{\alpha_5}\ \CK/{{\MCK'(1)}} \xrightarrow{\alpha_6} H^3(G,K^*_{\phi}),
\end{array}
\end{eqnarray}
completing the proof of Theorem A.

\section{The Commutative Case: $R = K$}\label{Sec6}
This section interprets the seven term sequence \eqref{main} in the case $C(K) := \{[K]\}$ (see \S\ref{secCK}(iii)).
To retrieve the sequence \eqref{7K} from \eqref{main}, we prove an isomorphism $\CK/{{\MCK}'(1)}\cong H^{1}(G,\Pic(K)_{\phi})$ (Lemma \ref{lambdalem}),
and show that all the monoids in \eqref{7K} are actually groups (Lemma \ref{R=K}).

Note that in this case, the groupoid $\MCK'(1)$ has single object $K$, and hom$(\MCK'(1))$ is a group identified with $\Pic(K)$.
From this point onwards we return to the conventional notation $\otimes_K$ rather than $\otimes_1$.
With the notation of \eqref{Ckp}, \eqref{CkPhiK}, \eqref{ClfPhi} and \eqref{ClfkPhiK} we have for every $\phi\in\Hom(G,\Aut_k(K))$
\begin{equation}\label{twomon}\tilde{\mbox{Cliff}}(\phi)=\tilde{\mbox{Cliff}}(\phi,K),\text{ and }  \CK =\CPK=\Hom_{\phi}(G,\Pic_k(K)).\end{equation}
The two monoids \eqref{twomon} are acted on by $\MCK'(1)$ via \eqref{actrule} and \eqref{GCC^C} respectively.
Its group structure, identified with Pic$(K)$ in this chapter (see Remark \ref{Kcase}), puts us in the setup of Lemma \ref{isoqot}.

Keeping in mind that
$$1_{\ClfK}=[K_{\phi}G],\text{ and } 1_{\CPK}=1_{\phi}$$
%$$1_{\text{Cliff}_k(\phi)}=[K_{\phi}G]=[(K_{\phi}G,(x\mapsto xu_e))],\text{ and } 1_{\CPK}=[1_{\phi}]_{\e}$$
we establish two isomorphisms
%\begin{equation}\label{Cliffmod} \text{Cliff}_k(\phi,K)/_{\mbox{Pic$(K)$}}\cong\ \text{Cliff}_k(\phi,K)/_{\mbox{Pic$(K)([K_{\phi}G])$}} \end{equation}
\begin{eqnarray}\label{Cliffmod}\begin{array}{ccc}
\ClfK/_{\mbox{Pic$(K)$}}&\xrightarrow{\cong}&\ \ClfK/_{\mbox{Pic$(K)([K_{\phi}G])$}}\\
\Pic(K)([A_G])&\mapsto &[A_G]\otimes_{\phi}\Pic(K)([K_{\phi}G])
\end{array},\end{eqnarray}
and
\begin{eqnarray}\label{Ckmod}\begin{array}{ccc}
\CPK/_{\mbox{Pic$(K)$}}&\xrightarrow{\cong}&\ \CPK/_{\mbox{Pic$(K)(1_{\phi})$}}\\
\Pic(K)(\Phi)&\mapsto &\Phi\otimes_{\phi}\Pic(K)(1_{\phi})
\end{array}.\end{eqnarray}

Next, recall \cite[Prop. II.5.4]{Bass} that $h_K:\p(K)\to\au$ (\eqref{hR0} for $R=K$) is surjective and splits, that is
$$\Pic_k(K)\cong\Pic(K)\rtimes\Aut_k(K),$$
and plug in \eqref{lambg}
$$\phi_0:=\phi,\ A:=\Pic(K)\text{ and }\Gamma:=\Aut_k(K).$$
Using the notation \eqref{homphi}, we obtain a group isomorphism
%\begin{equation}\label{idenCH}
$$\CPK\cong\Hom_{\phi}(G,\Pic(K)\rtimes \Aut_k(K)),$$
%\end{equation}
which, together with \eqref{lambg} yields a well-defined isomorphism
\begin{eqnarray}\label{lamb'g}
%\lambda'_{\phi}:
\begin{array}{ccc}
\CPK&\xrightarrow{\cong} &Z^1(G,\Pic(K)_{\phi})\\
 ~\Phi &\mapsto&\Phi\phi^{-1}
\end{array}.\end{eqnarray}
%For every $\Phi\in$Hom$(G,$Pic$(K))$ let $f_{\Phi}:\begin{array}{rcl} G&\to &\text{Pic}(K)\\ g&\mapsto&\Phi(g) \otimes_K\  [_{\phi(g)}K_1] \end{array}.$ We have
\begin{lemma}\label{lambdalem}
Let $\phi \in$ Hom$_k(G,$ Aut$_k(K))$. Then under the isomorphism \eqref{lamb'g}, the subgroup $\Pic(K)(1_{\phi})<\CPK$ is mapped onto the coboundary subgroup
$B^1(G,\Pic(K)_{\phi})$ of $Z^1(G,\Pic(K)_{\phi})$.
Consequently, there is a well-defined isomorphism
\begin{eqnarray}\label{lambarg}
\bar{\lambda}_{\phi}:\begin{array}{ccc}
\CPK/_{\mbox{Pic($K$)}}&\xrightarrow{\cong} &H^1(G,\Pic(K)_{\phi})\\
\Pic(K)(\Phi) &\mapsto&[\Phi\phi^{-1}]
\end{array}.\end{eqnarray}

\end{lemma}
\begin{proof}
An element in the orbit $\Pic(K)([1_{\phi}]_{})$ of $[1_{\phi}]_{}$ under the $\Pic(K)$-action is of the form (see \eqref{M^C} and \eqref{GCC^C})
\begin{equation}\label{P1phi}
[P](1_{\phi})=[P \otimes_{K} \ K_{1,\phi} \otimes_K \check{P}]
\end{equation}
for some $[P]\in\Pic(K).$

On the other hand, a map $G\to\Pic(K)$ is a 1-coboundary if and only if there exists $[P]\in\Pic(K)$ such that
\begin{equation}\label{cobou}
g\mapsto[P]\otimes_Kg([P]^{-1})=[P \otimes_K \ K_{1,\phi(g)} \otimes_K \check{P} \otimes_K \ K_{1,\phi(g^{-1})}]
\end{equation}
for every $g\in G$.
The fact that \eqref{P1phi} and \eqref{cobou} correspond under \eqref{lamb'g} settles the first part of the claim, which together with \eqref{Ckmod}
prove the isomorphism \eqref{lambarg}.
\end{proof}
%\begin{lemma}Let $\phi \in$ Hom$_k(G,$ Aut$_k(K))$. Then with the notation \eqref{lambarg}
%$$``\theta_6"=\bar{T}_{\phi}\circ\bar{\lambda}_{\phi}^{-1}:H^1(G,\Pic(K)_{\phi})\stackrel{}\to H^3(G,K^*_{\phi}).$$\end{lemma}\begin{proof}\ChY{Perhaps move to the next section.}\end{proof}
From the sequence \eqref{GS}, the isomorphism \eqref{lambarg}, we establish the following long exact sequence of monoids.
%defined in \eqref{E@C_k(phi)}, \eqref{E@Cliff}, \eqref{E@H2}, and \eqref{exactPKG}

\begin{eqnarray}\label{subexact}
\begin{array}{rl}
1&\to H^1(G,K^*_{\phi})\xrightarrow{\alpha_1}\ExtP\xrightarrow{\alpha_2}\text{Pic}(K)^{\phi}\xrightarrow{\alpha_3} H^2(G,K^*_{\phi})\\
&\xrightarrow{\alpha_4}\ClfK/_{\mbox{Pic($K$)}}\xrightarrow{\bar{\lambda}_{\phi}\circ\alpha_5}\ H^{1}(G,\Pic(K)_{\phi})
\xrightarrow{\alpha_6\circ\bar{\lambda}_{\phi}^{-1}}
H^3(G,K^*_{\phi}).
\end{array}
\end{eqnarray}
\begin{lemma}\label{R=K}
All the terms in the exact sequence \eqref{subexact} are groups.
\end{lemma}
\begin{proof}
All the terms, perhaps except $\ExtP$ and $\ClfK/_{\mbox{Pic($K$)}}$, are well-known groups.
These two terms fall into the setup of Lemma \ref{exactgrp}, and hence both of them are groups as well.
Then \eqref{subexact} is indeed a sequence of groups.
\end{proof}
The following consequence of Lemma \ref{R=K} is used in \S\ref{7.1}.
\begin{lemma}\label{extisinv}
With the notation \eqref{extP} $\ExtP=($Mod$(K_{\phi}G))^*.$
\end{lemma}
\begin{proof}
The inclusion ``$\supseteq$" follows from the definition of $\ExtP$ as the pre-image of the subgroup of invertible elements of the range Mod$(K)$ under a monoid homomorphism.
From Corollary \ref{R=K}, $\ExtP$ is a group, settling the converse inclusion $``\subseteq"$.
\end{proof}
\begin{Remark}
It can be proven that $\ClfK$ is a group. This shows directly that the monoid
$\ClfK/_{\mbox{Pic($K$)}}\cong\ClfK/_{\mbox{Pic$(K)([K_{\phi}G])$}}$ is a group.
\end{Remark}
%\begin{proof}The inclusion ``$\supseteq$" follows from the definition of $\ExtP$From Corollary \ref{R=K}, $\ExtP$ is a group, taking care of the converse inclusion $``\subseteq"$.\end{proof}

\section{The Galois case, proof of Theorem B}\label{Sec7}
Throughout this section $k\subseteq K$ is further assumed to be a Galois extension of commutative rings with Galois group  $G\stackrel{{\phi}}{\cong}\Aut_k(K)$ (see
\S\ref{GalBack}).
Let us rewrite the classic sequence \eqref{7 term}, this time labelling the group homomorphisms in accordance with those in \cite{Kanzaki}.
\begin{eqnarray}\label{7 termlabelled}\begin{array}{rl}
1&\rightarrow H^{1}(G,K^*_{\phi})\xrightarrow{\theta_1} \Pic(k)\xrightarrow{\theta_2} \Pic(K)^{\phi} \xrightarrow{\theta_3} H^{2}(G,K^*_{\phi})\\
&\xrightarrow{\theta_4} \text{Br}(K/k)\xrightarrow{\theta_5} H^{1}(G,\Pic(K)_{\phi})\xrightarrow{\theta_6} H^{3}(G,K^*_{\phi}).\end{array}
\end{eqnarray}
We handle the two assertions of Theorem B separately.
\begin{enumerate}
\item In Theorem \ref{Galcase1} we present an isomorphism $\ExtP \xrightarrow{\cong} \Pic(k)$
%. Lemma \ref{psicommutes}
and then show that this isomorphism completes a commutative 4-term diagram with the homomorphisms
 $\alpha_1,\alpha_2$ (from \eqref{subexact}), $\theta_1$ and $\theta_2$ (From \eqref{7 termlabelled}), settling the first part of Theorem B.
\item In Theorem \ref{Galcase2} we present an isomorphism between $\ClfK/_{\mbox{Pic($K$)}}$ and the relative Brauer group $\Br(K/k)$,
%and Lemma \ref{nucommutes} ascertains that
and show that it closes a commutative 4-term diagram with the homomorphisms
$\alpha_4,\bar{\lambda}_{\phi}\circ\alpha_5$ (from \eqref{subexact}), $\theta_4$ and $\theta_5$ (From \eqref{7 termlabelled}). This takes care of the second part of Theorem B.
\end{enumerate}
\subsection{$\ExtP$ is isomorphic to $\Pic(k)$}\label{7.1}
For $[M] \in$Mod$(K_{\phi}G)$, consider $M^{\phi}:=\{m\in M\mid u_g*m=m, \forall g\in G\}$ as a $k$-module.
\begin{lemma}
Let $\phi \in$ Hom$(G, \Aut_k(K))$ be a Galois action. Then with the above notation,
\begin{eqnarray}\label{prepsi}\begin{array}{ccc}
\Mod (K_{\phi}G) &\rightarrow& \Mod(k)\\
~[M] &\mapsto& [M^{\phi}]\end{array}.\end{eqnarray}
%$\begin{array}{c}[M]\mapsto [M^{\phi}]\end{array}$
is a well-defined isomorphism of monoids.
\end{lemma}
\begin{proof}
The fact that \eqref{prepsi} is well-defined is clear.
Next, the Galois setting yields an isomorphism of $K$-modules (see \cite[IV(2.1)]{FDEI})
\begin{equation}\label{theta2}K \otimes_k M^{\phi} \cong M. \end{equation}
By \eqref{theta2} we deduce that for every $[M_1],[M_2] \in\Mod(K_{\phi}G)$
%, using  \cite[IV(2.1)]{FDEI} yields a $K$-module isomorphism as follows:
\begin{eqnarray}\label{eqarr}\begin{array}{rl}
K \otimes_k M_1^{\phi} \otimes_k M_2^{\phi} \cong& K \otimes_k M_1^{\phi} \otimes_K K \otimes_k M_2^{\phi}\\
\cong& M_1 \otimes_K M_2 \\
\cong& K \otimes_k (M_1\otimes_K M_2)^{\phi}.
\end{array}\end{eqnarray}
By the Galois assumption, $K$ is a $k$-progenerator. This property allows us to cancel out $K\otimes_k-$ on both hands of \eqref{eqarr} \cite[I, Prop. 3.3]{FDEI} so as to obtain a $k$-module isomorphism
$$(M_1\otimes_K M_2)^{\phi} \cong M_1^{\phi} \otimes_k M_2^{\phi}, $$
thereby proving that \eqref{prepsi} is a group homomorphism. To complete the proof, here is an inverse of \eqref{prepsi}
\begin{eqnarray}\label{prepsiinv}%\psi^{-1}_{\phi}:
\begin{array}{ccc}
\Mod (k) &\rightarrow& \Mod(K_{\phi}G)\\
~[P] &\mapsto& [K\otimes_kP]\end{array},\end{eqnarray}
where the the $K_{\phi}G$-module $K \otimes_k P$, whose $G$-module action is given by
$$u_g*(x \otimes_k p) :=\ \phi(g)(x)\otimes_k p, \ \ \ x \in K, p \in P,g \in G.$$
Indeed, \eqref{theta2} ascertains that \eqref{prepsiinv} is a left inverse of \eqref{prepsi}.
Conversely, by putting $M:=K \otimes_k P,\ \ [P]\in \Mod (k)$ in \eqref{theta2}, we have
$$K \otimes_k (K \otimes_k P)^{\phi} \cong K\otimes_k P, \ \  [P]\in \Mod (k).$$
By the $k$-progenerator property of $K$, this further yields $(K \otimes_k P)^{\phi} \cong P$, proving that \eqref{prepsiinv} is also a right inverse of \eqref{prepsi}, and hence the claim.
%Furthermore, if $[M] \in \ker(\psi)$, we have $[M^{\phi}]$ is trivial in $\Pic(k)$, i.e., $M^{\phi} \cong k$. Again, by \eqref{theta2}, we have $M \cong K \otimes_k M^{\phi} \cong K$,
\end{proof}
\begin{theorem}\label{Galcase1}
Let $\phi \in$ Hom$(G, \Aut_k(K))$ be a Galois action. Then \eqref{prepsi} determines an isomorphism of groups
\begin{eqnarray}\label{psi}\psi_{\phi}:\begin{array}{ccc}\ExtP &\xrightarrow{\cong}& \Pic(k)\\
~[M] &\mapsto& [M^{\phi}]\end{array}.\end{eqnarray}
Moreover, with the notations \eqref{alpha1}, \eqref{alpha2}, \eqref{7 termlabelled} and \eqref{psi} the following group homomorphisms are equal
\begin{enumerate}
\item $\theta_1=\psi_{\phi} \circ \alpha_1:H^1(G,K^*_{\phi})\to\Pic(k)$, and
\item $\theta_2=\alpha_2\circ \psi_{\phi}^{-1}:\Pic(k)\to\Pic(K)^{\phi}.$
\end{enumerate}
\end{theorem}
\begin{proof}
The map \eqref{psi} is just the isomorphism \eqref{prepsi} restricted to the corresponding groups of units of both hands (see Lemma \ref{extisinv}).

The commutation properties (1) and (2) are deduced by a direct computation, applying \eqref{theta2} and the fact that all homomorphisms that participate in these two equalities are well-defined.
\end{proof}

\subsection{$\text{Cliff}_k(\phi,K)/_{\mbox{Pic($K$)}}$ is isomorphic to $\Br(K/k)$}\label{7.2}

The following result is needed in this section.
\begin{lemma}\label{cliffGal} (see \cite[Propositions 2,3]{Kanzaki})
Let $\phi \in$ Hom$(G, \Aut_k(K))$ be a Galois action.
Then a $k$-central algebra $A$ is Azumaya with $[A]_{\Br(k)}\in\Br(K/k)$
%admitting a maximal commutative subalgebra $K$ (i.e. central separable over $k$ with $[A]_{\Br(k)}\in\Br(K/k)$)
if and only if $A$ admits a $G$-grading of type $\phi$
such that $[A]\in \ClfK$ with $$A_g :=\{ a \in   A \mid g(x)a =\ {\phi{(g)}}(x)a,\ \forall x \in K, p \in P \}, \ \ \forall g\in G.$$
\end{lemma}
As mentioned in the beginning of \S\ref{Sec7}, Theorem \ref{Galcase2} below completes the proof of Theorem B(2). It relies on a nice property %of $\text{Cliff}_k(\phi,K)$ when $\phi$ is Galois, which is
given in Theorem \ref{L:2.4}, which says that
for every $A_G,A'_G$, both satisfying the equivalent conditions of Lemma \ref{cliffGal}, there exists an idempotent $\epsilon\in K\otimes_kK$ such that the graded product $A_G\otimes_{\phi}A'_G$
is embedded inside the tensor product of $A_G$ and $A'_G$ via $\epsilon\cdot A_G\otimes_{k}A'_G\cdot \epsilon$.
Moreover, due to Corollary \ref{L:2.5}, since $\epsilon$ is \textit{full} (see \eqref{prefull}) in $A_G\otimes_{k}A'_G$,
it turns out that $A_G\otimes_{\phi}A'_G$ and $A_G\otimes_{k}A'_G$ are Brauer similar.
We begin with
\begin{lemma}\label{invprogen}
Let $\phi \in$ Hom$(G, \Aut_k(K))$ be a Galois action, and let $[P] \in \Pic(K)$. Then $P$ is (faithful) $k$-projective. In particular, End$_k(P)$ is $k$-Brauer trivial.
\end{lemma}
\begin{proof}
The Galois assumption says that $K$ is a $k$-progenerator \cite[I, Cor. 1.10 and III, Prop. 1.2]{FDEI}. By \cite[I, Prop. 1.6]{FDEI} so is every invertible $K$-module.
\end{proof}
As customary, if $P$ is any module over a commutative $k$-algebra $K$, then $K$ is embedded in End$_k(P)$ as multiplications by its elements.
Under the additional assumptions of Lemma \ref{invprogen} on $\phi$ and $P$, this copy of $K$ is its own centralizer \cite[ I, Lemma 5.1]{FDEI},
and hence is a maximal commutative subalgebra of
the $k$-Azumaya algebra $\End_k(P)$.
This setup fits into that of Lemma \ref{cliffGal}, providing the following $G$-grading of type $\phi$ on $\End_k(P)$ over $K$ %(see the proof of \cite[Proposition 3]{Kanzaki}).
\begin{eqnarray}\label{kanzgr}\begin{array}{l}
\End_k(P) = \underset{g \in G}{\bigoplus} (\End_k(P))_g,\text{  where }\\
(\End_k(P))_g :=\{ f \in   \End_k(P) \mid f(xp) =\ {\phi{(g)}}(x)f(p),\ \forall x \in K, p \in P \}.
\end{array}\end{eqnarray}

On the other hand, {the graded algebras whose Clifford classes lie in the Pic$(K)$-orbit of the trivial class $[K_{\phi}G]$ have a special form when $\phi$ is Galois (see Lemma \ref{L:2.1}).
Given $[P]\in\Pic(K)$, consider the map
%the map determined by the rule
\begin{eqnarray}\label{mu}
\mu' :\begin{array}{rcl}
P(K_{\phi}G)& \rightarrow& \End_k(P)\\
p\otimes_K \sum_gx_gu_g \otimes_K f&\mapsto&(p'\mapsto \sum_gx_g\cdot {{\phi(g)}(f(p'))}\star p)
\end{array}
\end{eqnarray}
for every $p,p' \in P, x_g \in K,g\in G$ and $f \in\check{P}$=Hom$_K(P,K)$. We claim
\begin{lemma}\label{mumorphism}
The map \eqref{mu} is a $k$-algebra homomorphism.
\end{lemma}
\begin{proof}
For $a_g,a_h \in K, p,p_1,p_2 \in P,$ and $f_1,f_2 \in\check{P}$, apply the definition \eqref{mu}
$$ \begin{array}{ll}
  \mu'(p_1\otimes_K a_gu_g \otimes_K f_1)\circ\mu'(p_2\otimes_K a_hu_h \otimes f_2)(p)& = \mu'(p_1\otimes_K a_gu_g \otimes_K f_1)(a_h\cdot{\phi(h)}(f_2(p))\star p_2)
  \\ &=a_g\cdot{\phi(g)}(f_1(a_h\cdot{\phi(h)}(f_2(p))\star p_2))\star p_1\cdots\end{array}$$
%  on applying the definition of the map $\mu'$,
Since scalars from $K$ can be pulled out of $f_1 \in$ Hom$_K(P,K)$, we obtain
  \begin{align*}\cdots
  =& a_g\cdot{\phi(g)}(a_h)\cdot{\phi(g)}(f_1(p_2))\cdot{\phi(gh)}(f_2(p))\star p_1\\=& \mu'(p_1 \otimes_K a_g\cdot{\phi(g)}(a_h)\cdot{\phi(g)}(f_1(p_2))u_{gh} \otimes_K f_2)(p)\\
  =&  \mu'((p_1\otimes_K a_gu_g \otimes_K f_1)\cdot(p_2\otimes_K a_hu_h \otimes_K f_2))(p),
  \end{align*}
where the last equality follows from the multiplication rule in $P(K_{\phi}G)$ \eqref{mult}. Thus, $\mu'$ respects composition in $\End_k(P)$,
evaluated on every element $p\in P$, and is hence multiplicative.
\end{proof}}

With the notation \eqref{Delta_P}, let
\begin{equation}\label{mjf'js}
\sum\limits_{j}(p'_j \otimes_R f_j')= \Delta_P^{-1}(1) \in P \otimes_K \check{P},
\end{equation}
that is
\begin{equation}\label{mjf'j}
\sum\limits_j p'_j\cdot f'_j(p) = p,\ \forall p \in P.
\end{equation}
Then using \eqref{mjf'j}, for any $f \in\check{P}$ and $p \in P$, we have
\begin{equation*}
\sum\limits_j f(p'_j)\cdot f_j'(p) =  f(\sum\limits_jp'_j\cdot f_j'(p)) = f(p),
\end{equation*}
and thus
\begin{equation}\label{fm_jf_j'}
\sum\limits_j f(p'_j)\star f_j' = f.
\end{equation}
\begin{lemma}\label{L:2.1}
Let $\phi \in$ Hom$(G, \Aut_k(K))$ be a Galois action, and let $[P] \in \Pic(K)$. Then with the $G$-grading \eqref{kanzgr}, the homomorphism
\eqref{mu} yields a graded $k$-isomorphism $P(K_{\phi}G) \stackrel{G}{\cong} $ End$_k(P)$.
% in particular, $[K_{\phi}G]^{[\check{P}]} = $ [End$_k(P)]$ (see \eqref{actrule}).
\end{lemma}
\begin{proof}
It is not hard to verify that \eqref{mu} respects the $G$-grading.
We show that the $k$-algebra homomorphism \eqref{mu} has an inverse.
With the notation \eqref{mjf'j}, let
\begin{eqnarray}\label{mu'}
~\ \mu'' :\begin{array}{rcl}
\End_k(P)& \rightarrow &P(K_{\phi}G)\\
\varphi_g &\mapsto &\sum\limits_j \varphi_g(p_j') \otimes_K u_g \otimes_K f_j'
\end{array},\ \ \forall\varphi_g \in (\End_k(P))_g.\end{eqnarray}
%$m_j \in P, f_j'\in \mbox{Hom}_K(P,K)$ satisfies.  For $p \in P, p' \in \mbox{Hom}_K(P,K)$, and $a_g \in K$,
Then for every $p\in P$ and $f\in \check{P}$ we have
\begin{align*}
\mu''\circ\mu'(p\otimes_K u_g \otimes_K f) =&\  \sum\limits_j \mu'(p\otimes_K u_g \otimes_K f){(p'_j)} \otimes_K u_g \otimes_K f_j'=\cdots
\end{align*}
 by the definition of $\mu'$ in \eqref{mu} we have
\begin{align*}
\cdots=\ \sum\limits_j\ {\phi(g)}(f(p'_j))\star p\otimes_K u_g \otimes_K f_j'= \sum\limits_j p \otimes_K  {\phi(g)}(f(p'_j))u_g \otimes_K f_j'\\
= p \otimes_K u_g \otimes_K \sum\limits_jf(p'_j)f_j'
= p\otimes_K u_g \otimes_K f,
\end{align*}
where the last equality follows from \eqref{fm_jf_j'}. Then \eqref{mu'} is a left inverse of \eqref{mu}.
Conversely, let $\varphi_g \in (\End_k(P))_g$. Applying the definitions \eqref{mu} and \eqref{mu'} on any $p \in P$ yields
\begin{align*}
\mu'\circ\mu''(\varphi_g):p \mapsto&\ \mu'( \sum\limits_j \varphi_g(p'_j) \otimes_K u_g \otimes_K f_j')(p)= \sum\limits_j{\phi(g)}(f_j'(p))\star\varphi_g(p_j')=\cdots
\end{align*}
Since $\varphi_g \in (\End_k(P))_g$ (see \eqref{kanzgr}), we have
\begin{align*}
\cdots= \sum\limits_j \varphi_g(f_j'(p)\star p'_j)=\varphi_g( \sum\limits_j f_j'(p)\star p'_j)=\varphi_g( \sum\limits_j p'_j\cdot f_j'(p))
 = \varphi_g(p),
\end{align*}
where the last equality is \eqref{mjf'j}. Therefore, $\mu'\circ\mu''(\varphi)=\varphi$ for every $\varphi \in \End_k(P)$, and so
$\mu'$ is also a right inverse of \eqref{mu}, completing the proof.
\end{proof}
From Lemma \ref{L:2.1} we infer
\begin{corollary}\label{L:2.11}
Let $[A_G]\in \ClfK$, where $\phi \in$ Hom$(G, \Aut_k(K))$ is a Galois action.
Then $[A_G]\in\Pic(K)([K_{\phi}G])$ if and only if $A_G\stackrel{G}{\cong}\End_k(P)$ for some $[P]\in\Pic(K)$.
\end{corollary}

%For a Galois action $\phi \in$ Hom$(G, \Aut_k(K))$, let $\{\epsilon_g\}_{g\in G}\subset K \otimes_k K$ be the unique idempotents satisfying \eqref{theorth} as in Lemma \ref{L:2.2}.
%Let \eqref{AA'} be two strongly $G$-graded $k$-algebras of type $\phi$ and $g\in G$, let $a^g_i \in A_{g}, b^g_i \in A_{g^{-1}},a'^g_i \in A'_{g},$ and $b'^g_i \in A'_{g^{-1}}$ such that
%(see \eqref{})
%$$\sum\limits_i a^g_ib^g_i = 1\in A_e,\ \ \sum\limits_i a'^g_ib'^g_i = 1\in A'_e.$$
%The idempotents $\{\epsilon_g\}_{g\in G}$ lie in the tensor product $A_G \otimes_k A'_G$, where the following claim takes place.
Let $[A_G],[A'_G]\in$Cliff$_k{(\phi,K)}$ of Galois type $\phi \in$ Hom$(G, \Aut_k(K))$.
%Let \eqref{AA'} be two strongly $G$-graded $k$-algebras of Galois type $\phi \in$ Hom$(G, \Aut_k(K))$.
Then the idempotents $\{\epsilon_g\}_{g\in G}\subset K \otimes_k K$ satisfying \eqref{theorth} as in Lemma \ref{L:2.2}
lie in the tensor product $A_G \otimes_k A'_G$ and admit few nice properties as follows.
%For a Galois action $\phi \in$ Hom$(G, \Aut_k(K))$, let $\{\epsilon_g\}_{g\in G}\subset K \otimes_k K$ be Given two strongly $G$-graded $k$-algebras \eqref{AA'} of type $\phi$, then f
For every $g\in G$ there are
$$a^g_i \in A_{g^{-1}},\ \ b^g_i \in A_{g},\ \ a'^g_{i'} \in A'_{g},\ \ b'^g_{i'} \in A'_{g^{-1}}$$
assuring that $A_{g^{-1}} A_{g}=A_e(=K)$ as well as $ A'_{g} A'_{g^{-1}}=A'_e(=K)$ (see \eqref{AgAh}), that is
$$\sum\limits_i a^g_ib^g_i = 1\in A_G,\ \ \sum\limits_{i'} a'^g_{i'}b'^g_{i'} = 1\in A'_G.$$
%The idempotents $\{\epsilon_g\}_{g\in G}$ lie in the tensor product $A_G \otimes_k A'_G$, where the following claim takes place.
We have
\begin{lemma}\label{L:2.9}
Let $\phi \in$ Hom$(G, \Aut_k(K))$ be a Galois action. Then with the above notation we have for every $g,h\in G$
\begin{equation}\label{Miyash}
\sum\limits_i(a^g_i \otimes_k 1)\epsilon_h(b^g_i \otimes_k 1)=\epsilon_{hg},\ \ \sum\limits_{i'}(1\otimes_k a'^g_{i'})\epsilon_h(1\otimes_k b'^g_{i'})=\epsilon_{gh},
\end{equation}
and hence for every $g,h\in G,\ c_g\in A_{g^{}}$ and $c'_g\in A'_{g}$
\begin{equation}\label{pass}
(c_g \otimes_k 1)\epsilon_{hg}=\epsilon_{h}(c_g \otimes_k 1),\ \ \epsilon_{gh}(1\otimes_k c'_g)=(1\otimes_k c'_g)\epsilon_{h}.
\end{equation}
In particular, for every $h\in G$, $\epsilon_h$ is full, that is
\begin{equation}\label{prefull}
(A_G \otimes_k A'_G)\epsilon_h (A_G \otimes_k A'_G)=A_G \otimes_k A'_G.
\end{equation}
\end{lemma}
\begin{proof}
We prove the left equalities in \eqref{Miyash} and in \eqref{pass}. The right ones are proven in a similar way.
%By Lemma \ref{L:2.2} it is enough to
Let us first show that $\sum\limits_i(a^g_i \otimes_k 1)\epsilon_h(b^g_i \otimes_k 1)$
is an idempotent which satisfies an appropriate equation of the form \eqref{theorth}. We compute
$$\sum\limits_i(a^g_i \otimes_k 1)\epsilon_h(b^g_i \otimes_k 1)\sum\limits_j(a^g_j \otimes_k 1)\epsilon_h(b^g_j \otimes_k 1)=
\sum\limits_{i,j}(a^g_i \otimes_k 1)\epsilon_h(b^g_ia^g_j \otimes_k 1)\epsilon_h(b^g_j \otimes_k 1)=\cdots$$
since $b^g_ia^g_j \in A_e= K$ commutes with $\epsilon_h$ we obtain
$$\cdots=\sum\limits_{i,j}(a^g_i \otimes_k 1)(b^g_i\otimes_k 1)(a^g_j \otimes_k 1)\epsilon_h^2(b^g_j \otimes_k 1)=\sum\limits_j(a^g_j \otimes_k 1)\epsilon_h(b^g_j \otimes_k 1).$$
% \begin{align*}   \sum\limits_{i,j}(1\otimes a^g_i)(1\otimes b^g_i)(1\otimes a^g_j)e(1\otimes b^g_j) = \sum\limits_j(1\otimes a^g_j)e(1\otimes b^g_j). \end{align*}
Therefore, $\sum\limits_i(a^g_i \otimes_k 1)\epsilon_h(b^g_i \otimes_k 1)\in K \otimes_k K$ is an idempotent for every $g,h\in G$.
We next check condition \eqref{theorth}. For $x \in K$ we have
$$\sum\limits_i(a^g_i \otimes_k 1)\epsilon_h(b^g_i \otimes_k 1)(x\otimes_k 1) = \sum\limits_i(a^g_i \otimes_k 1)\epsilon_h(b^g_ix\otimes_k 1)=\cdots$$
since $b^g_i \in A_{g}$ we have $b^g_ix = \phi(g)(x)b^g_i$ (see \eqref{commrule}), and hence
$$\cdots=\sum\limits_i(a^g_i \otimes_k 1)\epsilon_h(\phi(g)(x)\otimes_k 1)(b^g_i\otimes_k 1)=\cdots$$
%$$\sum\limits_i(1\otimes a^g_i)e(1\otimes b^g_i)(1\otimes\ ^gx) = \sum\limits_i(1\otimes a^g_i)e(1\otimes x)(1\otimes b^g_i)=\cdots$$
using the rule \eqref{theorth} properly we obtain
$$\begin{array}{rl}
\cdots &=\sum\limits_i(a^g_i \otimes_k 1)\epsilon_h(1\otimes_k\phi(h)(\phi(g)(x)))(b^g_i\otimes_k 1)\\
 &=\sum\limits_i(a^g_i \otimes_k 1)\epsilon_h(1\otimes_k\phi(hg)(x))(b^g_i\otimes_k 1)\\ &=\sum\limits_i(a^g_i \otimes_k 1)\epsilon_h(b^g_i\otimes_k 1)(1\otimes_k\phi(hg)(x)).
\end{array}$$
%\sum\limits_i(1\otimes a^g_i)e(x\otimes 1)(1\otimes b^g_i)= \sum\limits_i(1\otimes a^g_i)e(1\otimes b^g_i)(x\otimes 1).$$
Thus, for every $g,h\in G$ and $x\in K$, the idempotent $\epsilon'({g,h}):=\sum\limits_i(a^g_i \otimes_k 1)\epsilon_h(b^g_i\otimes_k 1)\in K\otimes_kK$ satisfies
$$\epsilon'(g,h)(x\otimes_k 1)=\epsilon'(g,h)(1\otimes_k\phi(hg)(x)).$$
The third claim in Lemma \ref{L:2.2} implies that for every $g,h \in G$ there exists an idempotent $a(g,h)\in K$ such that
\begin{equation}\label{epsil'}
\epsilon'(g,h)=\epsilon_{hg}(a(g,h)\otimes_k1)=\epsilon_{hg}(1\otimes_k\phi(hg)(a(g,h))).
\end{equation}
To prove (the left equality in) \eqref{Miyash}, it is enough to show that $a(g,h)\stackrel{?}{=}1$ for all $g,h\in G$.
Indeed, multiply \eqref{epsil'} by $a^{g^{-1}}_j \otimes_k 1$ from the left and by $b^{g^{-1}}_j \otimes_k 1$ from the right,
and sum up over $j$ having
$$\begin{array}{rl}
\sum\limits_j(a^{g^{-1}}_j \otimes_k 1)\epsilon'(g,h)(b^{g^{-1}}_j \otimes_k 1)&=\sum\limits_j(a^{g^{-1}}_j \otimes_k 1)\epsilon_{hg}(1\otimes_k\phi(hg)(a(g,h)))(b^{g^{-1}}_j \otimes_k 1)\\
&=\sum\limits_j(a^{g^{-1}}_j \otimes_k 1)\epsilon_{hg}(b^{g^{-1}}_j \otimes_k 1)(1\otimes_k\phi(hg)(a(g,h)))=\\
&=\epsilon'(g^{-1},hg)(1\otimes_k\phi(hg)(a(g,h)))=\cdots
\end{array}$$
Applying \eqref{epsil'} for $\epsilon'(g^{-1},hg)
%= \sum\limits_j(a^{g^{-1}}_j \otimes_k 1)\epsilon_{hg}(b^{g^{-1}}_j \otimes_k 1)
$ yields an idempotent $a(g^{-1},hg)\in K$ with
\begin{eqnarray}\label{onehand}\begin{array}{rl}
\cdots&=\epsilon_h(a(g^{-1},hg)\otimes_k1)(1\otimes_k\phi(hg)(a(g,h)))\\
&=\epsilon_h(a(g^{-1},hg)\phi(g)(a(g,h))\otimes_k1).
\end{array}\end{eqnarray}
On the other hand,
$$\begin{array}{rl}
\sum\limits_j(a^{g^{-1}}_j \otimes_k 1)\epsilon'(g,h)(b^{g^{-1}}_j \otimes_k 1)&=\sum\limits_j(a^{g^{-1}}_j \otimes_k 1)
\sum\limits_i(a^g_i \otimes_k 1)\epsilon_h(b^g_i\otimes_k 1)(b^{g^{-1}}_j \otimes_k 1)\\
&=\sum\limits_{i,j}(a^{g^{-1}}_j a^g_i \otimes_k 1)\epsilon_h(b^g_ib^{g^{-1}}_j \otimes_k 1)=\cdots
\end{array}$$
Since $a^{g^{-1}}_j a^g_i,b^g_ib^{g^{-1}}_j\in K$ for every $i,j$, they all commute with $\epsilon_h$. We conclude
$$\cdots=\sum\limits_{i,j}(a^{g^{-1}}_j a^g_i \otimes_k 1)(b^g_ib^{g^{-1}}_j \otimes_k 1)\epsilon_h
=\sum\limits_{j}(a^{g^{-1}}_j(\sum\limits_{i} a^g_i b^g_i)b^{g^{-1}}_j \otimes_k 1)\ChC{)}\epsilon_h=\epsilon_h.$$
Comparing with \eqref{onehand} we have $\epsilon_h(a(g^{-1},hg)\phi(g)(a(g,h))\otimes_k1)=\epsilon_h.$
Then Lemma \ref{Kidem} says that the idempotents $a(g,h)\in K$ are invertible and thus equal to 1 for every $g,h\in G$,
%$\sum\limits_i(a^g_i \otimes_k 1)\epsilon_h(b^g_i\otimes_k 1)=\epsilon_{hg}$
verifying (the left equality in) \eqref{Miyash}.
At this point we promptly make use of \eqref{Miyash} to get
$$
(c_{g}\otimes_{k}1)\epsilon_{hg}=(c_{g}\otimes_{k}1)\sum\limits_i(a^{g}_i \otimes_k 1)\epsilon_{h}(b^{g}_i \otimes_k 1)
=\sum\limits_{i}(c_{g}a^{g}_i \otimes_k 1)\epsilon_{h}(b^{g}_i \otimes_k 1)=\cdots
$$
Assuming that $c_g\in A_g$, we have $c_{g}a^{g}_i\in K$ and so
$$\cdots=\sum\limits_{i}\epsilon_{h}(c_{g}a^{g}_i \otimes_k 1)(b^{g}_i \otimes_k 1)
=\sum\limits_{i}\epsilon_{h}(c_{g}\otimes_k 1)(a^{g}_i \otimes_k 1)(b^{g}_i \otimes_k 1)=\epsilon_{h}(c_{g}\otimes_k 1),$$
verifying (the left equality in) \eqref{pass}.
Finally, sum up either one of the equalities in \eqref{Miyash} over all $g\in G$ to establish that
$$1 \otimes_k 1=\sum\limits_{g\in G}\epsilon_{hg}=\sum\limits_{g\in G}\epsilon_{gh}\in (A_G \otimes_k A'_G)\epsilon_h (A_G \otimes_k A'_G),$$
and by that to deduce \eqref{prefull}.
\end{proof}
\begin{corollary}\label{L:2.5}
Let $\phi \in$ Hom$(G, \Aut_k(K))$ be a Galois action. Then with the notations in Lemma \ref{L:2.9},
for every $h\in G$ there is an equality of Brauer classes $[\epsilon_h(A_G\otimes_{k}A'_G)\epsilon_h]_{\Br(k)}=[A_G\otimes_{k}A'_G]_{\Br(k)}$.
 \end{corollary}
 \begin{proof}
This results follows from the fact \eqref{prefull} that $\epsilon_h$ is full in $A_G \otimes_k A'_G$ (see \cite[Corollary 1.12]{UF}).
%Equation \eqref{prefull} says that the idempotent $\epsilon_h$ is full in $A_G \otimes_k A'_G$, and so by \cite[Corollary 1.12]{UF}
%$\epsilon_h(A_G\otimes_{k}A'_G)\epsilon_h$ is Brauer similar to $A_G\otimes_{k}A_G$.
 \end{proof}
Let $[A_G],[A'_G]\in \ClfK$ of Galois type $\phi \in$ Hom$(G, \Aut_k(K))$
%Let \eqref{AA'} be two strongly $G$-graded $k$-algebras of Galois type $\phi \in$ Hom$(G, \Aut_k(K))$,
and let $$\epsilon:=\epsilon_e\in K \otimes_k K\subset A_G\otimes_{k}A'_G$$ be the idempotent given in Lemma \ref{L:2.2}.
%$$\epsilon(x \otimes_k 1) = \epsilon(1 \otimes_k x),\ \ \forall x \in K$$ in
We compute the product $\epsilon_{}(c_{g_1}\otimes_{k}c'_{g_2})\epsilon_{}\in\epsilon_{}(A_{g_1\otimes_{k}}A'_{g_2})\epsilon_{}$ for any $g_1,g_2\in G$.
Putting $g:=g_1$ and $g:=g_2$ in the left and the right equalities of \eqref{pass} respectively, as well as $h:=e$ in both equalities therein, we get
\begin{equation}\label{ee}
\epsilon_{}(c_{g_1}\otimes_{k}c'_{g_2})\epsilon_{}=\epsilon_{}(c_{g_1}\otimes_{k}1)(1\otimes_{k}c'_{g_2})\epsilon_{}=(c_{g_1}\otimes_{k}1)\epsilon_{g_1}\epsilon_{g_2}(1\otimes_{k}c'_{g_2}).
\end{equation}
We conclude
\begin{lemma}\label{L:2.3}
Let $[A_G],[A'_G]\in \ClfK$
%\eqref{AA'} be two strongly $G$-graded $k$-algebras
of Galois type $\phi \in$ Hom$(G, \Aut_k(K))$. Then with the above notation
\begin{equation}\label{eq:A2}
\epsilon_{}(c_{g_1}\otimes_{k}c'_{g_2})\epsilon_{}=\{
\begin{array}{cc}
\epsilon_{}(c_{g_1}\otimes_{k}c'_{g_2})=(c_{g_1}\otimes_{k}c'_{g_2})\epsilon_{}&\text{ if }g_1= g_2\\
0&\text{ if }g_1\neq g_2
\end{array}.
\end{equation}
\end{lemma}
\begin{proof}
The result follows from \eqref{pass} and \eqref{ee} by the orthogonality of the idempotent set $\{\epsilon_g\}_{g\in G}$.
\end{proof}

\begin{theorem}\label{L:2.4}With the notation in Lemma \ref{L:2.3},
 the map \begin{eqnarray}\label{eAe}
\theta:\begin{array}{ccc} A_G \otimes_{\phi} A'_G& \rightarrow& \epsilon(A_G \otimes_{k} A'_G)\epsilon\\
a_g \otimes_{\phi(g)} a'_g &\mapsto &\epsilon(a_g \otimes_k a'_g)\epsilon
\end{array} \end{eqnarray}
is an isomorphism of $k$-algebras.
 \end{theorem}
 \begin{proof} We first show that the map \eqref{eAe} is well-defined, then utilize \eqref{eq:A2} to prove that this map is a homomorphism and finally give an inverse map to deduce that \eqref{eAe}
 is in-fact an isomorphism. To prove that \eqref{eAe} is well-defined, we need to verify that the relation \eqref{scalpass} is preserved by this map.
For $a_g \in A_g, a'_g \in A'_g$ and $x \in K$ we have
%applying the definition of $\theta$ and using \eqref{eq:A2}
%\begin{align*}
$$\theta(a_g\cdot x \otimes_{\phi(g)} a'_g) = \epsilon(a_g\cdot x \otimes_{k} a'_g)\epsilon= \epsilon(a_g\otimes_k 1)(x \otimes_{k} 1)(1 \otimes_{k} a'_g)\epsilon=\cdots$$
By \eqref{pass} and then by \eqref{theorth}
$$\cdots= ( a_g\otimes_{k} 1)\epsilon_g(x\otimes_{k} 1)(1 \otimes_{k} a'_g)\epsilon=( a_g\otimes_{k} 1)\epsilon_g(1\otimes_{k} \phi(g)(x))(1 \otimes_{k} a'_g)\epsilon=\cdots,$$
and back again by \eqref{pass}
\begin{align*}
\cdots=\epsilon( a_g\otimes_{k} 1)(1\otimes_{k} \phi(g)(x) a'_g)\epsilon = \epsilon(a_g\otimes_{k} \phi(g)(x) a'_g)\epsilon= \theta (a_g \otimes_{\phi(g)}\phi(g)(x) a'_g),
\end{align*}
%again using \eqref{eq:A2} we have $$\cdots = \epsilon(a_g\otimes_{k} \phi(g)(x) b_g)\epsilon $$
verifying that \eqref{eAe} respects the skew product.

Next, consider $a_g \in A_g, a_h \in A_h, a'_g \in A'_g$ and $a'_h \in A'_h$. Then
$$\theta(a_g \otimes_{\phi(g)} a'_g)\theta(a_h\otimes_{\phi(h)} a'_h)= \epsilon(a_g\otimes_{k} a'_g)\epsilon^2(a_h\otimes_{k} a'_h)\epsilon=\cdots$$
By \eqref{eq:A2}
%\begin{align*}
$$\cdots= \epsilon(a_g\otimes_{k} a'_g)(a_h \otimes_{k} a'_h)\epsilon= \theta((a_g\otimes_{\phi(g)} a'_g)(a_h \otimes_{\phi(h)} a'_h)),$$
%\end{align*}
and hence \eqref{eAe} is multiplicative.
To end the proof, observe that due to \eqref{eq:A2}, $\epsilon(A_g \otimes_{k} A'_h)\epsilon=0$ for $g\neq h$. Therefore, it is not hard to check that the map
$$\begin{array}{ccc}
\epsilon(A_G \otimes_{k} A'_G)\epsilon &\rightarrow &A_G \otimes_{\phi} A'_G\\
\epsilon(a_g \otimes_{k} a'_g)\epsilon &\mapsto & a_g \otimes_{\phi(g)} a'_g
\end{array},\ \ a_g \in A_g, a'_g \in A'_g$$
serves as an inverse of \eqref{eAe}.
\end{proof}

Under the Galois assumption on the action $\phi \in$ Hom$(G, \Aut_k(K))$, if $[A_G]\in\ClfK$, then by Lemma \ref{cliffGal}
the $G$-graded $k$-algebra $A_G$ is $k$-Azumaya with a maximal commutative subalgebra $K$.
This determines a map
\begin{eqnarray}\label{prenu}
\rho'_{\phi} :  \begin{array}{ccc}
\ClfK &\to& \Br(K/k)\\
~[A_G]& \mapsto& [A_G]_{\Br(k)}
\end{array}.\end{eqnarray}
\begin{theorem}\label{Galcase2}
Let $\phi \in$ Hom$(G, \Aut_k(K))$ be a Galois action. Then the map \eqref{prenu} is a surjective homomorphism, whose kernel is the orbit Pic$(K)([K_{\phi}G])$
of $[K_{\phi}G]$ under the Pic$(K)$-action. Consequently, there is a well-defined group isomorphism
%$\mbox{Cliff}_K/_{\mbox{Pic}_K(\mbox{Cliff}_K)} \cong Br(K/k).$
\begin{eqnarray}\label{nu}
\rho_{\phi} : \begin{array}{ccc}
\ClfK/_{\mbox{Pic($K$)}} &\to& \Br(K/k)\\
\Pic(K)([A_G])& \mapsto& [A_G]_{\Br(k)}
\end{array}.\end{eqnarray}
Moreover, with the notations \eqref{subexact}, \eqref{7 termlabelled} and \eqref{nu}, the following group homomorphisms are equal
\begin{enumerate}
\item $\theta_4=\rho_{\phi} \circ {\alpha_4}:H^2(G,K^*_{\phi})\to\Br(K/k)$, and
\item $\theta_5={\bar{\lambda}_{\phi}\circ\alpha_5}\circ \rho_{\phi}^{-1}:\Br(K/k)\to H^{1}(G,\Pic(K)_{\phi}).$
\end{enumerate}
\end{theorem}
\begin{proof}
Let $[A_G],[A'_G]\in\text{Cliff}_k(\phi,K)$.
%\ChA{With the notation of Lemma \ref{L:2.3}} and
Implementing Theorem \ref{L:2.4} we have
$$\rho'_{\phi}([A_G]\cdot[A_G])=\rho'_{\phi}([A_G\otimes_{\phi}A'_G])=[A_G\otimes_{\phi}A'_G]_{\Br(k)}=[\epsilon(A_G\otimes_{k}A'_G)\epsilon]_{\Br(k)}=\ldots$$
By Corollary \ref{L:2.5} we have
$$\ldots=[A_G\otimes_{k}A'_G]_{\Br(k)}=[A_G]_{\Br(k)}\otimes_{k}[A'_G]_{\Br(k)}=\rho'_{\phi}([A_G])\otimes_{k}\rho'_{\phi}([A'_G]),$$
proving that \eqref{prenu} is a group homomorphism. Surjectivity of this homomorphism follows from \cite[Prop. 3]{Kanzaki}.
Let us compute the kernel of $\rho'_{\phi}$.
By \cite[Prop. 5.3]{AMGO}, $[A_G]\in\ker(\rho'_{\phi})$ if and only if $A_G \stackrel{G}{\cong}$ End$_k(P)$  for some faithful progenerator $k$-module $P$.
If the latter holds, then by \cite[Prop. 5(1)]{Kanzaki}, $[P] \in \Pic(K)$, and so by Corollary \ref{L:2.11}, $[A_G]\in\Pic(K)([K_{\phi}G])$. Conversely, suppose that $[A_G]\in\Pic(K)([K_{\phi}G])$.
Then by Corollary \ref{L:2.11}, $A_G\cong\End_k(P)$ for some $[P]\in\Pic(K)$. By Lemma \ref{invprogen}
%Since the invertibility of $P$ implies that it is also a faithful $k$-progenerator \cite[I, Prop. 1.6]{FDEI}, proving that
$[A_G]_{\Br(k)}$ is trivial, that is $[A_G]\in\ker(\rho'_{\phi})$.
The transition to \eqref{nu} is completed by the identification \eqref{Cliffmod}
between the classes of $[A_G]$ in $\ClfK/_{\mbox{Pic$(K)$}}$ and in $\ClfK/_{\mbox{Pic$(K)([K_{\phi}G])$}}$.

The commutation properties (1) and (2) are established by a direct computation, making use of \cite[Theorem 5.7]{AMGO},
and bearing in mind that all homomorphisms that participate in these two equalities are well-defined.
\end{proof}

{}

\begin{thebibliography}{}
\bibitem{AMGO} M. Auslander and O. Goldman, The Brauer group of a commutative ring. Trans.
Amer. Math. Soc. 97, 367-409 (1960).
\bibitem{Bass} H. Bass, Algebraic K-theory. W. A. Benjamin, Inc., New York-Amsterdam 1968.
\bibitem{MBDR} M. Beattie and A. del R\'{i}o, Graded equivalences and Picard groups. J. Pure Appl. Algebra 141, No. 2, 131-152 (1999).
\bibitem{CCH} M. Calvo, A. M. Cegarra and B. A. Heredia, {Structure and classification of monoidal groupoids}. Semigroup Forum 87, No. 1, 35-79 (2013).
\bibitem{CG} A. M. Cegarra and A. R. Garz\'{o}n, Obstructions to Clifford system extensions of algebras. Proc. Indian Acad.
Sci. Math. Sci. 111, no. 2, 151--161 (2001).
\bibitem{CHR} S.U. Chase, D.K. Harrison and A. Rosenberg, Galois theory and Galois cohomology
of cummutative rings. Mem. Amer. Math. Soc. 52, 15--33 (1965).
\bibitem{Dade80} E. C. Dade, Group graded rings and modules. Math. Z. 174, 241--262 (1980).
\bibitem{FDEI} F. DeMeyer and E. Ingraham, Separable algebras over commutative rings. Lecture Notes in Mathematics 181. Berlin-Heidelberg-New York: Springer-Verlag 1971.
\bibitem{DPP} M. Dokuchaev, A. Paques and H. Pinedo, Partial Galois cohomology and related homomorphisms. Q. J. Math. 70, no. 2, 737--766 (2019).
\bibitem{DPPR} M. Dokuchaev, A. Paques, H. Pinedo and I. Rocha, Partial generalized crossed products and a seven-term exact sequence. J. Algebra 572, 195-230 (2021).
\bibitem{D3} M. Dokuchaev, H. Pinedo and I. Rocha, Partial generalized crossed products, Brauer groups and a comparison of seven-term exact sequences.
\verb"arXiv":2411.00494
\bibitem{D4} M. Dokuchaev and I. Rocha, Partial generalized crossed products and a seven term exact sequence.
J. Pure Appl. Algebra 228, No. 5, Article ID 107558, 62 p. (2024).
\bibitem{KG1} L. El Kaoutit and J. G\'{o}mez-Torrecillas, Invertible unital bimodules over rings with local units, and related exact sequences of groups I.
J. Algebra 323, No. 1, 224-240 (2010).
\bibitem{KG2} L. El Kaoutit and J. G\'{o}mez-Torrecillas, Invertible unital bimodules over rings with local units, and related exact sequences of groups II. J. Algebra 370, 266-296 (2012).
\bibitem{UF} U. A. First, An 8-periodic exact sequence of witt groups of azumaya algebras with involution. Manuscr. Math. 170, No. 1-2, 313-407 (2023).
\bibitem{Gilbert05} N. D. Gilbert, Actions and expansions of ordered groupoids. J. Pure Appl. Algebra 198, 175--195 (2005).
\bibitem{YG1} Y. Ginosar, Realization-obstruction exact sequences for Clifford system extensions. Israel J. of Math. 247, 955--985 (2022).
\bibitem{YG24} Y. Ginosar, Mackey's obstruction map for discrete graded algebra.  Journal of Pure and Applied Algebra 228(4), 107533 (2024).
\bibitem{JHDR} J. Haefner and A. del R\'{i}o, Actions of Picard groups on graded rings. J. Algebra 218, No. 2, 573-607 (1999).
\bibitem{Kanzaki} T. Kanzaki, On generalized crossed product and Brauer group. Osaka J. Math. 5, 175-188 (1968).
%\bibitem{matsumura} H. Matsumura, Commutative ring theory, Cambridge Studies in Advanced Mathematics, Cambridge University Press, 1987.
\bibitem{miyashita} Y. Miyashita, An exact sequence associated with a generalized crossed product. Nagoya Math. J. 49, 21-51 (1973).
\bibitem{S01} P. Schauenburg, Turning monoidal categories into strict ones. New York J. Math. 7, 257-265 (2001).
\bibitem{VO} F. Van Oystaeyen, Separable algebras. Hazewinkel, M. (ed.), Handbook of algebra. Volume 2. Amsterdam: North-Holland. 461-505 (2000).
%\bibitem{Bass} H. Bass, Algebraic K-theory. W. A. Benjamin, Inc., New York-Amsterdam 1968.
%\bibitem{CG03} A. M. Cegarra, and A. R. Garz\'{o}n, Some algebraic applications of graded categorical group theory. Theory Appl. Categ. 11 (2003), No. 10, 215--251.
%\bibitem{D70} E. C. Dade, Compounding Clifford's theory. Ann. of Math. (2) 91 (1970), 236--290.
%\bibitem{EM47} S. Eilenberg and S. MacLane, Cohomology theory in abstract groups. II. Group extensions with a non-Abelian kernel.
%Ann. of Math. (2) 48 (1947), 326--341.
%\bibitem{EM} S. Eilenberg, and S. MacLane, Cohomology and Galois theory. I. Normality of algebras and Teichm\"{u}ller's cocycle.
%Trans. Amer. Math. Soc. 64 (1948), 1--20.
%\bibitem{H} I. N. Herstein, Noncommutative rings. The Carus Mathematical Monographs, No. 15 Published by The Mathematical Association of America; John Wiley \& Sons, Inc., New York 1968.
%\bibitem{NVO04} C. N{\u{a}}st{\u{a}}sescu and Van Oystaeyen, Methods of graded rings. Lecture Notes in Mathematics, 1836. Springer-Verlag, Berlin, 2004.
\end{thebibliography}
\end{document}